\documentclass[a4paper,10pt]{amsart}
\usepackage[latin1]{inputenc}
\usepackage{amsmath}
\usepackage{amsfonts}
\usepackage{amssymb}
\usepackage[T1]{fontenc}
\usepackage{dsfont}
\usepackage{float}
\usepackage{lscape}
\usepackage[all]{xy}
\usepackage{stmaryrd}

\theoremstyle{plain}
\newtheorem{theorem}{Theorem}[section]
\newtheorem{maprop}[theorem]{Proposition}
\newtheorem{monlem}[theorem]{Lemma}
\newtheorem{corol}[theorem]{Corollary}
\newtheorem{maconj}[theorem]{Conjecture}
\newtheorem{question}[theorem]{Question}

\newtheorem*{corolUd}{Corollary \ref{corol:Ud}}
\newtheorem*{lemclustermonomial}{Lemma \ref{lem:clustermonomial}}
\newtheorem*{theoremsemicanonicalbasisA}{Theorem \ref{theorem:semicanonicalbasisA}}
\newtheorem*{theoremdifferencedelta}{Theorem \ref{theorem:differencedelta}}
\newtheorem*{propreflexionaffine}{Proposition \ref{prop:reflexionaffine}}
\newtheorem*{lemXMlambda}{Lemma \ref{lem:XMlambda}}
\newtheorem*{corolreflectionbase}{Corollary \ref{corol:reflectionbase}}
\newtheorem*{propexplicitbase}{Proposition \ref{prop:explicitbase}}

\theoremstyle{definition}
\newtheorem{defi}[theorem]{Definition}
\newtheorem{notation}[theorem]{Notation}
\newtheorem{monexmp}[theorem]{Example}
\newtheorem{rmq}[theorem]{Remark}

\def\affA#1#2{\tilde{\mathbb A}_{#1,#2}}
\def\Aaffine{\tilde{\mathbb A}}
\def\Daffine{\tilde{\mathbb D}}
\def\Eaffine{\tilde{\mathbb E}}

\def\regrad{{\rm{q.rad}}}
\def\regsoc{{\rm{q.soc}}}
\def\reglen{{\rm{q.l}}}

\def\coker{{\rm{coker}\,}}
\def\Hom{{\rm{Hom}}}
\def\Ext{{\rm{Ext}}}
\def\add{{\rm{add}\,}}
\def\modg{{\rm{mod-}}}

\def\Gr{{\rm{Gr}}}

\def\End{{\rm{End}}}
\def\dim{{\rm{dim}\,}}
\def\ddim{{\textbf{dim}\,}}
\def\codim{{\rm{codim}}}
\def\rep{{\rm{rep}}}

\def\Ob{{\rm{Ob}}}

\def\im{{\rm{Im}\,}}
\def\Cl{{\rm{Cl}}}

\def\supp{{\rm{supp}\,}}
\def\A{{\mathbb{A}}}
\def\E{{\mathbb{E}}}
\def\N{{\mathbb{N}}}
\def\Z{{\mathbb{Z}}}

\def\Q{{\mathbb{Q}}}

\def\P{{\mathbb{P}}}

\def\<{\left<}
\def\>{\right>}

\def\d{{\partial}}

\def\ens#1{\left\{ #1 \right\}}
\def\fl{{\longrightarrow}\,}

\title[Generic variables and bases]{Generic variables in acyclic cluster algebras and bases in affine cluster algebras}
\author{\textsc{G. Dupont}}
 \address{Universit\'e de Lyon \\
Universit\'e Lyon 1 \\
Institut Camille Jordan CNRS UMR 5208 \\
43, boulevard du 11 novembre 1918\\
F-69622 Villeurbanne Cedex.}
 \email{dupont@math.univ-lyon1.fr}

\begin{document}
\maketitle

\begin{abstract}
	Let $Q$ be a finite quiver without oriented cycles and $\mathcal A(Q)$ be the coefficient-free cluster algebra with initial seed $(Q,\textbf u)$. Using the Caldero-Chapoton map, we introduce and investigate a family of generic variables in $\Z[\textbf u^{\pm 1}]$ containing the cluster monomials of $\mathcal A(Q)$. The aim of these generic variables is to give an explicit new method for constructing $\Z$-bases in the cluster algebra $\mathcal A(Q)$. 

	If $Q$ is an affine quiver with minimal imaginary root $\delta$, we investigate differences between cluster characters associated to indecomposable representations of dimension vector $\delta$. We define the notion of \emph{difference property} which gives an explicit description of these differences. We prove in particular that this property holds for quivers of affine type $\Aaffine$. When $Q$ satisfies the difference property, we prove that generic variables span the cluster algebra $\mathcal A(Q)$. If $\mathcal A(Q)$ satisfies some gradability condition, we prove that generic variables are linearly independent over $\mathbb Z$ in $\mathcal A(Q)$. In particular, this implies that generic variables form a $\Z$-basis in a cluster algebra associated to an affine quiver of type $\Aaffine$.
\end{abstract}

\setcounter{tocdepth}{1}
\tableofcontents

\section*{Introduction}
Cluster algebras were introduced by Fomin and Zelevinsky in a series of papers \cite{cluster1,cluster2,cluster3,cluster4} in order to design a general framework for understanding total positivity in algebraic groups and canonical bases in quantum groups. They turned out to be related to various subjects in mathematics like combinatorics, Lie theory, representation theory, Teichm\"uller theory and many other topics. 

A cluster algebra is a commutative algebra generated by indeterminates over $\Q$ called \emph{cluster variables}. They are gathered into sets of fixed cardinality called \emph{clusters}. The initial data for constructing a (symmetric coefficient-free) cluster algebra is a \emph{seed}, that is, a pair $(B,\textbf u)$ where $B \in M_q(\Z)$ is an anti-symmetric matrix and $\textbf u=(u_1, \ldots, u_q)$ is a $q$-tuple of indeterminates over $\Q$. The cluster variables are defined inductively by a process called \emph{mutation}. The cluster algebra associated to a seed $(B,\textbf u)$ is denoted by $\mathcal A(B)$, it is a $\Z$-subalgebra of the ring $\Z[u_1^{\pm 1}, \ldots, u_q^{\pm 1}]$ of Laurent polynomials in $\textbf u$. 

Viewing anti-symmetric matrices as incidence matrices of quivers, it is possible to define the cluster algebra $\mathcal A(Q)$ from a seed $(Q,\textbf u)$ where $Q$ is a quiver. Namely, $\mathcal A(Q)$ is the cluster algebra with initial seed $(B,\textbf u)$ where $B$ is the incidence matrix of $Q$. If $Q$ contains no oriented cycles, it is called \emph{acyclic} and $\mathcal A(Q)$ is called an \emph{acyclic} cluster algebra. Initiated in \cite{MRZ}, formally defined in \cite{BMRRT} and later developed in various papers (see \cite{BMR1,BMR2,BMR3,CCS1,CCS2,CC,CK1,CK2} for example), the \emph{cluster category} of an acyclic quiver $Q$ gives a fruitful categorification of an acyclic cluster algebra $\mathcal A(Q)$. Another categorification, using preprojective algebras is developed independently by Geiss, Leclerc and Schr\"oer (see \cite{GLS,GLS2,GLS:rigid2} for example). For a wide overview concerning categorifications of acyclic cluster algebras, one can refer to \cite{Keller:categorification}.

The investigation of $\Z$-bases in cluster algebras, whether it is very fundamental in the theory, is still widely open. A first idea would be to consider particular monomials in cluster variables in order to constitute a $\Z$-basis of $\mathcal A(Q)$. A privileged choice is given by the \emph{cluster monomials}, that is, monomials in cluster variables taken in a same cluster. It is proved in \cite{CK1} (see also \cite{shermanz} for rank two) that cluster monomials constitute indeed a basis if $Q$ is a Dynkin quiver. 

In \cite{shermanz}, the authors investigated all cluster algebras associated to matrices of rank 2 of Dynkin or affine type. It turned out that if the matrix $B$ is not of Dynkin type, the cluster monomials are not enough to generate $\mathcal A(B)$ as a $\Z$-module and one has to had somehow an ``imaginary'' part to the set of cluster monomials. For these cluster algebras of rank 2, they managed to obtain a $\Z$-basis called \emph{canonical basis}, characterized by a certain positivity property. Nevertheless, at this time there are no similar results for cluster algebras of higher ranks.

Later, in \cite{CZ}, Caldero and Zelevinsky investigated another basis arising naturally from the representation theory of the quiver $Q$ when $Q$ is the Kronecker quiver. They called it the \emph{semicanonical basis} of $\mathcal A(Q)$. It has no longer the positivity property of the canonical basis but it appeared to us that the definition of Caldero-Zelevinsky's basis, using the AR-quiver approach to cluster variables, could be generalized (see also \cite{Cerulli:thesis} for results about concerning rank three cluster algebras).

In \cite{GLS:rigid2}, the authors give a very general definition of bases in cluster algebras using Lusztig's dual semicanonical basis. Their methods consist in realizing certain cluster algebras as subalgebras of the graded dual of the enveloping algebra $U(\mathfrak n)$ where $\mathfrak n$ is the maximal nilpotent subalgebra of the Lie-Kac-Moody algebra $\mathfrak g$ associated to $Q$. Specializing coefficients, they construct a basis in the acyclic cluster algebra $\mathcal A(Q)$ using the dual $\mathcal S^*$ of Lusztig's semicanonical basis. By definition, the elements of $\mathcal S^*$ are generic constructible functions parametrized by irreducible components of the nilpotent varieties of modules over the preprojective algebra.

The idea of this article is to define generic values in an acyclic cluster algebra, giving an analogue in the context of cluster categories to the generic functions of Lusztig's dual semicanonical basis. This enables to carry out the methods used in \cite{CK1} for Dynkin quivers to quivers of affine types in order to define a $\Z$-basis. The interest of this method is that it provides a completely explicit realization of the generic values in an acyclic cluster algebra.

The main tool for constructing these generic variables is the so-called \emph{Caldero-Chapoton map}. Introduced in \cite{CC}, the Caldero-Chapoton map is an explicit map from the set of objects in the cluster category $\mathcal C_Q$ to the ring of Laurent polynomials in $\textbf u$ containing the acyclic cluster algebra $\mathcal A(Q)$. Among various interesting properties, the Caldero-Chapoton map turns out to be a constructible function, admitting generic values. It turns out that cluster monomials naturally appear as generic values. Regarding the works of Geiss, Leclerc and Schr\"oer, it is reasonable to think that generic values for the Caldero-Chapoton map form a good candidate for being a $\Z$-basis in the considered cluster algebra. The article is devoted to prove that this is indeed the case when $Q$ is an affine quiver satisfying a property called \emph{difference property}. In particular, this proves that generic variables form indeed a $\Z$-basis when $Q$ is a quiver of affine type $\Aaffine$.

Because of the similarity with Luzstig's dual semicanonical basis, it seemed to be relevant to call our basis the \emph{semicanonical basis} of the cluster algebra. However, this might look a bit confusing because, even if the terminology is the same, our semicanonical basis does not coincide with Caldero-Zelevinsky's semicanonical basis explicited in \cite{CZ} for the Kronecker quiver. We will investigate in subsection \ref{subsection:basesKronecker} the differences between known bases in this case. We will prove in particular that the Caldero-Zelevinsky's basis, the canonical basis and the semicanonical basis can be obtained from each other only by a locally unipotent base change.

The article is organized as follows. Section \ref{section:background} presents the used material and the main results of this paper. In section \ref{section:affine}, we give the essential background concerning cluster categories and cluster algebras associated to affine quivers. Section \ref{section:generic} is the heart of our study. We introduce the so-called \emph{generic variables} and study some of their properties for acyclic quivers. We then investigate their properties for affine quivers and prove that they form a generating set in cluster algebras $\mathcal A(Q)$ associated to affine quivers $Q$ satisfying a certain \emph{difference property}. In section \ref{section:reflection}, we prove that the Caldero-Chapoton map is compatible with Zhu's extended BGP-reflection-functors. Considering interactions of reflection functors and generic bases, this allows to prove that under some gradability condition on $\mathcal A(Q)$, generic variables are linearly independent in $\mathcal A(Q)$. In section \ref{section:examples}, we give explicit computations of generic variables for Dynkin quivers and affine quivers of rank lesser than four. Finally in section \ref{section:conjectures}, we give some conjectures and questions.

\section{Background and main results}\label{section:background}
	In this article, $k$ denotes the field $\mathbb C$ of complex numbers.
\begin{subsection}{The cluster category}
	Let $Q$ be a quiver, we denote by $Q_0$ the set of vertices, $Q_1$ the set of arrows and for any arrow $\alpha:i \fl j$ we denote by $s(\alpha)=i$ the \emph{source}\index{source} of $\alpha$ and by $t(\alpha)=j$ the \emph{tail}\index{tail} of $\alpha$. All the considered quivers will be connected (the underlying diagram is connected) and finite ($Q_0$ and $Q_1$ are finite sets). A quiver $Q$ is called \emph{acyclic}\index{acyclic} if $Q$ does not contain any oriented cycle.
	
	Let $Q$ be an acyclic quiver, $\Phi(Q)$ its root system, $\Phi_{\geq 0}(Q)$ its positive root system and $\Pi(Q)=\ens{\alpha_i, i \in Q_0}$ the set of simple roots of $Q$. As usual, we write 
	$$\Phi_{\geq -1}(Q)=\Phi_{\geq 0}(Q) \sqcup -\Pi(Q)$$
	the set of \emph{almost positive roots}\index{almost positive roots}. We will denote by $(-,-)$ the Tits form of $Q$. We identify the root lattice with $\Z^{Q_0}$ by sending $\alpha_i$ to the $i$-th vector of the canonical basis of $\Z^{Q_0}$.
	
	Let $\rep(Q)$ denote the category of finite dimensional representations of $Q$. If $M$ is a representation of $Q$, then for any $i \in Q_0$, $M(i)$ denotes the underlying vector space at vertex $i$ and for any $\alpha: i \fl j$ in $Q_1$, $M(\alpha):M(i) \fl M(j)$ denotes the corresponding linear map. The \emph{dimension vector}\index{dimension vector} of $M$ is the element in $\N^{Q_0}$ defined by
	$$\ddim M=(\dim M(i))_{i \in Q_0}.$$
	
	Let $kQ$ be the path algebra over $Q$ and $kQ$-mod be the category of finite dimensional left-$kQ$-modules. It is known that $kQ$-mod is equivalent to the category $\rep(Q)$, we will then often abuse the terminology by identifying modules and representations. For any vertex $i \in Q_0$, we denote by $S_i$ the simple module associated to $i$, by $P_i$ its projective cover and by $I_i$ its injective hull.
	
	For any dimension vector $\textbf d$, $\rep(Q,\textbf d)$ denotes the set of representations of $Q$ with dimension vector $\textbf d$. This is an affine variety isomorphic to 
	$$\rep(Q,\textbf d)=\prod_{\alpha \in Q_1} k^{d_{s(\alpha)}} \times k^{d_{t(\alpha)}}.$$
	In particular $\rep(Q,\textbf d)$ is an irreducible variety. The algebraic group 
	$$G_{\textbf d}=\prod_{i \in Q_0}GL(k,d_i)$$
	acts on $\rep(Q,\textbf d)$ and the $G_{\textbf d}$ orbits in $\rep(Q,\textbf d)$ coincide with the isoclasses of $kQ$-modules.
	
	The Grothendieck group $K_0(kQ)$ of $kQ$-mod is the free abelian group over the isoclasses of $kQ$-modules modulo the relations $X+Y=E$ for any short exact sequence $0 \fl X \fl E \fl Y \fl 0$. The dimension vector $\ddim $ induces an isomorphism of abelian groups 
	$$\ddim: K_0(kQ) \xrightarrow{\sim} \Z^{Q_0}$$ 
	sending the isoclass of the simple module $S_i$ to the simple root $\alpha_i$ for any $i \in Q_0$.
	
	As $kQ$ is hereditary, the Euler form on $kQ$-mod is given by
	$$\<M,N\>=\dim \Hom_{kQ}(M,N)-\dim \Ext^1_{kQ}(M,N).$$
	It is well defined on the Grothendieck group and for any two $kQ$-modules $M,N$, we have
	$$\<\ddim M, \ddim N\>=(\ddim M,\ddim N).$$
	
	We denote by $\tau$ the Auslander-Reiten translation on $kQ$-mod. An indecomposable module will be called \emph{preprojective}\index{preprojective} if it is in the $\tau$-orbit of a projective module. It will be called \emph{preinjective}\index{preinjective} if it is in the $\tau$-orbit of an injective module. It will be called \emph{regular}\index{regular} if it is neither preprojective nor preinjective. An arbitrary module is called preprojective (resp. preinjective, regular) if all its indecomposable direct summands are preprojective (resp. preinjective, regular). We denote by $\mathcal P(Q)$ (resp. $\mathcal I(Q)$, $\mathcal R(Q)$) the full subcategory of preprojective (resp. preinjective, regular) modules. If there is no possible confusion, we will omit the reference to $Q$.
	
 	We denote by $D^b(kQ)$ the bounded derived category of $kQ$-mod with translation $\tau$ and shift functor $[1]$. A complex concentrated in degree zero will still be called a module. $D^b(kQ)$ is a triangulated category and the functor $G=\tau^{-1}[1]$ is an auto-equivalence of $D^b(kQ)$. The orbit category $\mathcal C_Q=D^b(kQ)/G$, introduced in \cite{BMRRT}, is called the \emph{cluster category of $Q$}\index{cluster category}. It is proved in \cite{K} that $\mathcal C_Q$ is a triangulated category and that the canonical functor $D^b(kQ) \fl \mathcal C_Q$ is a triangle functor. The image of an object $M$ in $D^b(kQ)$ under this canonical functor will still be denoted by $M$.
	
	The cluster category $\mathcal C_Q$ is a $k$-linear Krull-Schmidt category whose indecomposable objects are indecomposable modules and shifts of indecomposable projective modules. Namely, if we denote by ind-$\mathcal K$ the family of indecomposable objects in a Krull-Schmidt category $\mathcal K$, we have:
	$$\textrm{ind-} \mathcal C_Q=\textrm{ind-} kQ\textrm{-mod} \sqcup \ens{P_i[1] \ : \ i \in Q_0}$$
	
	Every object $M$ in $\mathcal C_Q$ has thus an unique decomposition $M=M_0 \oplus P_M[1]$ where $M_0$ is a module and $P_M$ is a projective module. The homological functor $H^0=\Hom_{\mathcal C_Q}(kQ,-): \mathcal C_Q \fl kQ\textrm{-mod}$ allows to recover the module part of an object $M \in \mathcal C_Q$. Thus, we will always write
	$$M=H^0(M) \oplus P_M[1]$$
	the decomposition of an object in $\mathcal C_Q$ into a direct sum of a module and the shift of a projective module.
	
	The cluster category $\mathcal C_Q$ is a 2-Calabi-Yau category, it means that for any two objects $M,N \in \mathcal C_Q$, there is a duality
	$$\Ext^1_{\mathcal C_Q}(M,N) \simeq D\Ext^1_{\mathcal C_Q}(N,M)$$
	where $D=\Hom_k(-,k)$ is the standard duality. 
	Moreover, if $M$ and $N$ are modules, there is a precise description:
	$$\Ext^1_{\mathcal C_Q}(M,N)=\Ext^1_{kQ}(M,N) \oplus D\Ext^1_{kQ}(N,M)$$
	
	In the Auslander-Reiten quiver of a cluster category, indecomposable preprojective and indecomposable injective are in the same component. This component is called \emph{transjective}\index{transjective} and is denoted by $\mathcal {PI}(Q)$.
\end{subsection}

\begin{subsection}{Affine quivers}
	An \emph{affine quiver}\index{affine quiver} is an acyclic quiver whose underlying diagram in an extended Dynkin diagram. The representation theory of such quivers is deeply studied and the reader can for example refer to \cite{DR:memoirs, ringel:1099, CB:lectures}. We recall some useful background concerning representation theory of affine quivers. In this subsection we always assume that $Q$ is an affine quiver.

	We set
	$$\Phi^{\textrm{re}}(Q)=\ens{\alpha \in \Phi(Q) \ : \ (\alpha,\alpha)=1} \textrm{ and }\Phi_{\geq 0}^{\textrm{re}}(Q)=\Phi^{\textrm{re}}(Q)\cap \Phi_{\geq 0}(Q)$$
	the sets of \emph{real roots}\index{real roots} and \emph{positive real roots}, 
	$$\Phi^{\textrm{im}}(Q)=\ens{\alpha \in \Phi(Q) \ : \ (\alpha,\alpha)=0} \textrm{ and }\Phi_{\geq 0}^{\textrm{im}}(Q)=\Phi^{\textrm{im}}(Q)\cap \Phi_{\geq 0}(Q).$$
	the sets of \emph{imaginary roots}\index{imaginary roots} and \emph{positive imaginary roots}.
	It is known that we have the following decomposition:
	$$\Phi(Q)=\Phi^{\textrm{re}}(Q) \sqcup \Phi^{\textrm{im}}(Q) \textrm{ and }\Phi_{\geq 0}(Q)=\Phi_{\geq 0}^{\textrm{re}}(Q) \sqcup \Phi_{\geq 0}^{\textrm{im}}(Q).$$
	Moreover, there exists an unique positive imaginary root $\delta$ such that 
	$$\Phi^{\textrm{im}}(Q)=\Z\delta$$
	and this root $\delta$ is called the \emph{minimal imaginary root of $Q$}\index{minimal imaginary root}. Note that $\delta$ is a \emph{sincere root}\index{sincere root}, it means that $\delta_i \neq 0$ for every $i \in Q_0$.
	
	A positive root $\alpha$ is called a \emph{Schur root}\index{Schur roots} if there exists a (necessarily indecomposable) representation $M \in \rep(Q,\alpha)$ such that $\End_{kQ}(M) \simeq k$. Such a representation is called a \emph{Schur representation}\index{Schur representation}. According to Kac's theorem, there exists an indecomposable representation in $\rep(Q,\alpha)$ if and only if $\alpha$ is a positive root. Moreover, if $\alpha$ is a real root, there exists an unique indecomposable representation of dimension vector $\alpha$. If in addition $\alpha$ is a Schur root, then $M$ is \emph{rigid}\index{rigid}, that is, $\Ext^1_{kQ}(M,M)=0$. If $\alpha$ is a positive imaginary root, then, the set of indecomposable representations of dimension vector $\alpha$ is a $\P^1(k)$-family of pairwise non-isomorphic representations. In particular $\rep(Q,\alpha)$ (and thus $\rep(Q)$) contains infinitely many non-isomorphic indecomposable objects.
	
	The existence of a minimal imaginary root provides a very useful linear form $\d: \Z^{Q_0} \fl \Z$ called \emph{defect form}\index{defect}. This form is defined by 
	$$\d_\alpha=\<\delta,\alpha\>$$ for any $\alpha \in \Z^{Q_0}$. As the Euler-form is well defined on the Grothendieck group, for any $M$ in $\rep(Q)$, we define the \emph{defect of $M$} as
	$$\d_M=\<\delta, \ddim M\>.$$
	The defect allows to characterize whether an indecomposable module is preprojective, preinjective or regular. Namely if $M$ is an indecomposable module, then $M$ is preprojective iff $\d_M<0$, $M$ is preinjective iff $\d_M>0$ and $M$ is regular iff $\d_M=0$.
	
	We denote by $c: K_0(kQ) \fl K_0(kQ)$ the \emph{Coxeter transformation}\index{Coxeter} on $K_0(kQ)$, that is, the $\Z$-linear transformation induced by the translation on the Grothendieck group. For any $\alpha, \beta \in \Z^{Q_0}$, we have $\<\alpha,\beta\>=-\<\beta,c(\alpha)\>$. In particular as $c(\delta)=\delta$, we get
	$$\<\alpha,\delta\>=-\<\delta,\alpha\>$$
	for any $\alpha \in \Z^{Q_0}$.
	
	A \emph{tube}\index{tube} $\mathcal T$ is a category isomorphic to the mesh category of the stable translation quiver $\Z\A_\infty/(\tau^p)$ where $p \geq 1$ is an integer called \emph{rank}\index{rank} of $\mathcal T$. A tube will be called \emph{homogeneous}\index{homogeneous tube} if it has rank 1 and it will be called \emph{exceptional}\index{exceptional tube} if it is not homogeneous.
	Abusing the terminology, a quiver of the form $\Z\A_\infty/(\tau^p)$ will also be called a tube. It is known that the regular components of the AR quiver of $Q$ form a $\P^1$-family of tubes. At most three of the tubes are exceptional. The \emph{tubular type}\index{tubular type} of $Q$ is the set of ranks of exceptional tubes. For any $\lambda \in \P^1$, we denote by $\mathcal T_\lambda(Q)$ the corresponding tube in the AR-quiver of $Q$. We denote by $\P^1_0(Q)$ the set of $\lambda \in \P^1$ such that $\mathcal T_\lambda$ is homogeneous. If there is no possible confusion, the reference to $Q$ will be omitted.
	
	For a Dynkin quiver $Q$, one can define an ordering on the indecomposable objects in $kQ$-mod with respect to the existence of a morphism from an object to another. For affine quivers such an ordering does not exists. Nevertheless, there is still a result of ordering with respect to the components.
	If $P \in \mathcal P(Q)$, $I \in \mathcal I(Q)$ and $R \in \mathcal R(Q)$, then
	$$\Hom_{kQ}(R,P) \simeq \Hom_{kQ}(I,R) \simeq \Hom_{kQ}(I,P)=0,$$
	and
	$$\Ext^1_{kQ}(P,R)\simeq \Ext^1_{kQ}(R,I)\simeq \Ext^1_{kQ}(P,I)=0.$$
	If $M$ and $N$ are two regular indecomposable modules in different tubes, then 
	$$\Hom_{kQ}(M,N)=0 \textrm{ and } \Ext^1_{kQ}(M,N)=0.$$
	
	Every tube $\mathcal T$ is a uniserial abelian category closed under extensions, kernels and cokernels. A regular module $M$ in a tube $\mathcal T$ is called \emph{quasi-simple}\index{quasi-simple} if it does not contain any proper submodule in $\mathcal T$. In particular, if $M$ is quasi-simple, then it contains no proper regular submodule (these modules are sometimes called \emph{regular simple} in the literature). The \emph{quasi-composition series} of an indecomposable regular module $M$ is a sequence 
	$$0 =M_0 \subset M_1 \subset \cdots \subset M_r =M$$
	such that each $M_i$ is regular and $M_i/M_{i-1}$ is a quasi-simple module. Such a quasi-composition series is unique. The \emph{quasi-length}\index{quasi-length} $\reglen(M)$ of $M$ is the integer $r$, the \emph{quasi-socle}\index{quasi-socle} $\regsoc(M)$ of $M$ is $M_1$ and the \emph{quasi-radical}\index{quasi-radical} $\regrad(M)$ of $M$ is $M_{r-1}$.
	
	Note that the $\tau$-orbit of any projective or injective module is infinite. The situation is pictured in figure \ref{figure:clustercategory}.
	
	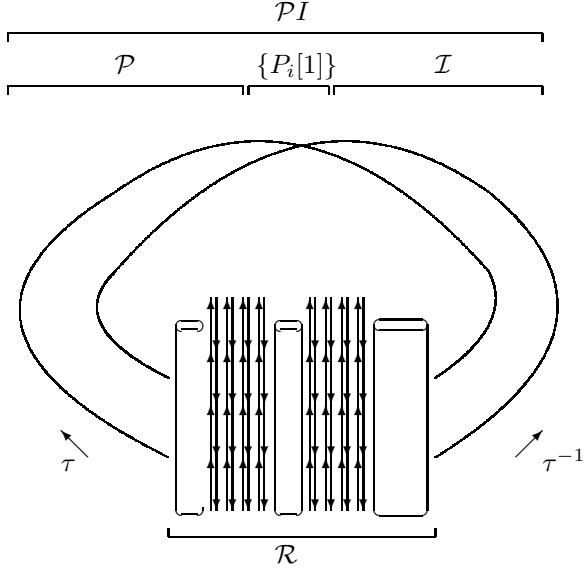
\begin{figure}[htb]
		\begin{picture}(400,300)(-50,0)
		
			\qbezier(80,150)(150,200)(220,120)
			\qbezier(220,120)(230,100)(200,80)
			\qbezier(80,150)(0,100)(100,50)
			\qbezier(80,120)(60,100)(100,80)
			\qbezier(80,120)(150,200)(220,150)
			\qbezier(220,150)(280,100)(200,50)
			
			\put(103,30){\line(0,1){70}}
			\put(113,30){\line(0,1){}}
			\put(108,30){\oval(10,3)[b]}
			\put(108,100){\oval(10,3)}
	
			\multiput(116,30)(6,0){4}{\multiput(0,0)(0,20){4}{\vector(0,1){20}}}
			\multiput(118,50)(6,0){4}{\multiput(0,0)(0,20){4}{\vector(0,-1){20}}}		
			
			\put(140,30){\line(0,1){70}}
			\put(150,30){\line(0,1){70}}
			\put(145,30){\oval(10,3)[b]}
			\put(145,100){\oval(10,3)}
	
			\multiput(153,30)(6,0){4}{\multiput(0,0)(0,20){4}{\vector(0,1){20}}}
			\multiput(155,50)(6,0){4}{\multiput(0,0)(0,20){4}{\vector(0,-1){20}}}
			
			\put(197,30){\line(0,1){70}}
			\put(177,30){\line(0,1){70}}
			\put(187,30){\oval(20,4)[b]}
			\put(187,100){\oval(20,4)}
			
			\put(40,210){\line(1,0){200}}
			\put(40,210){\line(0,-1){3}}
			\put(240,210){\line(0,-1){3}}
			\put(140,215){$\mathcal PI$}
	
			\put(40,190){\line(1,0){88}}
			\put(40,190){\line(0,-1){3}}
			\put(128,190){\line(0,-1){3}}
			\put(80,195){$\mathcal P$}
			\put(130,190){\line(1,0){30}}
			\put(130,190){\line(0,-1){3}}
			\put(160,190){\line(0,-1){3}}
			\put(133,195){$\ens{P_i[1]}$}
			\put(162,190){\line(1,0){78}}
			\put(240,190){\line(0,-1){3}}
			\put(162,190){\line(0,-1){3}}
			\put(200,195){$\mathcal I$}
			
			\put(100,20){\line(1,0){100}}
			\put(100,20){\line(0,1){3}}
			\put(200,20){\line(0,1){3}}
			\put(140,10){$\mathcal R$}
			
			\put(70,50){\vector(-1,1){10}}
			\put(60,45){$\tau$}
			
			\put(230,50){\vector(1,1){10}}
			\put(240,45){$\tau^{-1}$}
		\end{picture}
		\caption{The cluster category of an affine quiver}\label{figure:clustercategory}
	\end{figure}
\end{subsection}

\begin{subsection}{The Caldero-Chapoton map}
	We now return to the case where $Q$ is an arbitrary acyclic quiver. We denote by $\mathcal A(Q)$ the coefficient free cluster algebra with initial seed $(Q,\textbf u)$ where $\textbf u=(u_i, i \in Q_0)$ is a set of indeterminates over $\Q$. It has already been mentioned that the cluster category is a fruitful categorification for the cluster algebra $\mathcal A(Q)$. A central object for the `decategorification' in this framework is the so-called Caldero-Chapoton map introduced in \cite{CC}.
	
	Fix $M$ a $kQ$-module and $\textbf e$ a dimension vector, the \emph{quiver grassmannian of $M$ of dimension $\textbf e$}\index{quiver grassmannian} is the set
	$$\Gr_{\textbf e}(M)=\ens{N \in \rep(Q,\textbf e) \ : \ N \textrm{ is a subrepresentation of } M}.$$
	This is a closed subset of the standard vector space grassmannian and it is thus a projective variety. We can in particular define its Euler-Poincar\'e characteristic $\chi(\Gr_{\textbf e}(M))$ with respect to the \'etale cohomology with compact support. These varieties are of great interest, it is for example proved in \cite{CR} that $\chi(\Gr_{\textbf e}(M))>0$ if $\Gr_{\textbf e}(M)$ is not empty which is a very important step towards the proof of the \emph{positivity conjecture} of Fomin and Zelevinsky.
	
	\begin{defi}
		The \emph{Caldero-Chapoton map}\index{Caldero-Chapoton map} of an acyclic quiver $Q$ is the map $X_?^Q$ defined from the set of objects in $\mathcal C_Q$ to the ring of Laurent polynomials in the indeterminates $\ens{u_i, i \in Q_0}$ by:
		\begin{enumerate}
			\item If $M$ is an indecomposable $kQ$-module, then
				\begin{equation}\label{CCmap}
					X^Q_M = \sum_{\textbf v} \chi(\Gr_{\textbf v}(M)) \prod_{i \in Q_0} u_i^{-<\textbf v, \alpha_i>-\<\alpha_i, \ddim M - \textbf v\>}
				\end{equation}
			\item If $M=P_i[1]$ is the shift of the projective module associated to $i \in Q_0$, then $$X^Q_M=u_i$$
			\item For any two objects $M,N$ of $\mathcal C_Q$, 
				$$X_{M \oplus N}^Q=X_M^QX_N^Q$$
		\end{enumerate}
		
		For any object $M$ in the cluster category $\mathcal C_Q$, $X_M^Q$ will be called the \emph{generalized variable}\index{generalized variable} associated to $M$. If there is no possible confusion, we will omit the reference to $Q$.
	\end{defi}		
	Note that the Caldero-Chapoton map satisfies equality (\ref{CCmap}) for every $kQ$-module $M$ (see \cite{CC}). If there is no possible confusion, we will simply write $X_M$ for $X^Q_M$.
	
	We extend the dimension vector to the objects in the cluster category by setting for any $i \in Q_0$, $\ddim P_i[1]=-\alpha_i$ and if $M=\bigoplus_i M_i$ is a decomposition into indecomposable objects, we set 
	$$\ddim M=\sum_i \ddim M_i.$$		
	Considering a Laurent polynomial $F=P(\textbf u)/\prod_{i \in Q_0} u_i^{d_i}$ such that $P(\textbf u)$ is not divisible by any $u_i$, we define the \emph{denominator vector}\index{denominator vector} $\delta(F)$ of $F$ as the tuple $\textbf d=(d_i)_{i \in Q_0}$. The following theorem will be referred to as the \emph{denominators theorem}\index{denominators theorem}:		
	\begin{theorem}[\cite{CK1}]\label{theorem:denominators}
		Fix $Q$ an acyclic quiver. Then for any object $M$ in $\mathcal C_Q$, the denominator vector of $X^Q_M$ is $\ddim M$.
	\end{theorem}
	
	The central role of the Caldero-Chapoton map in the `decategorification' is that it allows to realize cluster variables as generalized variables associated to indecomposable rigid objects. Namely, the result is:		
	\begin{theorem}[\cite{CK2}]\label{theorem:correspondanceCK2}
		Let $Q$ be an acyclic quiver. Then 
		$$\Cl(Q)=\ens{M \in \Ob(\mathcal C_Q) \ : \ M \textrm{ is indecomposable and rigid}}.$$
		where $\Cl(Q)$ denotes the set of cluster variables in $\mathcal A(Q)$.
	\end{theorem}
	
	In \cite{CK1}, it turned out that using the properties of the Caldero-Chapoton map, one is able to prove that the set of cluster monomials is a $\Z$-basis of the $\Z$-module $\mathcal A(Q)$ when $Q$ is a Dynkin quiver. According to theorem \ref{theorem:correspondanceCK2}, an equivalent way to state this result is:		
	\begin{theorem}[\cite{CK1}]\label{theorem:baseCK1}
		Let $Q$ be a quiver of Dynkin type. Then the set 
		$$\ens{X_M \ : \ M \textrm{ is rigid in } \mathcal C_Q}$$
		is a $\Z$-basis of the $\Z$-module $\mathcal A(Q)$.
	\end{theorem}
\end{subsection}

\begin{subsection}{Main results}
	The aim of this article is to give a generalization of theorem \ref{theorem:baseCK1} when $Q$ is a quiver of affine type. It already appeared in \cite{shermanz, CZ} that if $Q$ is not of Dynkin type, then in general, the cluster monomials do not form a generating family of the $\Z$-module $\mathcal A(Q)$. This should be a special case of a general phenomenon for quivers of infinite representation type (see also \cite{Leclerc:imaginary_vectors}).
	
	For our purpose, we will define a generalization of the notion of cluster monomial. By theorem \ref{theorem:correspondanceCK2}, cluster monomials correspond to generalized variables associated to rigid objects. It is known that the $G_{\textbf d}$-orbit of a rigid object $M \in \rep(Q,\textbf d)$ is open dense in $\rep(Q,\textbf d)$. As the Caldero-Chapoton map on $\rep(Q,\textbf d)$ is $G_{\textbf d}$ invariant, cluster monomials can be viewed as values of the Caldero-Chapoton map on a dense open subset of $\rep(Q,\textbf d)$. The following result, which is a corollary of lemma \ref{lem:Ude}, will generalize the notion of cluster monomial with respect to this property.
	\begin{corolUd}
		Fix $Q$ an acyclic quiver and $\textbf d$ a dimension vector. Then there exists a dense open subset $U_{\textbf d} \subset \rep(Q,\textbf d)$ such that the Caldero-Chapoton map is constant over $U_{\textbf d}$. Moreover, if $U'_{\textbf d}$ is another such open subset, the values of the Caldero-Chapoton map are the same on $U_{\textbf d}$ and $U'_{\textbf d}$.
	\end{corolUd}
	
	If $\textbf d \in \N^{Q_0}$, we write $X_{\textbf d}$ the value of the Caldero-Chapoton map on $U_{\textbf d}$. For $\textbf d \in \Z^{Q_0}$, we write
	$$[\textbf d]_+=\left(\max(d_i,0)\right)_{i \in Q_0}$$
	and we set
	$$X_{\textbf d}=X_{[\textbf d]_+}.\prod_{d_i <0} u_i^{-d_i}$$
	the \emph{generic variable of dimension $\textbf d$}\index{generic variable}. The following lemma states that generic variables generalize cluster monomials:
	
	\begin{lemclustermonomial}
		Let $Q$ be an acyclic quiver, fix $x$ a cluster monomial with denominator vector $\delta(x)$. Then
		$$x=X_{\delta(x)}.$$
		Equivalently, if $M$ is a rigid object in $\mathcal C_Q$, we have
		$$X_M=X_{\ddim M}.$$
	\end{lemclustermonomial}
	
	We set
	$$\mathcal B'(Q)=\ens{X_{\textbf d} \ : \ \textbf d \in \Z^{Q_0}}$$
	to be the set of all generic variables. We can explicitly describe the generic variables for an affine quiver. If $\mathcal E$ is a set of objects in $\mathcal C_Q$, we denote by $\add \mathcal E$ its additive closure, that is, the full subcategory of $\mathcal C_Q$ whose objects are direct sums of directs summands of objects in $\mathcal E$. The result is the following:
	\begin{propexplicitbase}
		Let $Q$ be an affine quiver and $\textbf d \in \Z^{Q_0}$, denote by $\mathcal E$ a set of representatives of isoclasses of regular modules $M$. Then
		$$\mathcal B'(Q)=\ens{\textrm{cluster monomials}} \sqcup \ens{X_{M_\lambda^{\oplus n} \oplus E} \ : \ E \in \add \mathcal E \textrm{ is rigid}, n\geq 1}$$
		where $M_\lambda$ is any quasi-simple in an homogeneous tube.
	\end{propexplicitbase}
			
	We will prove several results for cluster algebras of affine types. The following first result is expected whereas the second one is very surprising. It gives very nice relations between generalized variables associated to indecomposable modules having the same dimension vectors for a cluster algebra of affine type $\Aaffine$. For any $\lambda \in \P^1_0$, we denote by $M_\lambda$ the unique indecomposable module of dimension $\delta$ belonging to the homogeneous tube $\mathcal T_\lambda$ and by $M_\lambda^{(n)}$ the unique indecomposable regular module of quasi-length $n$ and quasi-socle $M_\lambda$. For any quasi-simple module $E$ in an exceptional tube, we write $M_E$ the unique indecomposable module of dimension $\delta$ with quasi-socle $E$. 		
	\begin{lemXMlambda}
		Let $Q$ be a quiver of affine type. Then for any $\lambda,\mu \in \P^1_0$ and any $n \geq 1$, we have 
		$$X_{M_\lambda^{(n)}}=X_{M_\mu^{(n)}}$$
	\end{lemXMlambda}		
	\begin{theoremdifferencedelta}
		Fix $Q$ a quiver of affine type $\Aaffine$, $E$ a regular simple module in an exceptional tube, $\lambda \in \P^1_0$, then
		$$X_{M_E}=X_{M_\lambda}+X_{\regrad M_E/E}$$
	\end{theoremdifferencedelta}
	We say that an affine quiver $Q$ satisfies the \emph{difference property}\index{difference property} if it satisfies the above theorem. The main result of this article is the following theorem:
	\begin{theoremsemicanonicalbasisA}
		Let $Q$ be an affine quiver such that every quiver reflection-equivalent to $Q$ satisfies the difference property. Then $\mathcal B'(Q)$ is a $\Z$-basis for the $\Z$-module $\mathcal A(Q)$ called the \emph{semicanonical basis of $\mathcal A(Q)$}\index{semicanonical basis}.
	\end{theoremsemicanonicalbasisA}
	It proves in particular that generic variables form a $\Z$-basis in an affine cluster algebra of type $\Aaffine$.
	
	In subsection \ref{subsection:reflexions}, we will study the interaction between Zhu's extended BGP-reflection functors and the Caldero-Chapoton map. For terminology, the reader can refer to subsection \ref{subsection:reflexions}. An important result is the following:
	\begin{propreflexionaffine}
		Let $Q$ be an affine quiver with at least three vertices. Let $i$ be a sink in $Q_0$. Assume that $Q$ and $\sigma_i Q$ satisfy the difference property. Denote by $\Phi_i:\mathcal A(Q) \fl \mathcal A(\sigma_iQ)$ the canonical isomorphism and by $R_i^+:\mathcal C_Q \fl \mathcal C_{\sigma_iQ}$ the extended BGP functor. Then for any object $M$ in $\mathcal C_Q$, we have $\Phi_i(X^Q_M)=X^{\sigma_i Q}_{R_i^+M}$.
	\end{propreflexionaffine}
	We can then deduce the behaviour of the generic variables under reflection functors:
	\begin{corolreflectionbase}
		Let $Q$ be an affine quiver of affine type with at least three vertices. Let $i$ be a sink in $Q_0$. Assume that $Q$ and $Q'$ satisfy the difference property. Denote by $\Phi: \mathcal A(Q) \fl \mathcal A(\sigma_i Q)$ the canonical isomorphism. Then for any $\textbf d \in \Z^{Q_0}$, we have
		$$\Phi(X^Q_{\textbf d})=X^{\sigma_i Q}_{\sigma_i (\textbf d)}.$$
	\end{corolreflectionbase}
\end{subsection}

\section{Affine cluster algebras and cluster categories}\label{section:affine}
	\begin{subsection}{Cluster algebras of affine types}\label{subsection:affinetype}	
	Fomin and Zelevinsky proved in \cite{cluster2} that if $Q$ is a quiver (or more generally a valued graph) of a given Dynkin type, then every acyclic quiver mutation-equivalent to $Q$ is also a Dynkin quiver of the same Dynkin type. In this subsection, we generalize this result to the affine case.
	
	\begin{defi}
		A quiver $Q$ is said to be of \emph{affine type $\affA{r}{s}$} if it is isomorphic to a quiver with underlying diagram of affine type $\tilde{\mathbb A}_{r+s-1}$ with $r$ arrows going clockwise and $s$ arrows going anticlockwise for some integers $r,s>0$.
		It is said of \emph{affine type $\tilde{\mathbb D_{n}}$} (resp. $\tilde{\E_{n}}$) if the underlying diagram is of type $\tilde{\mathbb D_{n}}$ for some integer $n \geq 4$ (resp. $n=6,7,8$).
	\end{defi}
	
	\begin{rmq}
	 	Note in particular that a quiver is of affine type $\affA rs$ if and only if it is of affine type $\affA sr$.
	\end{rmq}

	The following proposition proves that there is a finite-affine-wild classification of acyclic cluster algebras. It is actually a particular case of \cite[Corollary 4]{CK2}.
	\begin{maprop}\label{prop:mutationtype}
	 	Let $Q$ be an acyclic quiver of a given affine type $\mathbb X$. Then all the acyclic quivers mutation-equivalent to $Q$ are of affine type $\mathbb X$. Moreover, if $Q'$ is a quiver of affine type $\mathbb X$, then $Q'$ is mutation equivalent to $Q$.
	\end{maprop}
	\begin{proof}
		If $Q'$ is an acyclic quiver mutation-equivalent to $Q$, then $Q'$ is the quiver of the opposite endomorphism ring of a cluster tilting object $T$ in $\mathcal C_Q$. It thus follows from \cite{KR} that the cluster categories $\mathcal C_Q$ and $\mathcal C_{Q'}$ are triangle equivalent. If we denote by $\textbf n=(n_1, \ldots, n_s)$ the tubular type of $Q$, then the AR-quiver of the cluster category of $\mathcal C_Q$ contains $s$ tubes of respective ranks $n_1, \ldots, n_s$. Assume now that $Q'$ is not of affine type. Then either $Q'$ is Dynkin but then $\mathcal C_{Q'}$ has only finitely many non-isomorphic objects and thus is not equivalent to $\mathcal C_Q$, or $Q'$ is wild and the the AR-quiver of $\mathcal C_Q$ does not contain tubes. It follows that $Q'$ is also a quiver of affine type and moreover it has the same tubular type $\textbf n$. Now, it is known that two affine quivers have the same tubular type if and only if they have the same affine type. This proves the first assertion.
		
		For the second assertion, assume first that $Q$ is a quiver of affine type $\affA rs$. If $Q'$ is another quiver of type $\affA rs$, then it is known (see \cite{ASS} for example) that $Q'$ is reflection-equivalent to $Q$ and in particular, it is mutation-equivalent to $Q$. Now assume that $Q$ and $Q'$ are of affine type $\tilde{\mathbb D}$ or $\tilde{\mathbb E}$. Then $Q$ and $Q'$ are two different orientations of a tree but is is known (see also \cite{ASS}) that any different orientations of a tree are reflection-equivalent. It follows that $Q$ and $Q'$ are mutation-equivalent.
	\end{proof}
	
	Proposition \ref{prop:mutationtype} allows to speak of cluster algebras of affine type and to define the \emph{affine type of such a cluster algebra}\index{affine cluster algebra}.	
	\begin{defi}
	 	A cluster algebra $\mathcal A$ is said to be \emph{of affine type $\mathbb X$} where $\mathbb X=\affA rs, \tilde{\mathbb D}_r$, or $\tilde{\mathbb E}_n$ for some non-negative integers $r,s$ or some $n=6,7,8$ if it contains a seed $(\textbf x, Q)$ where $Q$ is an acyclic quiver of affine type $\mathbb X$. 
	 	
	 	For simplicity, we say that $\mathcal A$ if \emph{of affine type $\tilde{\mathbb A}$} if it is of affine type $\affA rs$ for some non-negative integers $r,s$.
	\end{defi}
	
	\begin{rmq}
	 	Note that we exclude the quivers $Q$ of cyclic types $\tilde {\mathbb A}_{r,0}$ for some integer $r>0$. Indeed, it is known that in this case, the cluster algebra is of Dynkin type $\mathbb D$.
	\end{rmq}
\end{subsection}

\begin{subsection}{Multiplications in affine cluster algebras}\label{subsection:multiplications}
	
	Following the ideas of \cite{CK1} for cluster algebras of finite type, we will use cluster multiplication theorems in order to prove that $\mathcal B'(Q)$ generates the $\Z$-module $\mathcal A(Q)$. These theorems are fairly important for the study of cluster algebras because they provide a certain structure of Hall algebra for cluster categories. For different multiplication theorems, the reader can for example refer to \cite{CK1,CK2,Dupont:stabletubes,Hubery:cluster,XX,Xu}.
	
	The following multiplication formula will be referred to as the \emph{almost split multiplication formula}:
	\begin{maprop}[\cite{CC}]\label{prop:almostsplitmult}
		Let $Q$ be an acyclic quiver and $M$ be a non-projective $kQ$-module. Then 
		$$X_MX_{\tau M}=X_B +1$$
		where $B$ is the central term of the almost split sequence 
		$$0 \fl \tau M \fl B \fl M \fl 0.$$
	\end{maprop}
	
	Another important formula is the \emph{one-dimensional multiplication formula}:
	\begin{theorem}[\cite{CK2}]\label{theorem:onedimmult}
		Let $Q$ be an acyclic quiver and $M,N$ be any two objects in $\mathcal C_Q$ such that $\dim \Ext^1_{\mathcal C_Q}(M,N)=1$. Then 
		$$X_MX_N=X_B +X_{B'}$$
		where $B$ and $B'$ are the unique objects such that there exists non-split triangles
		$$M \fl B \fl N \fl M[1] \textrm{ and }N \fl B' \fl M \fl N[1].$$
	\end{theorem}
	The following theorem is a consequence of \cite{XX} or \cite{Xu} (see also \cite[Theorem 9.2]{Dupont:stabletubes}).
	
	\begin{theorem}\label{theorem:multiplication}
		Let $Q$ be an acyclic quiver and $M,N$ be two indecomposable $kQ$-modules. Then $X_MX_N$ can be written as a $\Q$-linear combination of $X_Y$ where $Y$ is either the middle term of short exact sequences of $kQ$-modules
		$$0 \fl M \fl Y \fl N \fl 0 \textrm{ or }0 \fl N \fl Y \fl M \fl 0$$
		or is isomorphic to $\ker f \oplus \coker f[-1]$ for some morphism $f$ in $\Hom_{kQ}(M,\tau N)$ or in $\Hom_{kQ}(N,\tau M)$.
	\end{theorem}
	
	\begin{corol}\label{corol:expansion}
		Let $Q$ be an acyclic quiver and $M,N$ be any two objects in $\mathcal C_Q$. Then $X_MX_N$ can be written as a $\Q$-linear combination of $X_Y$ where $Y$ is the middle term of a non split triangle in $\Ext^1_{\mathcal C_Q}(N,M)$ or in $\Ext^1_{\mathcal C_Q}(M,N)$.
	\end{corol}
	\begin{proof}
		We assume that $M,N \neq 0$. We prove it by induction on the number $n$ of indecomposable direct summands of $M \oplus N$. If $n=2$, then $M$ and $N$ are decomposable and the result is nothing but theorem \ref{theorem:multiplication}. Assume that $n>2$. Then $M\oplus N=M_1 \oplus M_2 \oplus N_1 \oplus N_2 $ with $M_1,N_1$ indecomposable. Then
		\begin{align*}
			X_MX_N 	&= X_{M \oplus N}\\
				&= X_{M_1 \oplus M_2 \oplus N_1 \oplus N_2}\\
				&= X_{M_2 \oplus N_2} \left(X_{M_1}X_{N_1}\right)\\
				&= X_{M_2 \oplus N_2} \sum_Y n_Y X_Y\\
		\end{align*}
		where the $n_Y$ are rational numbers and $Y$ runs over middle terms of non split triangles 
		$$M_1 \fl Y \fl N_1 \fl M_1[1] \textrm{ or }N_1 \fl Y \fl M_1 \fl N_1[1].$$
		But if $M_1 \xrightarrow{i} Y \xrightarrow{p} N_1 \xrightarrow{\epsilon} M_1[1]$ is a non-split triangle, there is a non-split triangle
		$$\xymatrix{
		M_1 \oplus M_2 
			\ar[r]^{\left[\begin{array}{cc}
				i & 0 \\ 0 & 1 \\ 0 & 0
			\end{array}\right]}
 		&  Y \oplus M_2 \oplus N_2 \ar[r] 
 			\ar[r]^{\left[\begin{array}{ccc}
				p & 0 & 0 \\ 0 & 0 & 1 
			\end{array}\right]}
 		&  N_1 \oplus N_2 \ar[r] 
 			\ar[r]^{\left[\begin{array}{cc}
				\epsilon & 0 \\ 0 & 0
			\end{array}\right]}
 		& M_1[1] \oplus M_2[1]
		}
		.$$
		It follows that
		$$X_MX_N=\sum_Y n_YX_{Y \oplus M_2 \oplus N_2}$$
		where $n_Y$ are rational numbers and $Y \oplus M_2 \oplus N_2$ runs over middle terms of non split triangles
		$$M \fl Y \oplus M_2 \oplus N_2 \fl N \fl M[1] \textrm{ or }N \fl Y \oplus M_2 \oplus N_2 \fl M \fl N[1].$$
	\end{proof}
	
	Now, we get interested in middle terms occurring in these non-split triangles. For any two $kQ$-modules $M,N$, we write $\Ext^1_{\mathcal C_Q}(N,M)_Z$ the set of triangles in $\Ext^1_{\mathcal C_Q}(N,M)$ with middle term isomorphic to $Z$. For any object $M$ in $\mathcal C_Q$, we write $[M,M]^1=\dim \Ext^1_{\mathcal C_Q}(M,M)$. We denote by $K_0^{\textrm{split}}(\mathcal C_Q)$ the split Grothendieck group, that is, the free abelian group over the isoclasses of $\mathcal C_Q$ modulo the split triangles. 
	
	\begin{monlem}\label{lem:dimautoext}
		Let $Q$ be an acyclic quiver, and $M,N$ be two objects $\Ext^1_{\mathcal C_Q}(N,M) \neq 0$. Fix $Z \neq M \oplus N$ an object such that $\Ext^1_{\mathcal C_Q}(N,M)_Z \neq \emptyset$. Then
		$$[Z,Z]^1 < [M \oplus N, M \oplus N]^1$$
	\end{monlem}
	\begin{proof}
		Fix a non-split triangle 
		$$M \fl Z \fl N \xrightarrow{\epsilon} M[1].$$
		
		$[-,-]^1$ induces a bilinear form on $K_0^{\textrm{split}}(\mathcal C_Q)$. 
		For any object $R$, the contravariant functor $\Hom_{\mathcal C_Q}(-,R[1])$ applied to the triangle 
		$M \fl Z \fl N \xrightarrow{\epsilon} M[1]$ gives rise to the exact sequence
		$$0 \fl K_R \fl \Hom_{\mathcal C_Q}(N,R[1]) \fl \Hom_{\mathcal C_Q}(Z,R[1]) \fl \Hom_{\mathcal C_Q}(M,R[1]) \fl C_R \fl 0$$
		where $K_R$ is the kernel of the map $\Hom_{\mathcal C_Q}(N,R[1]) \fl \Hom_{\mathcal C_Q}(Z,R[1])$ and $C_R$ is the cokernel.
		
		In particular, for any object $R$, we have
		$$\dim K_R + \dim C_R + \dim \Hom_{\mathcal C_Q}(Z,R[1])=\dim \Hom_{\mathcal C_Q}(N,R[1])+\dim \Hom_{\mathcal C_Q}(M,R[1]),$$
		thus
		$$\dim K_R + \dim C_R + [Z,R]^1=[N,R]^1+[M,R]^1.$$
		We the get
		$$[Z,N]^1 \leq [N,N]^1+[M,N]^1 \textrm{ and } [Z,Z]^1 \leq [N,Z]^1+[M,Z]^1.$$
		Moreover, $\epsilon$ is a non-zero element in the kernel $K_M$ of $\Hom_{\mathcal C_Q}(N,M[1]) \fl \Hom_{\mathcal C_Q}(Z,M[1])$, and thus $\dim K_M >0$. Finally,
		$$[Z,M]^1<[N,M]^1+[M,M]^1.$$
		
		It follows that
		\begin{align*}
			[M+N, M+N]^1
				&=[M,M]^1 + [N,N]^1 + [M,N]^1+[N,M]^1\\
				&> [Z,N]^1 + [Z,M]^1 \\
				&=[N,Z]^1 + [M,Z]^1\\
				& \geq [Z,Z]^1
		\end{align*}
		which proves the lemma.
	\end{proof}
	
	This implies the following corollary which will be essential to us:

	\begin{corol}\label{corol:inductionstep}
		Let $Q$ be an affine quiver and $M,N$ be two objects such that $\Ext^1_{\mathcal C_Q}(N,M) \neq 0$. Then, $X_MX_N$ is a (finite) $\Q$-linear combination of $X_Y$ where $Y$ runs over objects in $\mathcal C_Q$ such that $[Y,Y]^1<[M\oplus N,M\oplus N]^1$.
	\end{corol}
	\begin{proof}
		As $\tau$ is an auto-equivalence of $\mathcal C_Q$ and $\mathcal C_Q$ has infinitely many indecomposable objects in its transjective component, we can assume that $M$ and $N$ are modules. Then, every $Y$ occurring in the expansion of corollary \ref{corol:expansion} is the middle term of a non-split triangle in the cluster category. Lemma \ref{lem:dimautoext} gives the result.
	\end{proof}
	
	The following is an easy but useful lemma:
	\begin{monlem}\label{lem:dimtermecentral}
		Let $Q$ be an acyclic quiver and $M,N$ be two objects in $\mathcal C_Q$. Then for any triangle 
		$$M \fl E \fl N \fl M[1]$$
		we have $\ddim H^0(E) \leq \ddim H^0(M \oplus N)$.
	\end{monlem}
	\begin{proof}
		Applying $H^0$ to the triangle $M \xrightarrow{\alpha} E \xrightarrow{\beta} N \fl M[1]$, we get a long exact sequence 
		\begin{equation}\label{eq:longexactsequence}
			\cdots \fl H^0(M) \xrightarrow{H^0(\alpha)} H^0(E) \xrightarrow{H^0(\beta)} H^0(N) \fl \cdots
		\end{equation}
		We denote by 
		$$K_\alpha=\ker H^0(\alpha),\, C_\alpha=\coker H^0(\alpha),$$
		$$K_\beta=\ker H^0(\beta),\, C_\beta=\coker H^0(\beta).$$
		
		We have two exact sequences
		$$0 \fl K_\alpha \fl H^0(M) \xrightarrow{H^0(\alpha)} H^0(E) \fl C_\alpha \fl 0,$$
		$$0 \fl K_\beta \fl H^0(E) \xrightarrow{H^0(\beta)} H^0(N) \fl C_\beta \fl 0.$$
		
		In the Grothendieck group $K_0(kQ)$, we thus get the equalities 
		$$H^0(E)=H^0(M)+C_\alpha-K_\alpha$$
		$$H^0(E)=H^0(N)+K_\beta-C_\beta$$
		and thus 
		$$2H^0(E)=H^0(M) + H^0(N) +(C_\alpha-C_\beta)+(K_\beta-K_\alpha)$$
		
		As the sequence (\ref{eq:longexactsequence}) is exact, we have $K_\beta=\im(H^0(\alpha))$ and thus $H^0(E)-K_\beta=C_\alpha$. It follows that 
		$$H^0(E)=H^0(M) + H^0(N)-(K_\alpha+C_\beta)$$
		and thus identifying $K_0(kQ)$ and $\Z^{Q_0}$ with the dimension vector, we get 
		$$\ddim H^0(E) \leq \ddim H^0(M \oplus N).$$
	\end{proof}

	\begin{monlem}\label{lem:Centremodule}
		Let $Q$ be an affine quiver, $\mathcal T$ be a tube of the AR-quiver of $Q$ and $M,N$ be two indecomposable regular modules in $\mathcal T$. Then for any triangle 
		$$M \fl E \fl N \fl M[1]$$
		$E$ is a regular module in $\mathcal T$.
	\end{monlem}
	\begin{proof}
		Fix a triangle $M \xrightarrow{\alpha} E \xrightarrow{\beta} N \xrightarrow{\phi} M[1]$. With notations of lemma \ref{lem:dimtermecentral}, we have 
		$$H^0(E)=H^0(M) + H^0(N)-(K_\alpha+C_\beta).$$
		 
		Thus, $\d_{H^0(E)}=\d_{H^0(M)}+\d_{H^0(N)}-\d_{K_\alpha}-\d_{C_\beta}=-\d_{K_\alpha}-\d_{C_\beta}$. 
		The exact sequences
		$$0 \fl K_\alpha \fl H^0(M) \xrightarrow{H^0(\alpha)} H^0(E) \fl C_\alpha \fl 0,$$
		$$0 \fl K_\beta \fl H^0(E) \xrightarrow{H^0(\beta)} H^0(N) \fl C_\beta \fl 0.$$
		yield 
		$$\left\{\begin{array}{ll}
			\d_{H^0(E)} &=\d_{C_\alpha}-\d_{K_\alpha}\\
			\d_{H^0(E)} &=\d_{K_\beta}-\d_{C_\beta}\\
		\end{array}\right.$$
 		so that
 		$$\left\{\begin{array}{ll}
 			\d_{C_\beta} &=-\d_{C_\alpha}\\
 			\d_{K_\beta} &=-\d_{K_\alpha}.
 		\end{array}\right.$$

		Set $f=H^0(\phi): N \fl \tau M$, it is a morphism between regular modules. It follows that $\d_{\ker f}=0=\d_{\coker f}$. But from the long exact sequence \ref{eq:longexactsequence}, we get $\im(H^0(\beta))=\ker f$ so that the exact sequence 
		$$0 \fl \im(H^0(\beta)) \fl H^0(N) \simeq N \fl C_\beta \fl 0$$
		gives $\d_{C_\beta}=0$ and thus $\d_{H^0(E)}=\d_{K_\beta}$.  On the other hand, $K_\alpha=\ker H^0(\alpha) =\im(H^0(f[-1])) \simeq \im(H^0(\tau f))$ but $\tau f$ is a morphism between regular modules so $\im(\tau f)$ is a regular module and thus $\d_{K_\alpha}=0$. As $\d_{K_\alpha}=-\d_{K_\beta}=-\d_{H^0(E)}$, it follows that $\d_{H^0(E)}=0$.

		Write $E=P_E[1] \oplus P \oplus R \oplus I$ where $P_E$ is a projective module, $P$ is a preprojective module, $R$ is a regular module and $I$ is a preinjective module. Assume that $I$ is non-zero, thus by defect, $P$ is also non-zero. Denote by $p$ the rank of $\mathcal T$, there exists some $k \in p\Z$ such that $\tau^{-k}I=0$. Thus $I[-k]$ is either the shift of a projective module, or a preprojective module. Note that $P[-k]$ is a non-zero module. Consider the triangle associated to the morphism $\phi[-k]$, it is given by 
		$M \xrightarrow{\alpha[-k]} E[-k] \xrightarrow{\beta[-k]} N \xrightarrow{\phi[-k]} M[1]$
		It follows from the above discussion that $\d_{H^0(E[-k])}=0$ but 
		$$E[-k] = P_E[1-k] \oplus P[-k] \oplus R[-k] \oplus I[-k]$$
		so that $P_E[1-k] \oplus P[-k] \oplus I[-k]=0$ and thus $P_E=P=I=0$, which is a contradiction. It follows that $E=R$ is a regular module. Moreover, as there are neither morphisms nor extensions between different tubes, it follows that $E \in \add \mathcal T$.
	\end{proof}

	\begin{maprop}\label{prop:dcpXM}
		\begin{enumerate}
			\item Let $Q$ be an affine quiver and $M$ be an object in $\mathcal C_Q$. Then $X_M$ can be written as a $\Q$-linear combination of $X_Y$ where $Y$ is an object such that $\ddim H^0(Y) \leq \ddim H^0(M)$ and $\Ext^1_{\mathcal C_Q}(Y_i,Y_j)=0$ for any two distinct direct summands of $Y$. 
			\item If moreover $M$ is taken in $\add \mathcal T$ for some tube $\mathcal T$ of the AR-quiver of $Q$, then the $Y$ can be taken also is $\add \mathcal T$.
		\end{enumerate}
	\end{maprop}
	\begin{proof}
		We prove the it assertion by induction on the dimension of the self-extension. If $M$ is rigid, the result holds. Assume now that $M$ is not rigid. If $M$ is indecomposable, the result holds so we can assume that $M=M_1 \oplus M_2$ with $[M_1,M_2]^1\neq 0$. Corollary \ref{corol:inductionstep} implies that $X_M=X_{M_1}X_{M_2}$ is a $\Q$-linear combination of $X_Y$ with $[Y,Y]^1 < [M_1 \oplus M_2,M_1 \oplus M_2]=[M,M]^1$. Thus, by induction each $X_Y$ is a $\Q$-linear combination of $X_Z$ where $Z$ is such that $\Ext^1_{\mathcal C_Q}(Z_i,Z_j)=0$ for any two distinct direct summands of $Z$. The fact that $\ddim H^0(Y) \leq \ddim H^0(M)$ follows directly from lemma \ref{lem:dimtermecentral}. This proves the first point. The second point follows from lemma \ref{lem:Centremodule}.
	\end{proof}	
\end{subsection}

\section{Generic variables}\label{section:generic}
	\begin{subsection}{Generic generalized variables for acyclic quivers}\label{subsection:generic}
	In this subsection, we introduce the notion of generic variables, which will be the central objects in this article. They will turn out to be good candidates to construct a $\Z$-basis in a cluster algebra of affine type. As in the work of \cite{GLS}, these elements will be generic values for some particular constructible functions. The following construction holds for any acyclic quiver.
	
	\begin{subsubsection}{Existence of generic variables}	
		Let $Q$ be an acyclic quiver, we denote by $Q_0$ the set of vertices, $Q_1$ the set of arrows. For any $\alpha: i \fl j \in Q_1$, we denote by $s(\alpha)=i$ the source of $\alpha$ and $t(\alpha)=j$ the tail of $\alpha$. For any dimension vector $\textbf d \in \N^{Q_0}$ , we denote by $\rep(Q,\textbf d)$ the \emph{representation variety of dimension $\textbf d$}\emph{representation variety}, it is isomorphic to
		$$\prod_{\alpha \in Q_1} k^{d_{s(\alpha)}} \times k^{d_{t(\alpha)}},$$
		in particular it is an irreducible affine variety.
	
		For the existence of subset of $\rep(Q,\textbf d)$ giving a generic value to the Caldero-Chapoton map, we prove the existence of generic values for Euler characteristics of submodule grassmannians.
			
		The following lemma will be essential:
		\begin{monlem}\label{lem:Ude}
			Fix $Q$ an acyclic quiver, $\textbf d$ a dimension vector. For every dimension vector $\textbf e$, there exists a dense open subset $U_{\textbf d,\textbf e} \subset \rep(Q,\textbf d)$ such that the map 
			$$M \mapsto \chi(\Gr_{\textbf e}(M))$$ 
			is constant over $U_{\textbf d, \textbf e}$. Moreover, if $U'_{\textbf d, \textbf e}$ is another such dense open subset, the value is the same on $U_{\textbf d, \textbf e}$ and $U'_{\textbf d, \textbf e}$.
		\end{monlem}
		\begin{proof}
			We consider the algebraic variety 
			$$V_{\textbf e}=\ens{(M,N) \in \rep(Q,\textbf d) \times \rep(Q,\textbf e) \ : \ N\textrm{ is a submodule of }M}$$
			We denote by $f_{\textbf e}: V_{\textbf e} \fl \rep(Q,\textbf{d})$ the first projection. 
			For any representation $M$, the grassmannian $\Gr_{\textbf e}(M)$ is the fiber of $M$ under the morphism $f_{\textbf e}$. If $f_{\textbf e}$ is a dominant morphism, then according to \cite[Corollary 5.1]{verdier}, there exists a non-empty subset $U_{\textbf d,\textbf e} \subset \rep(Q,\textbf{d})$, open for the Zariski topology (and thus dense by irreducibility) such that the fibers of $f_{\textbf e}$ are isomorphic on $U_{\textbf d,\textbf e}$. In particular, they have the same Euler characteristic. Now if $f_{\textbf e}$ is not dominant, then $\codim(\im f_{\textbf e})>0$ and thus the Euler characteristic of $\Gr_{\textbf e}(M)$ vanishes on an open subset of $\rep(Q,\textbf d)$.
			Assume now that $U'_{\textbf d, \textbf e}$ is another such subset. Then by irreducibility $U_{\textbf d, \textbf e} \cap U'_{\textbf d, \textbf e} \neq \emptyset$ and thus the values coincide.
		\end{proof}
			
		\begin{notation}
			Fix $\textbf d, \textbf e \in \N^{Q_0}$, we will use the following notations:
			$$\textbf e \leq \textbf d \Leftrightarrow e_i \leq d_i \textrm{ for all }i \in Q_0$$
			$$\textbf e \lneqq \textbf d \Leftrightarrow e_i \leq d_i \textrm{ for all }i \in Q_0 \textrm{ and } \textbf e \neq \textbf d$$
			
			If $\textbf d \in \Z^{Q_0}$, we write
			$$[\textbf d]_+=\left(\max(d_i,0)\right)_{i \in Q_0}$$
			$$[\textbf d]_-=\left(\min(d_i,0)\right)_{i \in Q_0}$$
		\end{notation}
		
		Note that if $M$ is a rigid object in $\mathcal C_Q$, we have 
		$$[\ddim M]_+=\ddim H^0(M).$$
		Indeed, if $M$ is a module the result is clear. Now assume that $P_i[1]$ is a direct summand of $M$, then 
		$$\dim H^0(M)(i)=\dim \Hom_{kQ}(P_i,H^0(M))\leq \dim \Ext^1_{\mathcal C_Q}(H^0(M),P_i) =0.$$
	
		\begin{corol}\label{corol:Ud}
			Fix $Q$ an acyclic quiver and $\textbf d \in \N^{Q_0}$. Then there exists a dense open subset $U_{\textbf d} \subset \rep(Q,\textbf d)$ such that the Caldero-Chapoton map is constant over $U_{\textbf d}$. Moreover, if $U'_{\textbf d}$ is another such open subset, the values of the Caldero-Chapoton map are the same on $U_{\textbf d}$ and $U'_{\textbf d}$.
		\end{corol}
		\begin{proof}
			Consider 
			$$U_{\textbf d}=\bigcap_{\textbf e \leq \textbf d} U_{\textbf d, \textbf e}.$$
			It is a finite intersection of dense open subsets, it is thus open and dense in $\rep(Q,\textbf d)$. If $\textbf e \leq \textbf d$, we denote by $n_{\textbf e}$ the value of $\chi(\Gr_{\textbf e}(-))$ on $U_{\textbf d, \textbf e}$. For any $M \in U_{\textbf d}$, we thus have
			\begin{align*}
				X_M 
					&= \sum_{\textbf e \leq \textbf d} \chi(\Gr_{\textbf e}(M)) \prod_i u_i^{-\<\textbf e, \alpha_i\>-\<\alpha_i, \textbf d - \textbf e\>}\\
					&=\sum_{\textbf e \leq \textbf d} n_{\textbf e} \prod_i u_i^{-\<\textbf e, \alpha_i\>-\<\alpha_i, \textbf d - \textbf e\>}\\
			\end{align*}
			it is thus constant over $U_{\textbf d}$.
			
			Assume now that $U'_{\textbf d}$ is another such subset, then $U_{\textbf d} \cap U'_{\textbf d} \neq \emptyset$ and the values coincide on both subsets.
		\end{proof}
		
		Note that for every $\textbf d$, the set $U_{\textbf d}$ is a $G_{\textbf d}$-invariant subset of $\rep(Q,\textbf d)$.
		
		\begin{defi}
			Fix $\textbf d \in \N^{Q_0}$, we denote by $X_{\textbf d}$ the value of the Caldero-Chapoton map on $U_{\textbf d}$. Fix now $\textbf d \in \Z^{Q_0}$, we set
			\begin{align*}
				X_{\textbf d} 
					&=X_{[\textbf d]_+}\prod_{d_i<0} u_i^{-d_i}\\
					&=X_{[\textbf d]_+}\prod_{d_i<0} X_{P_i[1]}^{-d_i}\\
					&=X_{[\textbf d]_+}\prod_{d_i<0} X_{P_i[1]^{\oplus (-d_i)}}
			\end{align*}
			and $X_{\textbf d}$ is called the \emph{generic variable of dimension $\textbf d$}\index{generic variable}.
		\end{defi}
	\end{subsubsection}
	
	\begin{subsubsection}{Generic variables and canonical decomposition}\label{subsection:dcpcanonique}
		Fix $Q$ an acyclic quiver and $\textbf d \in \N^{Q_0}$ a dimension vector. According to \cite{Kac:infroot1, Kac:infroot2}, there exists a dense open subset $\mathfrak M_{\textbf d} \subset \rep(Q,\textbf d)$ and a family of Schur root $\ens{\textbf e_1, \ldots, \textbf e_n}$ in $\N^{Q_0}$ such that every representation $M \in \mathfrak M_{\textbf d}$ decomposes into 
		$$M=\bigoplus_{i=1}^n M_i$$
		with $M_i$ an indecomposable Schur representation of dimension $\textbf e_i$. In particular $\textbf d=\sum_{i=1}^n \textbf e_i$ and the family $\ens{\textbf e_1, \ldots, \textbf e_n}$ is uniquely determined. This decomposition of $\textbf d$ is called the \emph{canonical decomposition of $\textbf d$}\index{canonical decomposition} and is denoted by 
		$$\textbf d = \bigoplus_{i=1}^n \textbf e_i.$$
		
		We have the following characterization:
		\begin{maprop}[\cite{Kac:infroot2}]\label{prop:Kacdcp}
			Fix $Q$ an acyclic quiver and $\textbf d \in \N^{Q_0}$. Then $\textbf d=\bigoplus_{i=1}^n \textbf e_i$ if and only if there exist Schur representations $M_i \in \rep(Q,\textbf e_i)$ for every $i=1,\ldots, n$ such that $\Ext^1_{kQ}(M_i,M_j)=0$ if $i \neq j$. 
		\end{maprop}

		Following Schofield, we set the definition:
		\begin{defi}
			For any $\textbf d, \textbf d' \in \N^{Q_0}$, we say that $\Ext^1_{kQ}(\textbf d, \textbf d')$ \emph{vanishes generally}\index{general vanishing} if there exists $M \in \rep(Q,\textbf d)$, $M' \in \rep(Q,\textbf d')$ such that 
			$$\Ext^1_{kQ}(M,M')=0$$
			and if it is the case, we denote it by $\Ext^1_{kQ}(\textbf d, \textbf d')=0$.
		\end{defi}
	
		For any dimension vector $\textbf d \in \N^{Q_0}$, one has $U_{\textbf d} \cap \mathfrak M_{\textbf d} \neq \emptyset$, thus $X_{\textbf d}$ is equal to a $X_M$ for some module $M \in \mathfrak M_{\textbf d}$. One can thus use the canonical decomposition to compute generic variables. More precisely, a helpful result is the following:

		\begin{maprop}\label{prop:dcpcanonique}
			Fix $Q$ an acyclic quiver, $\textbf d \in \N^{Q_0}$ a dimension vector with canonical decomposition $\textbf d= \textbf e_1 \oplus \cdots \oplus \textbf e_n$. Then
			$$X_{\textbf d}=\prod_{i=1}^n X_{\textbf e_i}.$$
		\end{maprop}
		\begin{proof}
		 	Consider the injective morphism :
		 	$$\phi : \left\{ \begin{array}{rcl}
				\rep(Q,\textbf e_1) \times \cdots \times \rep(Q,\textbf e_n) & \fl & \rep(Q,\textbf d)\\
				(M_1, \ldots, M_n) & \mapsto & \bigoplus_{i=1}^n M_i
			\end{array}\right.$$

			By proposition \ref{prop:Kacdcp}, as $\textbf d=\textbf e_1 \oplus \cdots \oplus \textbf e_n$ is the canonical decomposition of $\textbf d$, one has $\Ext^1_{kQ}(\textbf e_i, \textbf e_j)=0$ for every $i \neq j$. It thus follows from \cite{CBS} that $\phi$ is a dominant morphism.
			
			We set
			$$\mathcal U=(\mathfrak M_{\textbf e_1} \cap U_{\textbf e_1})\times \cdots \times (\mathfrak M_{\textbf e_n} \cap U_{\textbf e_n}),$$
			this is a dense open subset in $\rep(Q,\textbf e_1) \times \cdots \times \rep(Q,\textbf e_n)$ and thus $\phi(\mathcal U)$ is a dense open subset in $\rep(Q,\textbf d)$. In particular $\phi(\mathcal U) \cap U_{\textbf d} \neq \emptyset$ and we can choose some $M \in \phi(\mathcal U) \cap U_{\textbf d}$. We thus have $X_{\textbf d}=X_M$ and $M$ decomposes into a direct sum 
			$$M=\bigoplus_{i=1}^n M_i$$
			with $M_i \in \mathfrak M_{\textbf e_i} \cap U_{\textbf e_i}$. It follows that 
			\begin{align*}
			 	X_\textbf d
			 		&=X_M \\
			 		&=X_{\bigoplus_{i=1}^n M_i}\\
			 		&=\prod_{i=1}^n X_{M_i}\\
			 		&=\prod_{i=1}^n X_{\textbf e_i}.
			\end{align*}
		\end{proof}
		
		The following lemma gives a multiplicative property for generic variables. For this, we extend the notion of general vanishing of the Ext spaces to the cluster category.
		
		\begin{defi}
		 	Let $\textbf d, \textbf d'$ be two elements in $\Z^{Q_0}$, we say that $\Ext^1_{\mathcal C_Q}(\textbf d, \textbf d')$ \emph{vanishes generally}\index{general vanishing} if there exists $M \in \rep(Q, [\textbf d]_+)$, $M' \in \rep(Q,[\textbf d']_+)$ such that
		 	$$\Ext^1_{\mathcal C_Q}\left( M \oplus \bigoplus_{d_i<0} P_i[1]^{\oplus(-d_i)}, M' \oplus \bigoplus_{d'_i<0} P_i[1]^{\oplus(-d'_i)}\right)=0$$
		 	and we denote this by $\Ext^1_{\mathcal C_Q}(\textbf d, \textbf d')=0$.
		\end{defi}

		Note in particular that for $\textbf d, \textbf d' \in \N^{Q_0}$, if $\Ext^1_{\mathcal C_Q}(\textbf d, \textbf d')=0$ then $\Ext^1_{kQ}(\textbf d, \textbf d')=0$ and $\Ext^1_{kQ}(\textbf d', \textbf d)=0$.

		\begin{monlem}\label{lem:multiplicativity}
			Let $Q$ be an acyclic quiver and $\textbf d, \textbf d' \in \Z^{Q_0}$ such that $\Ext^1_{\mathcal C_Q}(\textbf d, \textbf d')=0$. Then 
			$$X_{\textbf d}X_{\textbf d'}=X_{\textbf d+\textbf d'}.$$
		\end{monlem}
		\begin{proof}
			Assume that there is some $i \in Q_0$ such that $d_i<0$ and $d_i'>0$, then for every $M' \in \rep(Q,\textbf d')$, $0<d_i' =  \dim \Hom_{kQ}(P_i,M') = \dim \Ext^1_{\mathcal C_Q}(M',P_i[1])$ and thus $\Ext^1_{\mathcal C_Q}(\textbf d, \textbf d')$ does not vanish generally. Thus, we can assume that $d_i$ and $d_i'$ are of the same sign for every $i \in Q_0$, in particular, $[\textbf d+\textbf d']_+=[\textbf d]_++[\textbf d']_+$.
			As $\Ext^1_{\mathcal C_Q}(\textbf d, \textbf d')=0$, we have $\Ext^1_{kQ}(\textbf d, \textbf d')=0$ and $\Ext^1_{kQ}(\textbf d', \textbf d)=0$, it follows from \cite{CBS} that the morphism
			$$\phi: \left\{\begin{array}{rcl}
				\rep(Q,[\textbf d]_+) \times \rep(Q,[\textbf d']_+) & \fl & \rep(Q,[\textbf d+\textbf d']_+)\\
				(U,V) & \mapsto & U\oplus V
			\end{array}\right.$$
			is dominant.
			As $U_{[\textbf d]}$ (resp $U_{[\textbf d']_+}$) is open in $\rep(Q,[\textbf d]_+)$ (resp. in $\rep(Q,[\textbf d']_+)$), it follows that $U_{[\textbf d]} \oplus U_{[\textbf d']_+}$ is open and dense in $\rep(Q,[\textbf d+\textbf d']_+)$. Thus, $X_{[\textbf d+\textbf d']_+}=X_{[\textbf d]_+}X_{[\textbf d']_+}$ and 
			\begin{align*}
				X_{\textbf d}X_{\textbf d'}
					&= X_{[\textbf d]_+}\left(\prod_{d_i<0}X_{P_i[1]^{\oplus(-d_i)}}\right)X_{[\textbf d']_+}\left(\prod_{d'_i<0}X_{P_i[1]^{\oplus(-d'_i)}}\right)\\
					&= X_{[\textbf d+\textbf d']_+}\prod_{d_i+d_i'<0}X_{P_i[1]^{\oplus(-(d_i+d_i'))}}\\
					&= X_{\textbf d+\textbf d'}
			\end{align*}
		\end{proof}
		
		 The following proposition proves that generic variables are indeed generalizing cluster monomials.
		\begin{monlem}\label{lem:clustermonomial}
			Let $Q$ be an acyclic quiver, fix $x$ a cluster monomial in $\mathcal A(Q)$ with denominator vector $\delta(x)$. Then
			$$x=X_{\delta(x)}.$$
			Equivalently, if $M$ is a rigid object, we have
			$$X_M=X_{\ddim M}.$$
		\end{monlem}
		\begin{proof}
			According to theorem \ref{theorem:correspondanceCK2}, if $x$ is a cluster monomial, it can be written $X_M$ for some rigid object $M$ and theorem \ref{theorem:denominators} implies that $\delta(X_M)=\ddim M$. Moreover, as $M$ is rigid, we have $[\ddim M]_+=\ddim H^0(M)$ and $[\ddim M]_-=\ddim P_M[1]$. $H^0(M)$ being a rigid module in $\rep(Q,[\ddim M]_+)$, its orbit $\mathcal O_{H^0(M)}$ is dense in $\rep(Q,[\ddim M]_+)$ and thus $\mathcal O_{H^0(M)} \cap U_{[\ddim M]_+} \neq \emptyset$. This implies that $X_{H^0(M)}=X_{[\ddim M]_+}$. Now, we always have $X_{P_M[1]}=X_{[\ddim M]_-}$. It follows that
			\begin{align*}
				x
					&= X_M\\
					&= X_{H^0(M)\oplus P_M[1]}\\
					&= X_{H^0(M)}X_{P_M[1]}\\
					&= X_{[\ddim M]_+}X_{[\ddim M]_-}\\
			\end{align*}
			but $\Ext^1_{\mathcal C_Q}([\ddim M]_+,[\ddim M]_-)$ vanishes generally because $M$ is a rigid module and so lemma \ref{lem:multiplicativity} implies that
			\begin{align*}
			 		X_{[\ddim M]_+}X_{[\ddim M]_-}
			 		&=X_{[\ddim M]_++[\ddim M]_-}\\
			 		&=X_{\ddim M}\\
			 		&=X_{\delta(x)}\\
			\end{align*}
			which proves the lemma.
		\end{proof}

		Through the works of \cite{CK1,shermanz, GLS:rigid2, CZ}, it became clear that the $\Z$-bases for cluster algebras should naturally contain the cluster monomials. Lemma \ref{lem:clustermonomial} proves that they turn out to be generic variables. Moreover, the analogy with the works of \cite{GLS:rigid2} in the context of preprojective algebras gives us the hope that the set of generic variables is a good candidate for being a $\Z$-basis in cluster algebras. We thus get interested in the set of all generic variables.
		
		\begin{defi}
		 	Let $Q$ be an acyclic quiver, we denote by $\mathcal B'(Q)$ the set of all generic variables
		 	$$\mathcal B'(Q)=\ens{X_{\textbf d} \ : \ \textbf d \in \Z^{Q_0}}.$$
		\end{defi}
	\end{subsubsection}
\end{subsection}

\begin{subsection}{Generic variables for affine quivers}\label{subsection:genericA}
	In this subsection, we obtain an explicit description of the generic variables when $Q$ is a quiver of affine type.  We denote by $\delta$ its minimal imaginary root. We recall that the positive root system of $Q$ can be written 
	$$\Phi_{\geq 0}(Q)=\Phi_{\geq 0}^{\textrm{re}}(Q) \sqcup \Phi_{\geq 0}^{\textrm{im}}(Q)$$
	where $\Phi_{\geq 0}^{re}(Q)$ is the set of positive real roots and $\Phi_{\geq 0}^{\textrm{im}}(Q)=\N\delta$ is the set of imaginary roots. For details concerning representation theory of affine quivers, one can for example refer to \cite{ringel:1099, DR:memoirs}.
	
	\begin{subsubsection}{Generic variables associated to positive real Schur roots}
		\begin{monlem}\label{lem:realSchur}
			Let $Q$ be an affine quiver and $\textbf d$ be a positive real Schur root of $Q$. Then there exists an unique indecomposable rigid module $M$ of dimension $\textbf d$. Moreover, $X_{\textbf d}=X_{M}$.
		\end{monlem}
		\begin{proof}
			As $\textbf d$ is a real root, Kac's theorem ensures that there exists an unique indecomposable representation $M$ of dimension $\textbf d$. Moreover, as $\textbf d$ is a Schur root, this representation has a trivial endomorphism ring. Now $1=\<\textbf d, \textbf d\>=\<M,M\>=\dim \End_{kQ}(M)-\dim \Ext^1_{kQ}(M,M)$ so $\dim \Ext^1_{kQ}(M,M)=0$. It follows from lemma \ref{lem:clustermonomial} that $X_M=X_{\ddim M}=X_{\textbf d}$.
		\end{proof} 
	\end{subsubsection}

	\begin{subsubsection}{Generic variables associated to positive imaginary roots}
		We write $\P^1=\P^1(k)$ the index set of tubes in the AR quiver of $kQ$ and by $\P^1_0$ the index set of homogeneous tubes. For any $\lambda \in \P^1$, we will denote by $\mathcal T_\lambda$ the tube indexed by $\lambda$. If $\lambda \in \P^1_0$, $M_\lambda$ denotes the quasi-simple in $\mathcal T_\lambda$ and $M_\lambda^{(n)}$ denotes the indecomposable module in $\mathcal T_\lambda$ with quasi-socle $M_\lambda$ and quasi-length $n$. In particular, $\ddim M_\lambda^{(n)}=n \delta$ for any $n \geq 1$ and $\lambda \in \P^1_0$.
		
		It is convenient to introduce the so-called \emph{normalized Chebyshev polynomials of the second kind}\index{Chebyshev polynomial!normalized of the second kind} (called \emph{generalized Chebyshev polynomials of rank 2} in \cite{Dupont:stabletubes}). For any $n \geq 0$, the $n$-th normalized Chebyshev polynomial of the second kind is the polynomial $C_n$ defined by 
		$$C_n(t+t^{-1})=\sum_{k=0}^{n}t^{n-2k}$$
		On can refer to subsection \ref{subsection:basesKronecker} (see also \cite{CZ} and \cite{Dupont:stabletubes}) for details concerning these polynomials.
		
		The following lemma is proved in \cite{CZ}, we recall it for completeness.
		\begin{monlem}\label{lem:Chebyshev}
			Fix $Q$ an affine quiver, $n \geq 1$ and $\lambda \in \P^1_0$. Then
			$$X_{M_\lambda^{(n)}}=C_n(X_{M_\lambda})$$
		\end{monlem}
		\begin{proof}
		 	For any $n \geq 1$ and any $\lambda \in \P^1_0$, we have an almost split sequence
			$$0 \fl M_\lambda^{(n)} \fl M_\lambda^{(n+1)} \oplus M_\lambda^{(n-1)} \fl M_\lambda^{(n)} \fl 0.$$
			The almost split multiplication formula implies thus that
			$$X_{M_\lambda^{(n)}}^2=X_{M_\lambda^{(n-1)}}X_{M_\lambda^{(n+1)}}+1$$
			and thus that $M_\lambda^{(n)}=C_n(M_\lambda)$ where $C_n$ is the second kind normalized Chebyshev polynomial. 
		\end{proof}

		We now prove that for an affine quiver, the Caldero-Chapoton map does not depend on the considered homogeneous tube.
		\begin{monlem}\label{lem:XMlambda}
			Let $Q$ be an affine quiver of affine type. Then for any $\lambda,\mu \in \P^1_0$ and any $n \geq 1$, we have 
			$$X_{M_\lambda^{(n)}}=X_{M_\mu^{(n)}}$$
		\end{monlem}
		\begin{proof}
			By lemma \ref{lem:Chebyshev}, it is enough to prove it for $n=1$. For this, we recall Crawley-Boevey's construction of regular representations of dimension $\delta$ by one-point extensions \cite{CB:lectures}. Let $e$ be an extending vertex of $Q$, that is a vertex $e \in Q_0$ such that $\delta_e=1$ and that the quiver obtained from $Q$ by deleting the vertex $e$ is of Dynkin type. We denote by $P=P_e$ the projective module associated to the vertex $e$ and by $\textbf p=\ddim P$ its dimension vector. Then the defect of $P$ is $\d_P=-1$. Let $L$ be the unique indecomposable representation of dimension vector $\delta +\textbf p$. Then $L$ has also defect $-1$. Crawley-Boevey proved that $\Hom_{kQ}(P,L) \simeq k^2$ and that for every morphism $0 \neq \lambda \in \Hom_{kQ}(P,L)$, $\coker \lambda$ is an indecomposable regular module of dimension vector $\delta$. Moreover, $\coker \lambda \simeq \coker \lambda'$ if and only if $\lambda = x \lambda'$ for some $0 \neq x \in k$. Then $\lambda \fl \coker \lambda$ induces a bijection from $\P\Hom_{kQ}(P,L)$ to the set of all tubes of $Q$ by sending $\lambda \in \P\Hom_{kQ}(P,L)$ to an indecomposable regular module of dimension $\delta$. We thus identify $\P\Hom_{kQ}(P,L)$ and $\P^1$ and we will say that $\lambda \in \P^1_0$ to signify that $\coker \lambda$ is a quasi-simple in the homogeneous tube $\mathcal T_{\lambda}$. Note that in this case $M_\lambda \simeq \coker \lambda$.
			
			Fix $0 \neq \lambda \in \Hom_{kQ}(P,L)$, then
			\begin{align*}
				\Ext^1_{\mathcal C_Q}(P,\coker \lambda) 
					& \simeq \Ext^1_{kQ}(P,\coker \lambda) \oplus \Ext^1_{kQ}(\coker \lambda,P) \\
					& \simeq \Ext^1_{kQ}(\coker \lambda,P) \\
					& \simeq \Hom_{kQ}(P,\coker \lambda) \\
					& \simeq k
			\end{align*}
			because $P=P_e$ and $e$ is an extending vertex.
			
			It follows from Caldero-Keller's multiplication formula that 
			$$X_PX_{\coker \lambda}=X_L + X_B$$
			where $B \simeq \ker f \oplus \coker f[-1]$ for some morphism $0 \neq f \in \Hom_{kQ}(P,\tau \coker \lambda)$.
			
			Fix $\lambda \in \P^1_0$. Assume that $f$ is not surjective. Then $\im f$ is preprojective and we have a short exact sequence
			$$0 \fl \ker f \fl P \xrightarrow{f} \im f \fl 0$$
			and thus $\d_{\ker f}+\d_{\im f}=\d_P=-1$ but $\ker f$ and $\im f$ are preprojective so necessarily $\d_{\ker f}=0$ and $\ker f=0$. 
			It follows that $B \simeq \tau^{-1}\coker f$. But
			$$\d_{\coker f} =\d_{(\tau \coker \lambda)/P}=-\d_P=1$$
			and $\tau \coker \lambda \simeq \coker \lambda$ is quasi-simple in an homogeneous tube. Thus, $\coker f$ is preinjective and by defect, it is indecomposable. Moreover, $\ddim \coker f=\delta-\textbf p$. Thus, for every $\lambda \in \P^1_0$ and every non-surjective map $0 \neq f \in \Hom_{kQ}(P,\tau \coker \lambda)$, $\coker f$ is the unique indecomposable representation $V$ of dimension $\delta-\textbf p$. If there exists some non-surjective $f \in \Hom_{kQ}(P,\tau \coker \lambda)$, then $f$ is injective, so $\ddim P \lneqq \delta$ and thus every map $g \in \Hom_{kQ}(P,\tau \coker \lambda)$ is non-surjective for any $\lambda \in \P^1_0$. Thus,
			$$X_PX_{\coker \lambda}=X_L+X_V[-1]$$
			for any $\lambda \in \P^1_0$. In particular, $X_{M_\lambda}$ does not depend on the parameter $\lambda \in \P^1_0$.
			
			Assume now that $0 \neq f \in \Hom_{kQ}(P,\tau \coker \lambda)$ is surjective. Then $\ker f$ is a preprojective module of dimension $\textbf p-\delta$. Thus, $\ker f$ has defect $-1$, it is thus necessarily indecomposable. As there exists an unique indecomposable representation $U$ of dimension $\textbf p-\delta$, it follows that $\ker f \simeq U$ for every $\lambda \in \P^1_0$ and every surjective $f \in \Hom_{kQ}(P,\tau \coker \lambda)$. Then, it follows from the above discussion that if one of the $f \in \Hom_{kQ}(P,\tau \coker \lambda)$ is surjective for some $\lambda \in \P^1_0$, then every $f\in \Hom_{kQ}(P,\tau \coker \lambda)$ is surjective for any $\lambda \in \P^1_0$. In this case, we get
			$$X_PX_{\coker \lambda}=X_L+X_U$$
			for every $\lambda \in \P^1_0$. In particular, $X_{M_\lambda}$ does not depend on the parameter $\lambda \in \P^1_0$.
		\end{proof}

		If $E$ is a quasi-simple module in an exceptional tube $\mathcal T$, we denote by $E^{(n)}$ the indecomposable module with quasi-socle $E$ and quasi-length $n$. If $p$ is the rank of $\mathcal T$ and $k \geq 1$ is an integer, we will simplify the notations by writing $M_E^{(k)}=E^{(kp)}$ and $M_E=E^{(p)}$. Note in particular that $\ddim M_E^{(k)}=k\delta$.

		\begin{monlem}\label{lem:imaginary}
			Fix $n \geq 1$ a positive integer. Then
			$$X_{n \delta}=X_{M_\lambda}^n$$
			for any $\lambda \in \P^1_0$.
		\end{monlem}
		\begin{proof}
			We first prove it for $n=1$. The canonical decomposition of $\delta$ is $\delta$ itself. It follows that $X_\delta=X_M$ for some indecomposable module $M$ of dimension vector $\delta$. Now, we know that the indecomposable modules of dimension vector $\delta$ are either the $M_\lambda$ for $\lambda \in \P^1_0$, or the $M_E$ for $E$ quasi-simple in an exceptional tube. Fix now a quasi-simple $E$ in an exceptional tube, we claim that $\Gr_{\ddim E}(M)=\emptyset$ for any indecomposable module $M$ of dimension vector $\delta$ not isomorphic to $M_E$. Indeed, fix $\lambda \in \P^1_0$ and $U \subset M_\lambda$ a submodule such that $\ddim U=\ddim E$. As $M_\lambda$ is quasi-simple, $U$ has to be preprojective and thus $\d_U <0$ but $\d_U=\d_E=0$, which is a contradiction and thus $\Gr_{\ddim E}(M_\lambda)=\emptyset$. Fix now $F$ another quasi-simple and assume that $U \subset M_F$ is a submodule such that $\ddim U=\ddim E$. It follows that $U$ decomposes into $U=U_P \oplus U_R$ where $U_P$ is preprojective and $U_R$ is regular. As $\d_U=\d_E=0$, we have $U_P=0$ and thus $U_R$ is regular. Now $\ddim U_R=\ddim E$ but $U_R$ is a regular submodule of $M_F$. By uniseriality of the regular components, $U_R$ has to be indecomposable. As $\ddim E$ is a real root, it follows that $E \simeq U_R \subset M_F$ and thus $F=E$. As the $G_{\delta}$-orbit of $M_E$ is not open, it follows that the value of $\chi(\Gr_{\ddim E}(-))$ has to be zero on $U_{\delta,\ddim E}$ and so $U_\delta \cap \mathcal O_{M_E} = \emptyset$ for any quasi-simple $E$ in an exceptional tube. It follows that 
			$$U_{\delta} \cap \mathfrak M_{\delta} \subset \bigsqcup_{\lambda \in \P^1_0} \mathcal O_{M_\lambda}.$$ Thus, $X_\delta=X_{M_\lambda}$ for some $\lambda \in \P^1_0$ and lemma \ref{lem:XMlambda} implies that $X_\delta=X_{M_\lambda}$ for any $\lambda \in \P^1_0$.
			
			Now if $n >1$, the canonical decomposition of $n\delta$ is $\delta \oplus \cdots \oplus \delta$. Proposition \ref{prop:dcpcanonique} implies then that 
			$$X_{n\delta}=X_\delta^n=X_{M_\lambda}^n.$$
		\end{proof}
		
		\begin{rmq}
			Note that it is not clear at this time that $X_{M_E} \neq X_{M_\lambda}$ if $\lambda \in \P^1_0$ and $E$ is a quasi-simple in an exceptional tube. It will appear in theorem \ref{theorem:differencedelta} that these values are in fact always different and that the difference can be completely described. Moreover, it will also turn out that if $E$ and $F$ are non-isomorphic quasi-simple modules taken in exceptional tubes, then $X_{M_E} \neq X_{M_F}$ if $E$ and $F$ do not belong both to tubes of rank 2.
		\end{rmq}
	\end{subsubsection}

 	\begin{subsubsection}{Generic variables associated to positive real non-Schur roots}
		Now, it remains to compute $X_{\textbf d}$ when $\textbf d$ is a real root which is not a Schur root. If $\textbf d$ is such a root, then according to Kac's theorem, there exists an unique indecomposable module $M$ of dimension vector $\textbf d$. If $M$ is preprojective or preinjective, it is known that $\End_{kQ}(M) \simeq k$ and thus $\textbf d$ is a real Schur root. It follows that $M$ has to be a regular module and  $\delta \lneqq \textbf d$ (see \cite{CB:lectures} for example). It is proved in \cite{Kac:infroot2} that $\textbf d = \delta \oplus \cdots \oplus \delta \oplus \textbf d_0$ where $\textbf d_0$ is the root of smallest height in $(\textbf d+ \Z \delta) \cap \Phi_{\geq 0}(Q)$. In particular, $\textbf d_0$ is a real Schur root. The following proposition gives an explicit description of $X_{\textbf d}$ in this case.
		
		\begin{maprop}\label{real:nonSchur}
			Let $Q$ be an affine quiver, $\textbf d$ be a real root which is not a Schur root and write $\textbf d=\delta^{\oplus n} \oplus \textbf d_0$ its canonical decomposition. Then 
			$$X_{\textbf d}=X_{M_\lambda}^nX_{M_0}$$
			for any $\lambda \in \P^1_0$ and $M_0$ being the unique indecomposable module of dimension vector $\textbf d_0$.
		\end{maprop}
		\begin{proof}
			According to lemmas \ref{lem:realSchur} and \ref{lem:imaginary}, $X_{\textbf d_0}=X_{M_0}$ where $M_0$ is the unique indecomposable module in $\rep(Q,\textbf d_0)$ and $X_{n \delta}=X_{\delta}^n=X_{M_\lambda}^n$ for any $\lambda \in \P^1_0$. Proposition \ref{prop:dcpcanonique} implies then that 
			$$X_{\textbf d}=X_{n\delta}X_{\textbf d_0}=X_{M_\lambda}^nX_{M_0}.$$
		\end{proof}
		
		We can now give a complete description of the generic variables:		
		\begin{maprop}\label{prop:explicitbase}
			Let $Q$ be an affine quiver and $\textbf d \in \Z^{Q_0}$, denote by $\mathcal E$ a set of representatives of isoclasses of regular modules $M$. Then
			$$\mathcal B'(Q)=\ens{\textrm{cluster monomials}} \sqcup \ens{X_{M_\lambda^{\oplus n} \oplus E} \ : n \geq 1, \ E \in \add \mathcal E \textrm{ is rigid}}$$
		\end{maprop}
		\begin{proof}
			Fix $\textbf d \in \Z^{Q_0}$, write $[\textbf d]_+=\delta^k \oplus \bigoplus_{i=1}^n e_i$ the canonical decomposition of $[\textbf d]_+$. If $k \neq 0$, then as $\delta$ is sincere, we have $\textbf d \in \N^{Q_0}$. According to proposition \ref{prop:Kacdcp} and lemma \ref{lem:imaginary}, there is some $M$ in $\mathfrak M_{\textbf d} \cap U_{\textbf d}$ such that
			$$M=\bigoplus_{j=1}^k M_{\lambda_j} \oplus \bigoplus_{i=1}^n M_i$$
			where the $\lambda_j \in \P^1_0$ are pairwise distinct and the $M_i$ are indecomposable rigid modules such that $\End_{kQ}(M_i) \simeq k$ and $\Ext^1_{kQ}(M_i,M_l)=0$ if $i \neq l$.
			
			If one of the $M_i$ is preprojective, then there is some vertex $v \in Q_0$ and some integer $s \geq 0$ such that $M_i\simeq \tau^{-s}P_v$ and then $$\Ext^1_{\mathcal C}(M_i, M_{\lambda_1})=\Hom_{\mathcal C}(P_v, M_{\lambda_1}) =\dim M_{\lambda_1}(v)=\delta_v=1$$ which is a contradiction. Similarly, none of the $M_i$ can be preinjective. It follows that each $M_i$ is regular and thus the $M_i$ are indecomposable regular modules in exceptional tubes such that $\ddim M_i \lneqq \delta$. As $\Ext^1_{kQ}(M_i,M_l)=0$ for $i \neq l$, it follows that $\bigoplus_{i=1}^n M_i$ is regular rigid.
			
			If $k=0$, then $M \in \mathfrak M_{\textbf d} \cap U_{\textbf d}$ is rigid and thus $M \oplus \bigoplus_{d_i<0}P_i[1]^{\oplus -d_i}$ is a rigid object. Thus $X_{\textbf d}$ is a cluster monomial.
			
			This proves the inclusion 
			$$\mathcal B'(Q) \subset \ens{\textrm{cluster monomials}} \sqcup \ens{X_{M_\lambda^{\oplus n} \oplus E} \ : n \geq 1, \ E \in \add \mathcal E \textrm{ is rigid}}.$$
			
			We now prove the reverse inclusion. According to lemma \ref{lem:clustermonomial} it suffices to prove that $\ens{X_{M_\lambda^{\oplus n} \oplus E} \ :n \geq 1 \ E \in \add \mathcal E \textrm{ is rigid}} \subset \mathcal B'(Q)$. Assume that $E \in \add \mathcal E$ is rigid. Then $X_{\ddim E}=X_E$ by lemma \ref{lem:clustermonomial}. On the other hand, $X_{M_\lambda^{\oplus n}}=X_{n\delta}$ by lemma \ref{lem:imaginary}. Fix now $\lambda_1, \ldots, \lambda_n$ pairwise distinct elements in $\P^1_0$, as there are no extensions between the tubes we have
			$$\Ext^1_{\mathcal C_Q}(E, M_{\lambda_1} \oplus \cdots \oplus M_{\lambda_n})=0$$
			so  $\Ext^1_{\mathcal C_Q}(\ddim E, n \delta)$ vanishes generally and by lemma \ref{lem:multiplicativity}, we have 
			$$X_{\ddim E + n \delta}=X_{\ddim E}X_{n \delta}=X_{E \oplus M_\lambda^{\oplus n}}$$
			and this proves the proposition.
			
			It remains to notice that the union is disjoint. Indeed, fix $X_{\textbf d}=X_{M_\lambda^{\oplus n} \oplus E}$ with $E$ rigid in $\add \mathcal E$ and $n\geq 1$. Fix $\lambda_1, \ldots, \lambda_n$ pairwise distinct elements in $\P^1_0$, and decompose $E=E_1 \oplus \ldots \oplus E_m$ into indecomposable summands. Then 
			$$M_{\lambda_1} \oplus \cdots \oplus M_{\lambda_n}\oplus E_1 \oplus \cdots \oplus E_m$$
			is a direct sum of Schur representations such that $\Ext^1_{kQ}(U,V)=0$ for any two indecomposable direct summands $U,V$. It follows that 
			$$\textbf d=\delta^{\oplus n} \oplus \bigoplus_{i=1}^n \ddim E_i$$
			is the canonical decomposition of $\textbf d$. As $\dim \Ext^1_{kQ}(M,M) \neq 0$ for any representation in $\rep(Q,n\delta)$ for $n \geq 1$, there cannot be any rigid module in $\rep(Q,\textbf d)$. The denominators theorem ensures then that there is no rigid object $M$ in $\mathcal C_Q$ such that $X_M=X_{\textbf d}$.
		\end{proof}
	\end{subsubsection}
\end{subsection}

\begin{subsection}{Generic variables as generators in type $\Aaffine$}\label{subsection:generating}
	Following the ideas of \cite{CK1}, in order to prove that $\mathcal B'(Q)$ is a $\Z$-basis in $\mathcal A(Q)$ when $Q$ is affine of type $\Aaffine$. We first prove that every $X_M$ can be written as a $\Z$-linear combination of $X_{\textbf d}$. We will do it in two steps: the first step will consist in a study of the values of $X_M$ when $M$ is indecomposable and the second step will be devoted to the case where $M$ is decomposable.

	Fix $M$ an indecomposable object in $\mathcal C_Q$. If $M$ is in the transjective component. Then $M$ is rigid and thus lemma \ref{lem:clustermonomial} implies that $X_{M}\in \mathcal B'(Q)$. Assume now that $M$ is not in the transjective component. Then there are three possibilities:
	\begin{itemize}
		\item $\ddim M$ is a real Schur root,
		\item $\ddim M$ is an imaginary Schur root,
		\item $\ddim M$ is a positive non-Schur root.
	\end{itemize}

	In the case $\ddim M$ is a real Schur root, $M$ is rigid and thus $X_M \in \mathcal B'(Q)$ by lemma \ref{lem:clustermonomial}. The other cases need a deeper study.
	 	
	\begin{subsubsection}{Modules from homogeneous tubes}
	 	\begin{monlem}\label{lem:XMlambdaninbase}
	 		Let $Q$ be a quiver of affine type, $\lambda \in \P^1_0$ and $n \geq 1$. Then 
	 		$$X_{M_\lambda^{(n)}} \in \Z\mathcal B'(Q).$$
	 	\end{monlem}
	 	\begin{proof}
	 		We first note that $X_{M_\lambda}=X_{\delta} \in \mathcal B'(Q)$. Now, if $n\geq 1$, it follows from corollary \ref{lem:Chebyshev} that $X_{M_\lambda^{(n)}}=C_n(X_{M_\lambda})$. As $C_n$ is a polynomial of degree $n$ with integer coefficients, $X_{M_\lambda^{(n)}}$ is a $\Z$-linear combination of $X_{M_\lambda}^k$ for $k=0,\ldots,n$. According to lemma \ref{lem:imaginary}, $X_{M_\lambda}^n=X_{n\delta}$ for any $n \geq 1$ and then $X_{M_\lambda^{(n)}} \in \Z \mathcal B'(Q)$.
	 	\end{proof}
	\end{subsubsection} 
	
	\begin{subsubsection}{Modules from exceptional tubes}
		From now on, we focus on the case where $Q$ is a quiver of affine type $\Aaffine$ which is not isomorphic to the Kronecker quiver. We fix a quasi-simple module $E$ in an exceptional tube of rank $p>1$. We denote by $M_E^{(n)}$ the unique indecomposable module of dimension vector $n \delta$ (or equivalently of quasi length $np$) and quasi-socle $E$. If $M$ is indecomposable regular, we denote by $\regrad M$ its quasi-radical. We prove that for any $n \geq 1$, $X_{M_E^{(n)}} \in \Z \mathcal B'(Q)$. First, we prove the striking \emph{difference property}\index{difference property}
		$$X_{M_E}=X_{M_\lambda}+X_{\regrad M_E/E}$$
		for any $\lambda \in \P^1_0$. This will be one of the essential points of this article and require preliminary results. 
		
		In our study, grassmannians of submodules of the quasi-simple modules $M_\lambda$ will be of great interest. The following lemma simplifies this study in the particular case of a quiver of type $\Aaffine$.
		\begin{monlem}\label{lem:caracun}
			Let $Q$ be an acyclic quiver, $\textbf d \in \N^{Q_0}$ be a dimension vector such that $d_i \leq 1$ for any $i \in Q_0$. Fix $M \in \rep(Q,\textbf d)$ and $\textbf e \leq \textbf d$ another dimension vector. Then if $\Gr_{\textbf e}(M) \neq \emptyset$, it is a point. In particular, we have $\chi(\Gr_{\textbf e}(M))=1$.
		\end{monlem}
		\begin{proof}
			Assume that $\Gr_{\textbf e}(M) \neq \emptyset$ and fix $N \in \Gr_{\textbf e}(M)$. For every $i$ such that $e_i \neq 0$, $N(i)$ is a non-zero subspace of the one-dimensional vector space $M(i)$ and then $N(i)=M(i)$. 
			Thus, $N$ is the representation given by
			$$N(i) =\left\{\begin{array}{rl}
				M(i) & \textrm{ if } e_i\neq 0\\
				0 & \textrm{ otherwise }
			\end{array}\right.$$
			and $N(i \fl j)=M(i \fl j)_{|N(i)}$. Therefore, $\Gr_{\textbf e}(M)=\ens{N}$ is a point.
		\end{proof}
		
		The following theorem will be referred to as the \emph{Schofield's theorem}. It is of a great use in the study of grassmannians of submodules.
		\begin{theorem}[\cite{Schofield:generalrepresentations}]\label{theorem:Schofield}
			Let $Q$ be an acyclic quiver and $\textbf d, \textbf e \in \N^{Q_0}$ such that $\Ext^1_{kQ}(\textbf d, \textbf e)=0$. Then any representation of dimension vector $\textbf d+ \textbf e$ contains a sub-representation of dimension vector $\textbf d$.
		\end{theorem}
	 	
	 	\begin{monlem}\label{lem:caracpreproj}
	 	 	Fix $Q$ a quiver of affine type $\Aaffine$, $E$ a quasi-simple module in an exceptional tube, $\lambda \in \P^1_0$ and $\textbf v$ a dimension vector. If $\Gr_{\textbf v}(M_\lambda) \neq \emptyset$. Then $\Ext^1_{kQ}(\textbf v, \delta-\textbf v)$ vanishes generally
	 	 	and $\chi(\Gr_{\textbf v}(M_\lambda))=\chi(\Gr_{\textbf v}(M_E))$.
	 	\end{monlem}
	 	\begin{proof}
	 		If $\textbf v = \delta$, then $\Gr_{\delta}(M_\lambda)=\ens{M_\lambda}$ and $\Gr_{\delta}(M_E)=\ens{M_E}$ so 
	 		$$\chi(\Gr_{\delta}(M_E))=\Gr_{\delta}(M_\lambda)=1.$$
	 		Assume now that $\textbf v \neq \delta$ is such that $\Gr_{\textbf v}(M_\lambda) \neq \emptyset$ and fix some proper submodule $V \subset M_\lambda$ of dimension vector $\textbf v$. Then $V$ is preprojective and $M_\lambda/V$ is a direct sum of preinjective or regular modules. It follows that $\Ext^1_{kQ}(V,M_\lambda/V)=0$ and thus $\Ext^1_{kQ}(\textbf v, \delta-\textbf v)=0$. According to theorem \ref{theorem:Schofield} every representation of dimension $\delta$ contains a sub-representation of dimension $\textbf v$. In particular, $\Gr_{\textbf v}(M_E) \neq \emptyset$. Now, lemma \ref{lem:caracun} implies that $\chi(\Gr_{\delta}(M_E))=\chi(\Gr_{\delta}(M_\lambda))=1$ and the lemma is proved.
	 	\end{proof}

	 	\begin{maprop}\label{prop:grassmannians}
	 		Fix $Q$ a quiver of affine type $\Aaffine$, $E$ a quasi-simple module in an exceptional tube, $\lambda \in \P^1_0$ and $V \subset M_E$ a proper submodule of dimension vector $\textbf v$. The following holds:
	 		\begin{enumerate}
				\item If $E \not \subset V$, then $\Gr_{\textbf v}(M_\lambda) \neq \emptyset$,
				\item If $E \subset V \subset \regrad M_E$, then $\Gr_{\textbf v}(M_\lambda)=\emptyset$,
				\item If $E \subset V \not \subset \regrad M_E$, then $\Gr_{\textbf v}(M_\lambda) \neq \emptyset$.
	 		\end{enumerate}
	 	\end{maprop}
		\begin{proof}
			First we notice that if $V \subset M_E$, $V$ does not have any preinjective summand. We thus write $V=V_P \oplus V_R$ with $V_P$ preprojective and $V_R$ regular. Note also that the uniseriality of the regular components force $V_R$ to be indecomposable and to contain $E$ if it is not zero.
			
			If $E \not \subset V$, then $V_R=0$ and $V=V_P$ is preprojective. $M_E/V$ being a direct sum of regular and preinjective summands, it follows that $\Ext^1_{kQ}(V,M_E/V)=0$ and thus theorem \ref{theorem:Schofield} implies that $\Gr_{\textbf v}(M_\lambda) \neq \emptyset$. This proves the first point.
				
			If $E \subset V$, then $V$ is not preprojective. Now $M_E/V$ is a direct sum of regular and preinjective summands. It follows that 
			$$\Ext^1_{kQ}(V,M_E/V) \simeq \Hom_{kQ}(M_E/V, \tau V)=\Hom_{kQ}(M_E/V, \tau V_R).$$
			As $E \subset V \subset \regrad M_E$, we have 
			$$M_E/\regrad M_E \simeq (M_E/V) / (\regrad M_E/V)$$
			and thus $\tau E \simeq M_E/\regrad M_E$ is a quotient of $M_E/V$.
			In particular the space $\Hom_{kQ}(M_E/V,\tau E)$ is not reduced to zero. As $\tau E \subset \tau V_R$, we get
			$$\Hom_{kQ}(M_E/V,\tau V_R) \neq 0$$
			and thus 
			$$\Ext^1_{kQ}(V,M_E/V) \neq 0.$$
			As $\delta_i=1$ for every $i \in Q_0$, it follows that $\Hom_{kQ}(V,M_E/V)=0$ and thus 
			$$\<\textbf v, \delta- \textbf v\><0.$$
			In particular, for any two representations $X$ and $Y$ of respective dimensions $\textbf v$ and $\delta - \textbf v$, we have $\Ext^1_{kQ}(X,Y) \neq 0$ and thus $\Ext^1_{kQ}(\textbf v,\delta - \textbf v)$ does not vanish generally. It thus follows from lemma \ref{lem:caracpreproj} that $\Gr_{\textbf v}(M_\lambda)=\emptyset$. This proves the second point.
			
			Assume now that $E \subset V \not \subset \regrad M_E$. As $E \subset V$, $V_R \neq 0$. Now if $V_P =0$, $V=V_R$ is a proper regular submodule of $M_E$, it is thus contained in its quasi-radical $\regrad M_E$, this is a contradiction and thus $V=V_P \oplus V_R$ is a non-trivial decomposition. We now prove that $\Ext^1_{kQ}(V_R,V_P) \neq 0$. We denote by 
			$$\pi: M_E \fl M_E/\regrad M_E \simeq \tau E$$
			the canonical projection. 
			As $V \not \subset \regrad M_E$, $\pi(V) \neq 0$ and thus $\pi_{|V}: V \fl \tau E$ is a non-zero $V \fl \tau E$, thus
			$$\Hom_{kQ}(V, \tau E) \simeq  \Hom_{kQ}(V_P, \tau E) \oplus \Hom_{kQ}(V_R, \tau E) \neq 0.$$

			As $E \subset V_R \subset \regrad M_E$, the combinatorics of tubes proves that
			$$\Hom_{kQ}(V_R, \tau E) = 0$$
			and thus $\Hom_{kQ}(V_P,\tau E) \neq 0$. As $\tau E \subset \tau V_R$, it follows that  $\Hom_{kQ}(V_P, \tau V_R) \neq 0$ and thus $\Ext^1_{kQ}(V_R,V_P) \neq 0$.
			
			Fix now a non-split short exact sequence 
			\begin{equation}\label{sequenceX}
			 0 \fl V_P \xrightarrow{\iota} X \xrightarrow{p} V_R \fl 0,
			\end{equation}
			we prove that $X$ is preprojective. Applying the functor $\Hom_{kQ}(-,M_\lambda)$ to the sequence, we get a long exact sequence
			\begin{align*}
				0 	& \fl \Hom_{kQ}(V_R,M_\lambda) \fl \Hom_{kQ}(X,M_\lambda) \fl \Hom_{kQ}(V_P,M_\lambda) \\
					& \fl \Ext^1_{kQ}(V_R,M_\lambda) \fl \Ext^1_{kQ}(X,M_\lambda) \fl \Ext^1_{kQ}(V_P, M_\lambda) \fl 0
			\end{align*}
			As there are neither morphisms nor extensions between the tubes, we get an isomorphism
			$$\Hom_{kQ}(X,M_\lambda) \simeq \Hom_{kQ}(V_P,M_\lambda).$$
			More precisely, every morphism $V_P \fl M_\lambda$ factorises in $V_P \xrightarrow{\iota} X \fl M_\lambda$. Decompose $X=X_P \oplus X_R \oplus X_I$ into preprojective, regular and preinjective direct summands. As $\P^1_0$ is infinite, we can assume that $\lambda$ is chosen such that $V_R \not \in \mathcal T_\lambda$ and thus $\Hom_{kQ}(X_R,M_\lambda)=0$. Consider the inclusion $f: V_P \fl M_\lambda$. Then there is a factorization
			$$V_P \xrightarrow{[i_1,i_2,i_3]^t} X_P \oplus X_R \oplus X_I \xrightarrow{[g_1,0,0]} M_\lambda$$
			and thus $f=g_1i_1$ factorizes through $i_1$. In particular, $i_1$ is injective and thus $V_P \subset X_P$.  Also, we have $\ddim X = \ddim V \leq \delta$ so $\dim X(i) \leq 1$ for all $i$. It follows that $\Hom_{kQ}(V_P,X_R)=\Hom_{kQ}(V_P,X_I)=0$. Thus $i_2=i_3=0$ and then the sequence (\ref{sequenceX}) gives 
			$$V_R \simeq X/V_P \simeq X_P/V_P \oplus X_R \oplus X_I.$$
			Now $V_R$ is regular so $X_I=0$. It follows that $\d_{X_P}=\d_X=\d_V=\d_{V_P}$ and thus $X_P/V_P$ is a regular module. By uniseriality, $V_R$ is indecomposable and thus either $X_P/V_P$ is zero, or $X_R$ is zero. If $X_P/V_P$ is zero, the sequence (\ref{sequenceX}) splits, which is a contradiction. It follows that $X_R=0$ and thus that $X=X_P$ is preprojective. 
			
			As $M_E/V$ is a direct sum of regular and preinjective summands, we thus have $$\Ext^1_{kQ}(X,M_E/V)=0$$
			and thus $\Ext^1_{kQ}(\textbf v, \delta-\textbf v)$ vanishes generally. By theorem \ref{theorem:Schofield}, we thus have $\Gr_{\textbf v}(M_\lambda) \neq \emptyset$ and this proves the last point.
		\end{proof}
		
		\begin{defi}
		 	We say that a quiver $Q$ of affine type satisfies the \emph{difference property}\index{difference property} if for every quasi-simple module $E$ in an exceptional tube, the following identity holds
		 	\begin{equation}\label{eq:differencedelta}
		 	 	X_{M_E}=X_{M_\lambda}+X_{\regrad M_E/E}
		 	\end{equation}
		 	for any $\lambda \in \P^1_0(Q)$.
		\end{defi}

		\begin{theorem}\label{theorem:differencedelta}
			A quiver $Q$ of affine type $\Aaffine$ satisfies the difference property.
		\end{theorem}
		\begin{proof}
			For any dimension vector $\textbf v$, proposition \ref{prop:grassmannians} and lemma \ref{lem:caracun} imply that 
			$$\Gr_{\textbf v}(M_\lambda) \neq \emptyset \Leftrightarrow \Gr_{\textbf v}^E(\regrad M_E) = \emptyset$$
			where $\Gr_{\textbf v}^E(\regrad M_E)$ denotes the modules $M \in \Gr_{\textbf v}(\regrad M_E)$ such that $E \subset M$. Moreover, if $\Gr_{\textbf v}(M_E) \neq \emptyset$ and $\Gr_{\textbf v}^E(\regrad M_E) = \emptyset$, we have $\chi(\Gr_{\textbf v}(M_\lambda))=\chi(\Gr_{\textbf v}(M_E))=1$ and if $\Gr_{\textbf v}^E(\regrad M_E) \neq \emptyset$, we have $\chi(\Gr_{\textbf v}^E(\regrad M_E))=\chi(\Gr_{\textbf v}(M_E))=1$.
			
			We write $\textbf e=\ddim E$, we have 
			\begin{align*}
			X_{M_E}	
				&=\sum_{\textbf v} \chi(\Gr_{\textbf v}(M_E)) \prod_i u_i^{-\<\textbf v, \alpha_i\>-\<\alpha_i, \delta-\textbf v\>}\\
				&=\sum_{\Gr_{\textbf v}(M_\lambda) \neq \emptyset} \chi(\Gr_{\textbf v}(M_E)) \prod_i u_i^{-\<\textbf v, \alpha_i\>-\<\alpha_i, \delta-\textbf v\>}\\
				&+\sum_{\Gr_{\textbf v}(M_\lambda) =\emptyset} \chi(\Gr_{\textbf v}(M_E)) \prod_i u_i^{-\<\textbf v, \alpha_i\>-\<\alpha_i, \delta-\textbf v\>}\\
				&=\sum_{\textbf v} \chi(\Gr_{\textbf v}(M_\lambda)) \prod_i u_i^{-\<\textbf v, \alpha_i\>-\<\alpha_i, \delta-\textbf v\>}\\
				&+\sum_{\Gr_{\textbf v}(M_\lambda) =\emptyset} \chi(\Gr_{\textbf v}(M_E)) \prod_i u_i^{-\<\textbf v, \alpha_i\>-\<\alpha_i, \delta-\textbf v\>}\\
				&=X_{M_\lambda}+\sum_{\Gr_{\textbf v}(M_\lambda) =\emptyset} \chi(\Gr_{\textbf v}(M_E)) \prod_i u_i^{-\<\textbf v, \alpha_i\>-\<\alpha_i, \delta-\textbf v\>}\\
			\end{align*}
			
			Now, we have 
			$$\sum_{\Gr_{\textbf v}(M_\lambda) =\emptyset} \chi(\Gr_{\textbf v}(M_E)) \prod_i u_i^{-\<\textbf v, \alpha_i\>-\<\alpha_i, \delta-\textbf v\>}=\sum_{\textbf v} \chi(\Gr_{\textbf v}^E(\regrad M_E))\prod_i u_i^{-\<\textbf v, \alpha_i\>-\<\alpha_i, \delta-\textbf v\>}$$
			As 
			$$\left\{\begin{array}{rcl}
				\Gr_{\textbf v}^E(\regrad M_E) & \fl & \Gr_{\textbf v - \textbf e}(\regrad M_E/E)\\
				U & \mapsto & U/E
			\end{array}\right.$$
			is an isomorphism, it follows that 
			$$\sum_{\Gr_{\textbf v}(M_\lambda) =\emptyset} \chi(\Gr_{\textbf v}(M_E)) \prod_i u_i^{-\<\textbf v, \alpha_i\>-\<\alpha_i, \delta-\textbf v\>}=\sum_{\textbf v} \chi(\Gr_{\textbf v-\textbf e}(\regrad M_E/E))\prod_i u_i^{-\<\textbf v, \alpha_i\>-\<\alpha_i, \delta-\textbf v\>}$$
			But if we write $\textbf n=\ddim \regrad (M_E/E)=\delta-c(e)-e$ where $c$ is the Coxeter transformation, we compute
			\begin{align*}
				-\<(\textbf v-\textbf e), \alpha_i\>-\<\alpha_i, \textbf n-(\textbf v-\textbf e)\>
					&= -\<\textbf v-\textbf e, \alpha_i\>-\<\alpha_i, \delta-\textbf e-c(\textbf e)-\textbf v+\textbf e\>\\
					&= -\<\textbf v,\alpha_i\>-\<\alpha_i, \delta-\textbf v\>+\<\textbf e, \alpha_i\>+\<\alpha_i, c(\textbf e)\>\\
					&= -\<\textbf v,\alpha_i\>-\<\alpha_i, \delta-\textbf v\>
			\end{align*}
			and thus 
			\begin{align*}
			X_{\regrad M_E/E}
			&=\sum_{\textbf v} \chi(\Gr_{\textbf v-\textbf e}(\regrad M_E/E))\prod_i u_i^{-\<(\textbf v-\textbf e), \alpha_i\>-\<\alpha_i, \textbf n-(\textbf v-\textbf e)\>}\\
			&=\sum_{\textbf v} \chi(\Gr_{\textbf v-\textbf e}(\regrad M_E/E))\prod_i u_i^{-\<\textbf v, \alpha_i\>-\<\alpha_i, \delta-\textbf v\>}\\
			&=\sum_{\Gr_{\textbf v}(M_\lambda) =\emptyset} \chi(\Gr_{\textbf v}(M_E)) \prod_i u_i^{-\<\textbf v, \alpha_i\>-\<\alpha_i, \delta-\textbf v\>}\\
			\end{align*}
			This implies that 
			$$X_{M_E}=X_{M_\lambda}+X_{\regrad M_E/E}$$
		\end{proof}
		
		\begin{corol}
			Let $Q$ be a quiver of affine type satisfying the difference property. If $E$ is a quasi-simple module contained in a tube of rank $p>2$. Then $X_{M_E} \neq X_{M_F}$ for any quasi-simple module $F \neq E$ in an exceptional tube. 
		\end{corol}
		\begin{proof}
			As $E$ is contained in an exceptional tube of rank $p>2$, it follows that $\regrad M_E$ is of quasi-length $l>1$ and thus $\regrad M_E/E$ is a non-zero indecomposable rigid module. Fix $F$ a quasi-simple not isomorphic to $E$. Then $\regrad M_F/F$ is an indecomposable rigid module which is not isomorphic to $\regrad M_E/E$. It follows that $\ddim \regrad M_F/F \neq \ddim \regrad M_E/E$ and thus the denominators theorem implies that $X_{\regrad M_F/F} \neq X_{\regrad M_E/E}$. Thus, 
			$$X_{M_E}=X_{M_\lambda}+X_{\regrad M_E/E} \neq X_{M_\lambda}+X_{\regrad M_F/F}=X_{M_F}.$$
		\end{proof}

		\begin{rmq}
			It is conjectured that every quiver of affine type satisfies the difference property.
		\end{rmq}
		
		\begin{monexmp}\label{exmp:differencedeltaA21}
			The most simple case of affine quivers after the Kronecker quiver are the quivers of affine type $\affA{2}{1}$. They are all isomorphic to 
			$$\xymatrix{
					&& 2 \ar[rd]\\
				Q: &1 \ar[rr] \ar[ru] &&3
			}$$
			The Auslander-Reiten quiver of $kQ$-mod contains only one exceptional tube of rank 2. The quasi-simple modules of this tube are given by 
			$$\xymatrix{
					&& 0 \ar[rd]\\
				E_0:& k \ar[rr]^1 \ar[ru]^0 &&k
			}$$
			and 
			$$\xymatrix{
				&& k \ar[rd]^0\\
				E_1\simeq S_2: &0 \ar[rr] \ar[ru] &&0
			}$$
			satisfying $\tau E_1 \simeq E_0$ and $\tau E_0 \simeq E_1$.
			The corresponding indecomposable modules of dimension $\delta$ are:
			$$\xymatrix{
					&& k \ar[rd]^1\\
				M_{E_0}:& k \ar[rr]^1 \ar[ru]^0 &&k
			}$$
			and 
			$$\xymatrix{
				&& k \ar[rd]^0\\
				M_{E_1}: &k \ar[rr]^1 \ar[ru]^1 &&k.
			}$$
			
			For any $0 \neq \lambda \in k$, we set
			$$\xymatrix{
				&& k \ar[rd]^\lambda\\
				M_{\lambda}: &k \ar[rr]^1 \ar[ru]^1 &&k
			}$$
			and we set 
			$$\xymatrix{
				&& k \ar[rd]^1\\
				M_{\infty}: &k \ar[rr]^0 \ar[ru]^1 &&k
			}$$
			We identify $\P^1$ and $k \sqcup \ens \infty$. Then the $M_\lambda$ for $\lambda \in \P^1 \setminus \ens 0$ are quasi-simple representations in homogeneous tubes and $\P^1_0(Q)=\P^1 \setminus \ens 0$.
			
			The following table sums up the situation where $\lambda \in \P^1_0(Q)$:
			$$\begin{array}{|c|c|c|c|c|c|c|c|c|c|}
			\hline
			\textbf e 
				& [000] & [001] & [010] & [100] & [011] & [101] & [110] & [111] \\
			\hline
			\Gr_{\textbf e}(M_\lambda)
				& 0 & S_3 & \emptyset & \emptyset & P_2 & \emptyset & \emptyset & M_\lambda \\
				\hline
			\chi(\Gr_{\textbf e}(M_\lambda))
				& 1 & 1 & 0 & 0 & 1 & 0 & 0 & 1 \\
			\hline
			\Gr_{\textbf e}(M_{E_0})
				& 0 & S_3 & \emptyset & \emptyset & P_2 & E_0 & \emptyset & M_{E_0} \\
			\hline 
			\chi(\Gr_{\textbf e}(M_{E_0}))
				& 1 & 1 & 0 & 0 & 1 & 1 & 0 & 1 \\
			\hline
			\Gr_{\textbf e}(M_{E_1})
				& 0 & S_3 & E_1 & \emptyset & S_2 \oplus S_3 & \emptyset & E_1 & M_{E_1} \\
			\hline 
			\chi(\Gr_{\textbf e}(M_{E_1}))
				& 1 & 1 & 1 & 0 & 1 & 0 & 0 & 1 \\
			\hline 
			\end{array}$$
			
			It is important to note that even if the grassmannians of dimension $[011]$ have the same Euler characteristic, they are not equal. More precisely, 
			$$\Gr_{[011]}(M_{E_1})=\ens{S_2 \oplus S_3}$$
			whereas $$\Gr_{[011]}(M_\lambda)=\Gr_{[011]}(M_{E_0})=\ens{P_2}.$$
			With the notations of proposition \ref{prop:grassmannians}, if $V=S_2 \oplus S_3$, we have $V_P=S_3$, $V_R=S_2$ and there is a non-trivial extension
			$$0 \fl S_3 \fl P_2 \fl S_2 \fl 0,$$
			illustrating the second point of proposition \ref{prop:grassmannians}.
			
			Now, if we compute the generalized variables corresponding to these modules, we find:
			$$X_{M_{E_0}}=X_{M_{E_1}}=\frac{u_2u_1^2+u_3u_1u_2+u_1+u_3+u_2u_3^2}{u_1u_2u_3}$$ and 
			$$X_{M_\lambda}=\frac{u_2u_1^2+u_1+u_3+u_2u_3^2}{u_1u_2u_3}=X_{M_{E_1}}-1.$$
			As $\regrad M_{E_i}=E_i$ for $i=1,2$, it follows that $\regrad M_{E_i}/E_i=0$, this illustrates theorem \ref{theorem:differencedelta}.
		\end{monexmp}
		
		\begin{rmq}
		 	Note that the tubes of rank 2 give nice examples of modules giving the same generalized variables but having different characteristics of grassmannians.
		\end{rmq}
		
		\begin{monexmp}\label{exmp:differencedeltaA31}
			Consider the quiver of affine type $\affA{3}{1}$:
			$$\xymatrix{
				&&2 \ar[r] & 3\ar[rd]\\
			Q:& 1 \ar[ru] \ar[rrr] &&&4
			}$$
			
			For any $\lambda \in k$, we set
			$$\xymatrix{
				&&k \ar[r]^1 & k\ar[rd]^{\lambda}\\
			M_\lambda:& k \ar[ru]^1 \ar[rrr]^1 &&&k
			}$$
			and 
			$$\xymatrix{
				&&k \ar[r]^1 & k\ar[rd]^{1}\\
			M_\infty:& k \ar[ru]^1 \ar[rrr]^0 &&&k
			}$$
			We identify $\P^1$ and $k \sqcup \ens \infty$. Then, for any $\lambda \in \P^1(k)$, $M_\lambda$ is an indecomposable representation of dimension $\delta$ and moreover, if $\lambda \neq 0$, then $M_\lambda$ belongs to an homogeneous tube and $\P^1_0(Q)=\P^1 \setminus \ens 0$.
			
			The representation
			$$\xymatrix{
				&&0 \ar[r] & k\ar[rd]^{0}\\
			E_0:& 0 \ar[ru] \ar[rrr] &&&0
			}$$
			is quasi-simple and, is given by 
			$$\xymatrix{
				&&k \ar[r]^0 & 0\ar[rd]\\
			E_1=\tau^{-1} E_0:& 0 \ar[ru] \ar[rrr] &&&0
			}$$
			$$\xymatrix{
				&&0 \ar[r] & 0\ar[rd]\\
			E_2=\tau^{-2} E_0:& k \ar[ru]^0 \ar[rrr]^1 &&&k
			}$$
			and $\tau^{-3}E_0=E_0$. Thus $E_0$ is contained in an exceptional tube of rank 3 and moreover $M_0=M_{E_0}$.
			
			The tube is thus 
			\begin{figure}[H]
 			\begin{picture}(300,170)(-50,0)
 				\setlength{\unitlength}{1mm}
 				
 				\put (-1,9){$\bullet$}
 				\put (-5,5){$E_{0}$}
 				\put (19,9){$\bullet$}
 				\put (7,5){$E_1$}
 				\put (39,9){$\bullet$}
 				\put (37,5){$E_{2}$}
 				\put (59,9){$\bullet$}
 				\put (55,5){$E_{0}$}
 				
 				\put (19,29){$\bullet$}
 				\put (22,29){$M_0$}
 				\put (20,30){\circle{5}}
 				
 				\put (20,10){\circle{5}}
 								
 				\multiput(0,10)(0,20){2}{\multiput(0,0)(20,0){3}{\vector(1,1){10}}}
 				\multiput(0,30)(0,20){2}{\multiput(0,0)(20,0){3}{\vector(1,-1){10}}}
 				\multiput(10,20)(0,20){2}{\multiput(0,0)(20,0){3}{\vector(1,-1){10}}}
 				\multiput(10,20)(0,20){2}{\multiput(0,0)(20,0){3}{\vector(1,1){10}}}
 		
 				\put (0,10){\line(0,1){40}}
 				\put (60,10){\line(0,1){40}}
 						
 			\end{picture}
 			\caption{The exceptional tube of $\tilde A_{3,1}$}
 			\end{figure}
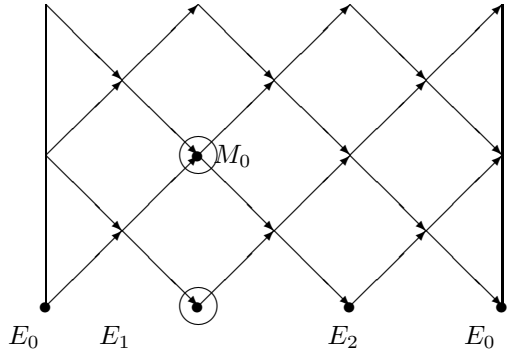
			
			Note that $\regrad M_0=E_0^{(2)}$ and $\regrad M_0/E_0=E_1$. We compute the corresponding generalized variable and we get 
			$$X_{E_1}=\frac{u_1+u_3}{u_2}$$

			We now compute the grassmannians for some $\lambda \in \P^1_0$. We only write dimension vectors for which one of the two considered grassmannians is non-empty:
			$$\begin{array}{|c|c|c|c|c|c|c|c|}
			\hline
			\textbf e 
				& [0000] & [0001] & [0010] & [0110] & [0011] & [0111] &[1111]\\
			\hline
			\Gr_{\textbf e}(M_\lambda)
				& 0 & S_4 & \emptyset & \emptyset & P_3 & P_2 & M_\lambda\\
			\hline
			\chi(\Gr_{\textbf e}(M_\lambda))
				& 1 & 1 & 0 & 0 & 1 & 1 & 1\\
			\hline
			\Gr_{\textbf e}(M_0)
				& 0 & S_4 & E_0 & E_0^{(2)} & E_0 \oplus S_4 & E_0^{(2)} \oplus S_4 & M_0\\
			\hline 
			\chi(\Gr_{\textbf e}(M_{E_0}))
				& 1 & 1 & 1 & 1 & 1 & 1 & 1\\
			\hline 
			\end{array}$$
			
			A direct computation gives:
			$$X_{M_0}=\frac{u_1u_4u_3^2+u_3u_1^2u_4+u_3u_2u_4^2+u_4u_3+u_3u_2u_1^2+u_1u_4+u_2u_1}{u_1u_2u_3u_4}$$
			and 
			$$X_{M_\lambda}=\frac{u_3u_2u_1^2+u_2u_1+u_1u_4+u_4u_3+u_3u_2u_4^2}{u_1u_2u_3u_4}.$$
			
			Computing the difference, we get
			$$X_{M_0}-X_{M_\lambda}=\frac{u_1+u_3}{u_2}=X_{E_1},$$
			this illustrates theorem \ref{theorem:differencedelta}.
			
			It is interesting to note that for $\textbf e=[0011]$, $V \in \Gr_{\textbf e}(M_0)$ is the representation $E_0 \oplus S_4$. Using the notations of proposition \ref{prop:grassmannians}, $V_R=E_0$, $V_P=S_4$ and $P_3 \in \Gr_{\textbf e}(M_\lambda)$ is indeed the central term of an extension 
			$$0 \fl S_4 \fl P_3 \fl E_0 \fl 0,$$
			illustrating proposition \ref{prop:grassmannians}.
		\end{monexmp}

		\begin{corol}
			Fix $Q$ a quiver of affine type satisfying the difference property, $E$ a quasi-simple module in an exceptional tube. Then $X_{M_E} \in \Z\mathcal B'(Q)$.
		\end{corol}
		\begin{proof}
			If we denote by $p$ the rank of the exceptional tube containing $E$, $\regrad M_E/E$ is an indecomposable regular module of quasi-length $p-2$. If $p=2$, then $\regrad M_E/E$ is zero and thus equality (\ref{eq:differencedelta}) implies that 
			$$X_{M_E}=X_{M_\lambda}+1 \in \Z\mathcal B'(Q).$$
			If $p>2$, it follows that $\ddim E$ is a Schur root and that $X_{\regrad M_E/E}=X_{\ddim \regrad M_E/E} \in \mathcal B'(Q)$. Now, we know that $X_{M_\lambda}=X_{\delta} \in \mathcal B'(Q)$ so equality (\ref{eq:differencedelta}) implies that $X_{M_E} \in \Z\mathcal B'(Q)$.
		\end{proof}
		
		We now want to prove that $X_{M_E^{(n)}}$ is generated by $\mathcal B'(Q)$ over $\Z$ for any $n \geq 2$ and any quasi-simple $E$ in an exceptional tube.
		
		\begin{maprop}\label{prop:TubeinZB}
			Let $Q$ be a quiver of affine type satisfying the difference property and $\mathcal T$ be an exceptional tube in the AR quiver of $Q$. Then $X_M \in \Z \mathcal B'(Q)$ for any module $M$ in $\add \mathcal T$.
		\end{maprop}
		\begin{proof}
			We prove it by induction on the dimension vector of $M$. If $M$ is quasi-simple, then $\ddim M$ is a real Schur root and $X_M \in \mathcal B'(Q)$. Assume now that $M$ is not quasi-simple and that for any $N$ in $\mathcal T$ such that $\ddim N \leq \ddim M$, we have $X_N \in \Z\mathcal B'(Q)$. We denote by $p$ the rank of $\mathcal T$ and $E_0, \ldots, E_{p-1}$ the quasi-simple modules of $\mathcal T$ ordered such that $\tau E_i=E_{i-1}$ for every $i \in \Z/p\Z$.
			
			Assume first that $M=M_E^{(n)}$ for some quasi-simple module $E$ in $\mathcal T$. From equality (\ref{eq:differencedelta}), we have
			$$X_{M_E}=X_{M_\lambda}+X_{\regrad M_E/E}$$
			Thus, by induction,
			$$X_{M_E}^n - X_{M_\lambda}^n \in \Z \mathcal B'(Q)$$
			but $X_{M_\lambda}^n=X_{n \delta} \in \mathcal B'(Q)$ so $X_{M_E}^n \in \Z\mathcal B'(Q)$. 
			
			Now, we know from \cite{Dupont:stabletubes} that 
			$$X_{M_E^{(n)}}=P_{np,p}(X_{E_0}, \ldots, X_{E_{p-1}})$$
			where $P_{np,p}$ denotes the $np$-th generalized Chebyshev polynomial of rank $p$. In particular
			$$X_{M_E^{(n)}}=X_{E_0 \oplus \cdots \oplus E_{p-1}}^n +\sum_{Y} s_{Y,n} X_Y$$
			where $s_{Y,n}$ are integers and $Y$ runs over modules of the form $\bigoplus_{i=0}^{p-1}E_i^{\oplus n_i}$ with $\sum n_i \ddim E_i \lneqq n\delta$. In particular, 
			$$X_{M_E} \in X_{E_0 \oplus \cdots \oplus E_{p-1}} + \sum_{Y} s_{Y,1} X_Y$$
			and thus 
			$$X_{M_E}^n \in X_{E_0 \oplus \cdots \oplus E_{p-1}}^n + \sum_{Y} t_{Y} X_Y$$
			where the $t_Y$ are integers and $Y$ run over modules of the form $\bigoplus_{i=0}^{p-1}E_i^{\oplus n_i}$ with $\sum n_i \ddim E_i \lneqq n\delta$.
			It follows that 
			$$X_{M_E^{(n)}}=X_{M_E}^n+\sum_{Y} k_{Y} X_Y$$
			where the $k_Y$ are integers and $Y$ run over modules of the form $\bigoplus_{i=0}^{p-1}E_i^{\oplus n_i}$ with $\sum n_i \ddim E_i \lneqq n\delta$. As $X_{M_E}^n \in \Z \mathcal B'(Q)$ by the above discussion and all the $X_Y$ are also in $\Z \mathcal B'(Q)$ by induction, we get $X_{M_E^{(n)}} \in \Z \mathcal B'(Q)$ for every $n \geq 1$.
			
			Now assume that $M$ is indecomposable but not of the form $M_E^{(n)}$. Then $\ddim M$ is a real root. If $\ddim M \leq \delta$, then $\ddim M$ is a real Schur root and $X_M \in \mathcal B'(Q)$. We can thus assume that $\delta \leq \ddim M$. More precisely, $\ddim M=\delta^{\oplus n} \oplus \beta$ where $\beta$ is a real Schur root. We write $N$ the unique indecomposable of dimension vector $\beta$, it belongs to $\mathcal T$ and we have an extension
			$$0 \fl N \fl M \fl M_E^{(n)} \fl 0$$
			for some quasi-simple module $E$ in $\mathcal T$. 
			Note that 
			$$\dim \Ext^1_{kQ}(N,M_E^{(n)})=\dim \Hom_{kQ}(M_E^{(n)},\tau N)=1,$$
			so it follows from the multiplication formula (see also \cite{Dupont:stabletubes}), that
			$$X_{M_E^{(n)}}X_N=X_M + X_B$$
			where $B \simeq \ker f \oplus \coker f[-1]$ for some morphism $f: M_E^{(n)} \fl \tau N$. In particular $\ddim B \lneqq \ddim N + n \delta= \ddim M$ and then, by induction $X_B \in \Z\mathcal B'(Q)$. Now, the above discussion proves that 
			$$X_{M_E^{(n)}}=X_{M_\lambda}^n + \sum_Y r_Y X_Y$$
			where the $r_Y$ are integers and $Y$ runs over modules in $\add \mathcal T$ such that $\ddim Y \lneqq n \delta$. In particular, 
			\begin{align*}
			 	X_{M_E^{(n)}}X_N
			 		&=X_{M_\lambda}^nX_N + \sum_Y r_Y X_{Y \oplus N}\\
			 		& \in X_{n \delta}X_{\beta} + \Z \mathcal B'(Q)\\
			 		& \in X_{n \delta+\beta} + \Z \mathcal B'(Q)\\
			 		& \in \Z \mathcal B'(Q)
			\end{align*}
			and thus $X_M = X_{M_E^{(n)}}X_N-X_B \in \Z \mathcal B'(Q)$.
			
			Now assume that $M$ is decomposable. Then $M=M_1 \oplus M_2$ for some regular objects $M_1,M_2$. But according to proposition \ref{prop:dcpXM}, we can write $X_M$ as a linear combination of $X_Y$ where $Y=\bigoplus_{i=1}^n Y_i$ (indecomposable decomposition) is in $\add \mathcal T$ such that $\Ext^1_{\mathcal C_Q}(Y_i,Y_j)=0$ for $i \neq j$. We prove by induction on the dimension vector that $X_M$ can be written as a $\Z$-linear combination of $X_V$ where $V$ is of the form $M_\lambda^kX_R$ where $R \in \add \mathcal T$ is rigid, $\lambda \in \P^1_0$ and $k \geq 0$. In particular, each $X_V$ will belong to $\mathcal B'(Q)$ and thus $X_M \in \Z\mathcal B'(Q)$.

			Fix some integer $1 \leq i \leq n$. The indecomposable summand $Y_i$ is of the form $E_i^{(n_i)}$ for some integer $n_i \geq 0$ and some quasi-simple module $E_i$. Let denote by $p$ the rank of the tube $\mathcal T$ and write $n_i=k_i p+r_i$ with $0 \leq r_i <p$ the euclidean division of $n_i$ by $p$. By theorem \ref{theorem:multiplication}, we can write
			$$X_{Y_i}=X_{E_i^{(r_i)}}X_{M_{F_i}^{(k_i)}}+X_B$$
			with $\ddim B \lneq \ddim Y_i$ and $F_i$ a quasi-simple in $\mathcal T$. Now, thanks to the difference property $X_{M_{F_i}^{(k_i)}}$ can be written
			$$X_{M_{F_i}^{(k_i)}}=X_{M_{\lambda}}^{k_i}+\sum_{\ddim Z_i \lneq k_i \delta} r_{Z_i} X_{Z_i}$$
			where $Z_i \in \add(\mathcal T)$.
		
			Then
			$$X_Y=\prod_{i=1}^n X_{E_i}^{(r_i)}X_{M_\lambda}^{k_i}+\sum_{\ddim U \lneq \ddim Y} s_U X_U$$
			with $U \in \add(\mathcal T)$. By induction on the dimension vector, each $X_U$ can be written as the expected $\Z$-linear combination.
			
			Now
			$$\prod_{i=1}^n X_{E_i}^{(r_i)}X_{M_\lambda}^{k_i}=X_{M_\lambda}^{\sum_{i=1}^n k_i}X_{\bigoplus_{j=1}^n E_j^{(r_j)}}.$$
			If $k_j=0$ for every $j=1, \ldots, n$, then each $Y_j$ is rigid and thus $Y$ is rigid and the result holds. Otherwise, $\ddim \bigoplus_{j=1}^n E_j^{(r_j)} \lneq \ddim Y$ and thus, by induction,  $X_{\bigoplus_{j=1}^n E_j^{(r_j)}}$ can be written as a $\Z$-linear combination of $X_{M_\lambda}^k X_R$ with $R$ rigid in $\add \mathcal T$ and thus $\prod_{i=1}^n X_{E_i}^{(r_i)}X_{M_\lambda}^{k_i}$ can be written as a $\Z$-linear combination of $X_{M_\lambda}^{k+\sum_{j=1}^n k_j} X_R$, which proves the assertion.
		\end{proof}
		
		\begin{corol}\label{corol:indcp}
			Let $Q$ be a quiver of affine type satisfying the difference property. Then for any indecomposable object $M$ in $\mathcal C_Q$, we have $X_M \in \Q \mathcal B'(Q)$.
		\end{corol}
		
		\begin{monlem}\label{lem:modgeninZB}
		 	Let $Q$ be an affine quiver satisfying the difference property and $Y=\bigoplus_{i} Y_i$ be an object in $\mathcal C_Q$ such that $\Ext^1_{\mathcal C}(Y_i,Y_j)=0$ for any $i \neq j$. Then $X_Y \in \Q \mathcal B'(Q)$.
		\end{monlem}
		\begin{proof}
			If all the indecomposable direct summands of $Y$ are rigid, then $Y$ is also rigid and thus $X_Y \in \mathcal B'(Q)$. We thus assume that $Y$ is not rigid.
			
		 	Write
		 	$$Y = Y_P \oplus \bigoplus_{\lambda \in \P^1} Y_{\lambda}$$
		 	where $Y_P$ has all its direct summands in the transjective component and for any $\lambda$ in $\P^1$, $Y_{\lambda}$ has its direct summands in the tube $\mathcal T_{\lambda}$.
		 	
		 	If $Y_{\lambda} = 0$ for every $\lambda \in \P^1$, then $Y=Y_P$ is rigid, thus we can assume that $Y_{\lambda} \neq 0$ for some $\lambda \in \P^1$ and that $Y_\lambda$ has an indecomposable direct summand $M$ such that $\Ext^1_{kQ}(M,M) \neq 0$. 
		 	
		 	Assume that $Y_P \neq 0$ and fix an indecomposable direct summand $N$ of $Y_P$. Then $N=P_i[s]$ for some $i \in Q_0$ and $s \in \Z$. As $M$ is an indecomposable regular non-rigid object, if follows that $\ddim M$ is sincere and so is $\ddim \tau^{1-s}M \simeq M[1-s]$.
		 	\begin{align*}
		 		0 & \neq \dim \tau^{1-s}M(i) \\
		 		&= \dim \Hom_{kQ}(P_i,\tau^{1-s}M) \\
		 		&\leq \dim \Hom_{\mathcal C}(P_i,M[1-s])\\
		 		&=\dim \Hom_{\mathcal C}(P_i[s],M[1])\\
		 		&=\dim \Ext^1_{\mathcal C}(M,P_i[s]).
		 	\end{align*}
		 	so that $\Ext^1_{\mathcal C}(M,Y_P) \neq 0$ which is a contradiction. Thus, $Y_P=0$ and 
			$$Y=\bigoplus_{\lambda \in \P^1} Y_{\lambda}$$
			
			We now prove by induction on the dimension vector of $Y$ that $X_Y \in \Q\mathcal B'(Q)$. If $Y$ is a quasi-simple module, then either $Y=M_\mu$ for some $\mu \in \P^1_0$ or $Y$ is quasi-simple in an exceptional tube and is thus rigid. In both cases $X_Y \in \mathcal B'(Q)$.
			
			We now return to the general case $Y=\bigoplus_{\lambda \in \P^1} Y_{\lambda}$.	If $Y_{\lambda}=0$ for all $\lambda \in \P^1\setminus \P^1_0$, then it follows from corollary \ref{lem:Chebyshev} that $X_{Y_\lambda}$ is a polynomial in $X_\delta$. Thus, $X_Y$ is also a polynomial in $X_\delta$ and thus as $X_{n \delta}=X_\delta ^n$ for every $n \geq 1$, $X_Y$ is in $\Z \mathcal B'(Q)$.
		 
		 	From now on, we assume that $Y_\lambda \neq 0$ for some $\lambda \in \P^1\setminus \P^1_0$. 
			
		 	We claim that 
		 	$$X_{Y_\lambda}=X_\delta^{n_\lambda} X_{R_\lambda} + \sum_{\ddim Z \lneqq \ddim Y_\lambda} r_Z X_Z$$
		 	for some integer $n_\lambda \geq 0$, some rigid module $R_\lambda \in \add T_\lambda$ and some rational numbers $r_Z$ indexed by modules $Z \in \add \mathcal T_\lambda$. We prove it by induction on the dimension vector of $Y_\lambda$.
		 	
		 	If $Y_\lambda$ is quasi-simple, then $R_\lambda = Y_\lambda$ gives the result. Assume now that $Y_\lambda$ is any indecomposable module in $\mathcal T_\lambda$. Then it follows from the proof of \ref{prop:TubeinZB} that 
		 	$$X_Y=X_{M_E^{(n)}}X_N-X_B$$
		 	for some quasi-simple $E$ in $\mathcal T_\lambda$, some integer $n \geq 1$ and some regular module $B$ such that $\ddim B \lneqq \ddim Y$. In particular
		 	$$X_Y=X_{M_\lambda}^nX_N+ \sum r_z X_Z$$
		 	where $Z$ runs over objects in $\mathcal T_\lambda$ such that $\ddim Z \lneqq \ddim Y$. This proves the claim for $Y_\lambda$ indecomposable.
		 	
		 	Assume now that $Y_\lambda$ is decomposable. Then each indecomposable direct summand of $Y_\lambda$ satisfies the claim. Using inductively proposition \ref{prop:dcpXM}, we prove the claim for any $Y_\lambda \in \mathcal T_\lambda$.
		 	
		 	As for any $\lambda \in \P^1_0$, $X_{Y_\lambda}$ is a polynomial in $X_{\delta}$, it suffices to prove that 
		 	$$X_\delta^n \prod_{\lambda \in \P^1 \setminus \P^1_0} X_{Y_\lambda} \in \Z \mathcal B'(Q)$$
		 	for any $\lambda \in \P^1_0$ and for $n \leq n_0$ where $n_0 \delta =\ddim \bigoplus_{\mu \in \P^1_0} Y_\mu$. We have
		 	\begin{align*} 
		 	 	X_\delta^n \prod_{\lambda \in \P^1 \setminus \P^1_0} X_{Y_\lambda}
		 	 		&= X_\delta^n \prod_{\lambda \in \P^1 \setminus \P^1_0} \left(X_\delta^{n_\lambda} X_{R_\lambda} + \sum_{\ddim Z \lneqq \ddim Y_\lambda} r_Z X_Z\right)\\
		 	 		&=X_{\delta}^{n+\sum_\lambda n_\lambda} X_{\bigoplus_\lambda R_\lambda} + \sum_{\ddim Z \lneqq n\delta + \sum_\lambda Y_\lambda} r_Z X_Z
		 	\end{align*}
			where the $\lambda$ run over $\P^1 \setminus \P^1_0$. Note that for any $n \leq n_0$, we have
			$$\ddim Z \lneqq n\delta + \sum_\lambda Y_\lambda  \leq \ddim Z \lneqq n_0\delta + \sum_\lambda Y_\lambda = \ddim Y.$$ 
			
			By induction on proposition \ref{prop:dcpXM}, we can moreover assume that all the $X_Z$ satisfy the hypothesis of the lemma. Thus, by induction, it follows that all the $X_Z$ occurring in the last sum are in $\Q \mathcal B'(Q)$.
			
			As there are no extensions between the tubes, $\bigoplus_\lambda R_\lambda$ is rigid and thus it follows from proposition \ref{prop:explicitbase} (and also from lemma \ref{lem:multiplicativity}) that 
			$$X_{\delta}^{n+\sum_\lambda n_\lambda} X_{\bigoplus_\lambda R_\lambda} \in \Q \mathcal B'(Q).$$
		\end{proof}
		
		We are now able to prove that generic variables generate the cluster algebra in affine type $\Aaffine$.
		\begin{maprop}\label{prop:Qspan}
			Let $Q$ be a quiver of affine type satisfying the difference property. Then every element in $\mathcal A(Q)$ is in the $\Q$-vector space generated by $\mathcal B'(Q)$. Moreover, every element of $\mathcal B'(Q)$ belongs to $\mathcal A(Q)$.
		\end{maprop}
		\begin{proof}
			We first prove that $\mathcal B'(Q) \subset \mathcal A(Q)$. If $Q$ is the Kronecker quiver, this is proved in \cite{CZ}. Assume now that $Q$ is not the Kronecker quiver, it thus contains an exceptional tube.
			
			Fix $\textbf d$ a dimension vector. If $\textbf d \in \Z_{\leq 0}^{Q_0}$, $X_{\textbf d}$ is a product of $u_i$ and is thus in $\mathcal A(Q)$. As $X_{\textbf d}=X_{[\textbf d]_+}X_{[\textbf d]_-}$, it suffices to prove the result for $\textbf d \in \N^{Q_0}$.
			
			If $\textbf d$ is a real Schur root, then there exists an indecomposable rigid module $M \in \rep(Q,\textbf d)$ and thus theorem \ref{theorem:correspondanceCK2} implies $X_M$ is a cluster variable. 
			
			If $\textbf d=\delta$, then $X_{\textbf d}=X_{M_\lambda}$ for some $\lambda \in \P^1_0$. Fix $E$ a quasi-simple module in an exceptional tube $\mathcal T$, we know from \cite{Dupont:stabletubes} that $X_{M_E}$ is a polynomial in the $X_{E_i}$ where the $E_i$ are the quasi-simples of $\mathcal T$. In particular each $X_{E_i}$ being a cluster variable, it follows that $X_{M_E} \in \mathcal A(Q)$. As $\regrad M_E/E$ is indecomposable rigid, $X_{\regrad M_E/E}$ is also a cluster variable and thus 
			$$X_{\delta}=X_{M_\lambda}=X_{M_E}-X_{\regrad M_E/E} \in \mathcal A(Q).$$
			
			Now, we know that $X_{\textbf d} \in \mathcal A(Q)$ for any Schur root $\textbf d$. Fix now any element $\textbf d \in \N^{Q_0}$, and write $\textbf d=\textbf e_1 \oplus \cdots \oplus \textbf e_n$ its canonical decomposition, then every $\textbf e_i$ is a Schur root and thus $X_{\textbf e_i} \in \mathcal A(Q)$ for every $i=1, \ldots, n$. Now proposition \ref{prop:dcpcanonique} implies that 
			$$X_{\textbf d}=\prod_{i=1}^nX_{\textbf e_i} \in \mathcal A(Q),$$
			this proves the claim.
			
			Fix now an object $M$ in $\mathcal C_Q$. According to proposition \ref{prop:dcpXM}, $X_M$ can be written as a $\Q$-linear combination of $X_Y$ where $Y=\bigoplus_i Y_i$ is such that $\Ext^1_{\mathcal C}(Y_i,Y_j)=0$ for any $i \neq j$. It follows from lemma \ref{lem:modgeninZB} that $X_Y \in \Q\mathcal B'(Q)$ and thus $X_M \in \Q \mathcal B'(Q)$.
			
			Fix now a monomial $x \in \mathcal A(Q)$. Then $x=\prod_i x_i$ is a product of cluster variables $x_i$. According to theorem \ref{theorem:correspondanceCK2}, each $x_i$ can be written as a $X_{M_i}$ and thus $x=X_{\bigoplus_i M_i} \in \Q \mathcal B'(Q)$.
		\end{proof}
	\end{subsubsection}
\end{subsection}

\section{Reflection functors and Caldero-Chapoton map}\label{section:reflection}
	\begin{subsection}{Reflections and generic variables}\label{subsection:reflexions}
	We are now interested in proving that the elements in $\mathcal B'(Q)$ are linearly independent over $\Q$. For this purpose, we will follow the ideas of \cite{CK1} introducing a certain grading on the cluster algebra. This grading will depend on the orientation of the quiver $Q$ and thus we will need to understand the behaviour of generic variables under changes of orientations. In \cite{Zhu:equivalence}, the author investigated interactions between a certain extension of the BGP-reflection functors and the cluster combinatorics. This subsection is devoted to the study of the interaction between these extended reflection functors and the Caldero-Chapoton map. This is a generalization of the the works of Zhu.

	\begin{subsubsection}{Reflection functors and Caldero-Chapoton map}
		A \emph{sink}\index{sink} (resp. \emph{source}\index{source}) in $Q$ is a vertex in $Q_0$ such that there is no arrow starting (resp. ending) at $i$. Let $Q$ be a quiver and $i$ be a sink or a source in $Q$. We define the reflected quiver $\sigma_i(Q)$ by reversing all the arrows ending at $i$. An \emph{admissible sequence of sinks (resp. sources)} is a sequence $(i_1, \ldots, i_n)$ such that $i_1$ is a sink (resp. source) in $Q$ and $i_k$ is a sink (resp source) in $\sigma_{i_{k-1}}\cdots \sigma_{i_1}(Q)$ for any $k=2, \ldots, n$. A quiver $Q'$ is called \emph{reflection-equivalent}\index{reflection-equivalent} to $Q$ if there exists an admissible sequence of sinks or sources $(i_1, \ldots, i_n)$ such that $Q'=\sigma_{i_{n}}\cdots \sigma_{i_1}(Q)$. Note that this is an equivalence relation whose equivalence classes are called \emph{reflection classes}\index{reflection class}. In the following, we will only work with sinks but a straightforward adaptation gives the same results for sources.
		
		It is important to notice that mutations can be viewed as generalizations of reflections. Namely, if $i$ is a sink or a source in a quiver $Q$, then $\mu_i(Q)=\sigma_i(Q)$ where $\mu_i$ denotes the mutation in the direction $i$.
		
		If $Q$ is a quiver, we still denote by $\mathcal A(Q)$ the coefficient free cluster algebra with initial seed $(Q,\textbf u)$. If $Q'$ is a quiver mutation-equivalent to à $Q$, there exists some seed $(Q',\textbf v)$ in $\mathcal A(Q)$ and $\mathcal A(Q')$ will denote the cluster algebra with initial seed $(Q',\textbf v)$. There is a natural isomorphism of cluster algebras
		$$\Phi': \mathcal A(Q') \fl \mathcal A(Q)$$
		sending $v_i$ to its expansion in $\Z[u_i^{\pm 1}, i \in Q_0]$ for any $i \in Q_0$. Similarly, every $u_i$ can be written as a Laurent polynomial in $\textbf v$ and we write 
		$$\Phi: \mathcal A(Q) \fl \mathcal A(Q')$$
		the corresponding algebra isomorphism. Note in particular that $\Phi'$ are $\Phi$ are inverse isomorphisms. These isomorphisms will be referred to as the \emph{canonical cluster algebras isomorphisms}\index{canonical cluster algebras isomorphism}.

		From now on, we assume that $Q$ is acyclic and that $i$ is a sink in $Q_0$. We denote by $Q'=\sigma_i(Q)$ the reflected quiver. Let $\Sigma_i^+:\rep(Q) \fl \rep(Q')$ be the standard BGP-reflection functor and $R_i^+:\mathcal C_Q \fl \mathcal C_{Q'}$ be the extended BGP-reflection functor defined in \cite{Zhu:equivalence}. It is given on the objects of $\mathcal C_Q$ by:
		$$R_i^+:\left\{\begin{array}{rcll}
			X & \mapsto & \Sigma_i^+(X) & \textrm{ if }X \not \simeq S_i \textrm{ is a module}\\
			S_i & \mapsto & P_i[1] \\
			P_j[1] & \mapsto & P_j[1] & \textrm{ if }j \neq i\\
			P_i[1] & \mapsto & S_i
		\end{array}\right.$$
		The following holds:
		\begin{maprop}[\cite{Zhu:equivalence}]\label{prop:Riequivalence}
			Let $Q$ be an acyclic quiver and $i$ be a sink in $\mathcal C_Q$. Then:
			\begin{enumerate}
				\item $\Sigma_i^+$ induces a triangle equivalence $D^b(kQ) \fl D^b(kQ')$ commuting with the shift $[1]$ and the AR-translation $\tau$.
				\item $R_i^+$ induces a triangle equivalence $\mathcal C_Q \fl \mathcal C_{Q'}$.
			\end{enumerate}
		\end{maprop}

		Set $T_i=\tau^{-1}P_i$, $T_j=P_j$ for $j \neq i$ and $T=\bigoplus_{k=1}^n T_k$ the APR-tilting module associated to the sink $i$ and denote by the functor $F=\Hom_{\mathcal C_Q}(T,-)$. It is known that $F$ induces a triangle equivalence $\mathcal C_Q/\add T[1] \fl \modg k\sigma_i Q$. Moreover, $F$ and $R_i^+$ coincide on the objects of $\mathcal C_Q/\add T$.
		
		We denote by $X^Q_?$ (resp. by $X^{\sigma_i Q}_?$) the Caldero-Chapoton map associated to $Q$ (resp. to $\sigma_i Q$). We denote by $X^T_?$ the Palu's cluster character on $\mathcal C_Q$ associated to the cluster-tilting object $T$ introduced in \cite{Palu}. It is defined on indecomposable objects of $\mathcal C_Q$ by
		$$X_M^T=\left\{\begin{array}{lll}
			\sum_{\textbf e} \chi(\Gr_{\textbf e}(FM)) \prod_j v_j^{\<S_j, \textbf e\>_a-\<S_j,FM\>} & \textrm{ if }M \in \mathcal C_Q/\add T[1]\\
			v_j & \textrm{ if } M \simeq T_j[1] \textrm{ for any } j \in Q_0\\
		\end{array}\right.$$
		where $\<-,-\>_a$ is the symmetrized Euler form defined by
		$$\<M,N\>_a=\<M,N\>-\<N,M\>$$ for any two $k\sigma_i Q$-modules $M$ and $N$. 
		In our case, as $\sigma_i Q$ is acyclic, the Euler form is thus well defined on the Grothendieck group of $k\sigma_i Q$-mod, thus 
		$$\<S_j, \textbf e\>_a=\<S_j,\textbf e\>-\<\textbf e,S_j\>$$
		for any dimension vector $\textbf e$ and any $j \in \sigma_i Q_0$. It satisfies
		$$X^T_{M\oplus N}=X^T_MX^T_N$$
		for any two objects $M,N$ in $\mathcal C_Q$.
		The following lemma gives the link between the cluster character and the Caldero-Chapoton map in the particular case where $T$ is an APR-tilting module.
		
		\begin{monlem}\label{lem:Finitial}
			Let $Q$ be an acyclic quiver, $i$ a sink in $Q$ and $T$ be the APR-tilting $kQ$-module associated to $i$. Then for any object $M$ in $\mathcal C_Q$, we have
			$$X^T_M=X^{\sigma_i Q}_{R_i^+M}$$
		\end{monlem}
		\begin{proof}
			As $X^T_{M \oplus N}=X^T_MX^T_N$ and $X^{\sigma_i Q}_{M \oplus N}=X^{\sigma_i Q}_{M}X^{\sigma_i Q}_{N}$, it suffices to prove the lemma for $M$ indecomposable. If $M=T_i[1]$, then $R_i^+M=P_i[1]$ and thus $X^{\sigma_i Q}_{R_i^+M}=X^{\sigma_i Q}_{P_i[1]}=v_i$ but by definition $X^T_M=X^T_{T_i[1]}=v_i$. Let's assume that $M$ is indecomposable and non-isomorphic to $T_i[1]$ for any $i \in Q_0$. As $R_i^+(M)=FM$, we compute
			\begin{align*}
				X^T_M
					&= \sum_{\textbf e} \chi(\Gr_{\textbf e}(FM)) \prod_i v_i^{\<S_i,\textbf e\>_a-\<S_i,FM\>}\\
					&= \sum_{\textbf e} \chi(\Gr_{\textbf e}(FM)) \prod_i v_i^{\<S_i,\textbf e\>-\<\textbf e,S_i\>-\<S_i,FM\>}\\
					&= \sum_{\textbf e} \chi(\Gr_{\textbf e}(FM)) \prod_i v_i^{-\<S_i,FM-\textbf e\>-\<\textbf e,S_i\>}\\
					&= X^{\sigma_i Q}_{FM}\\
					&= X^{\sigma_i Q}_{R_i^+M}\\
			\end{align*}
		\end{proof}
	\end{subsubsection}
	
	\begin{subsubsection}{Reflections and cluster variables}
		We keep the above notations. We denote by $\Phi_i: \mathcal A(Q) \fl \mathcal A(\sigma_iQ)$ the canonical isomorphism of cluster algebras and by $\Phi_i': \mathcal A(\sigma_iQ) \fl \mathcal A(Q)$ the inverse isomorphism. It satisfies
		$$\Phi_i(X^Q_{T_j[1]})=X^{\sigma_iQ}_{R_i^+T_j[1]}=X^{\sigma_iQ}_{P_j[1]}=v_j$$
		for every $j \in Q_0$. 
		As $T[1]$ is a cluster-tilting object in $\mathcal C_Q$, it follows from theorem \ref{theorem:correspondanceCK2} that $\ens{X^Q_{T_j[1]} \ , \ j \in Q_0}$ is a cluster in $\mathcal A(Q)$ and $\ens{\Phi_i(X^Q_{T_j[1]}) \ ,\ j \in Q_0}=\ens{X^{\sigma_iQ}_{P_j[1]} \ , \ j \in Q_0}$ is a cluster in $\mathcal A(\sigma_i Q)$. We have the commutative diagram:
		$$
		\xymatrix{
			\ens{T_j[1] \ , \ j \in Q_0} \ar[r]^{R_i^+} \ar[d]_{X^Q_?} & \ens{P_j[1] \ , \ j \in Q_0} \ar[d]^{X^{\sigma_i Q}_?} \\
			\ens{X_{T_j[1]} \ , \ j \in Q_0 } \ar[r]_{\Phi_i} & \ens{v_j \ , \ j \in Q_0}.
		}
		$$
		The following lemma is a consequence of the works of Zhu and of Palu. We give an independent proof for completeness.
		\begin{monlem}\label{lem:exceptionnel}
			Let $Q$ be an acyclic quiver, $i$ be a sink in $\mathcal C_Q$. Then 
			$$\Phi_i(X^Q_M)=X^{\sigma_i Q}_{R_i^+M}$$
			for every rigid object $M$ in $\mathcal C_Q$.
		\end{monlem}
		\begin{proof}
			Assume that $R$ is a cluster-tilting object such that for every direct summand $M$ of $R$, we have $\Phi_i(X^Q_M)=X^{\sigma_iQ}_{R_i^+M}$. Fix $R'$ a cluster-tilting object in $\mathcal C_Q$ next to $R$ in the tilting graph (see \cite{BMRRT} for terminology and results about cluster-tilting theory and exchange pairs). Then ,there exists an exchange pair $(U,U^*)$ such that $R=U \oplus \overline U$ and $R'=U^* \oplus \overline U$. We denote by 
			$$U \fl B \fl U^* \fl U[1] \textrm{ and } U^* \fl B' \fl U \fl U^*[1]$$
			the corresponding triangles. Thus, $B$ and $B'$ are in $\add \overline U$. According to theorem \ref{theorem:onedimmult}, we have 
			$$X^Q_UX^Q_{U^*}=X^Q_B+X^Q_{B'}$$
			and thus
			$$\Phi_i(X^Q_U)\Phi_i(X^Q_{U^*})=\Phi_i(X^Q_B)+\Phi_i(X^Q_{B'}).$$
			The induction gives
			$$\Phi_i(X^Q_{U^*})=\frac{X^Q_{R_i^+B}+X^Q_{R_i^+B'}}{X^Q_{R_i^+U}}.$$
			
			On the other hand $R_i^+$ is a triangle equivalence so $(R_i^+U,R_i^+U^*)$ is an exchange pair and the corresponding triangles are 
			$$R_i^+U \fl R_i^+B \fl R_i^+U^* \fl R_i^+U[1] \textrm{ and } R_i^+U^* \fl R_i^+B' \fl R_i^+U \fl R_i^+U^*[1]$$
			The one-dimensional multiplication formula in $\mathcal C_{\sigma_i Q}$ gives:
			$$X^{\sigma_i Q}_{R_i^+U^*}=\frac{X^{\sigma_i Q}_{R_i^+B}+X^{\sigma_iQ}_{R_i^+B'}}{X^{\sigma_iQ}_{R_i^+U}}$$
			and thus
			$$X^{\sigma_i Q}_{R_i^+U^*}=\Phi_i(X^Q_{U^*}).$$
			According to lemma \ref{lem:Finitial}, the proposition holds for the cluster-tilting object $T[1]$ and as the tilting graph is connected (see \cite{BMRRT} and \cite{HU}), it follows that for every direct summand $M$ of a tilting object, we have $\Phi_i(X^Q_M)=X^{\sigma_i Q}_{R_i^+M}$. As every indecomposable rigid object $M$ can be completed into a cluster-tilting object, it follows that 
			$$\Phi_i(X^Q_M)=X^{\sigma_i Q}_{R_i^+M}$$
			for every rigid object $M$ in $\mathcal C_Q$. This proves the lemma.
		\end{proof}
	\end{subsubsection}
	
	\begin{subsubsection}{Reflections for affine quivers}
		We now return to the case when $Q$ is an affine quiver.
		\begin{monlem}\label{lem:Ftubes}
			Let $Q$ be an affine quiver and $i \in Q_0$ be a sink. The following hold:
			\begin{enumerate}
				\item A indecomposable $kQ$-module $M$ is a regular module if and only if $\Sigma_i^+(M)$ is a regular $k\sigma_i Q$-module.
				\item Let $M$ be an indecomposable regular $kQ$-module. Then $\Sigma_i^+$ induces an equivalence of categories from the tube containing $M$ in $kQ$-mod to the tube containing $\Sigma_i^+M$ in $k\sigma_i Q$-mod.
			\end{enumerate}
		\end{monlem}
		\begin{proof}
			We recall that $\Sigma_i^+$ induces an equivalence $kQ\textrm{-mod}/S_i \fl k\sigma_i Q\textrm{-mod}/S_i$. As the simple $kQ$-module $S_i$ is projective and the simple $k\sigma_i Q$-module $S_i$ is an injective $kQ$-module. Denote by $\pi: kQ\textrm{-mod} \fl kQ\textrm{-mod}/S_i$ the canonical functor. Then for any regular component $\mathcal T$, the restriction of the functor $\pi$ to $\mathcal T$ is isomorphic to the restriction of the identity to $\mathcal T$. If $M$ is a regular module, then it belongs to a tube $\mathcal T$ and $\Sigma_i^+M \simeq \Sigma_i^+(\pi(M))$ belongs to $\Sigma_i^+(\pi(\mathcal T))$. But $\Sigma_i^+(\pi(\mathcal T))\simeq \pi(\mathcal T) \simeq \mathcal T$ and then $\Sigma_i^+(M)$ belongs to a tube in $k\sigma_iQ\textrm{-mod}/S_i$, it is thus a regular module and moreover $\Sigma_i^+$ induces an equivalence between the tubes containing $M$ and $\Sigma_i^+(M)$.
		\end{proof}

		\begin{corol}\label{corol:Flongueur}
			Let $Q$ be an affine quiver, $i$ be a sink in $Q_0$ and $M$ be an indecomposable regular $kQ$-module with quasi-composition series 
			$$0=M_0 \subset M_1 \subset \cdots \subset M_r=M,$$
			then $R_i^+M$ is a regular $k\sigma_iQ$-module with regular composition series
			$$0=R_i^+M_0 \subset R_i^+M_1 \subset \cdots \subset R_i^+M_r=R_i^+M.$$
			In particular $R_i^+$ send quasi-socles to quasi-socles, quasi-radicals to quasi-radicals and preserves quasi-lengths.
		\end{corol}
		\begin{proof}
			As $i$ is a sink, $S_i$ is a projective module and then for any regular module $R_i^+(M) \simeq \Sigma_i^+(M)$. The result follows then directly from lemma \ref{lem:Ftubes}. 
		\end{proof}

		Now we extend lemma \ref{lem:exceptionnel} to any object $M$ in $\mathcal C_Q$. 
		\begin{maprop}\label{prop:reflexionaffine}
			Let $Q$ be a quiver of affine type with at least three vertices satisfying the difference property. Let $i$ be a sink in $Q_0$ such that $\sigma_i Q$ satisfies the difference property. Denote by $\Phi_i:\mathcal A(Q) \fl \mathcal A(\sigma_iQ)$ the canonical isomorphism and by $R_i^+:\mathcal C_Q \fl \mathcal C_{\sigma_iQ}$ the extended BGP functor. Then for any object $M$ in $\mathcal C_Q$, we have $\Phi_i(X^Q_M)=X^{\sigma_i Q}_{R_i^+M}$.
		\end{maprop}
		\begin{proof}
			It suffices to prove it for $M$ indecomposable. If $M$ is rigid the result follows from lemma \ref{lem:exceptionnel}. If $M$ is not rigid then it is regular.
			
			Assume first that $M$ is in an exceptional tube $\mathcal T$ of rank $p>1$. Denote by $E_0, \ldots, E_{p-1}$ the quasi-simple modules of $\mathcal T$ ordered such that $\tau E_i=E_{i-1}$ for all $i \in \Z/p\Z$. Denote by $l$ the quasi-length of $M$ and assume that the $E_i$ are indexed in such a way that $E_0=\regsoc M$. Then we know from \cite{Dupont:stabletubes} that
			$$X^Q_M=P_{l,p}(X^Q_{E_0}, \ldots, X^Q_{E_{p-1}})$$
			where $P_{l,p}$ is the $l$-th generalized Chebyshev polynomial of rank $p$.
			
			According to lemma \ref{lem:Ftubes}, $R_i^+M$ is regular and belongs to a tube $\mathcal T'$ of rank $p$. Moreover corollary \ref{corol:Flongueur} implies that the quasi-simples of $\mathcal T'$ are the $R_i^+E_i$ for $i \in \Z/p\Z$ and $R_i^+E_0=\regsoc R_i^+M$. Moreover, $\reglen(R_i^+M)=\reglen(M)=l$ and thus
			$$X^{\sigma_i Q}_{R_i^+M}=P_{l,p}(X^{\sigma_i Q}_{R_i^+E_0}, \ldots, X^{\sigma_i Q}_{R_i^+E_{p-1}}).$$
			
			As every quasi-simple in an exceptional tube is rigid, it follows that 
			$$X^{\sigma_i Q}_{R_i^+E_i}=\Phi_i(X^Q_{E_i}) \textrm{ for any }i \in \Z/p\Z$$
			and thus
			\begin{align*}
				X^{\sigma_i Q}_{R_i^+M}
					&=P_{l,p}(X^{\sigma_i Q}_{R_i^+E_0}, \ldots, X^{\sigma_i Q}_{R_i^+E_{p-1}}) \\
					&=P_{l,p}(\Phi_i(X^Q_{E_0}), \ldots, \Phi_i(X^Q_{E_{p-1}})) \\
					&=\Phi_i(P_{l,p}(X^Q_{E_0}, \ldots, X^Q_{E_{p-1}}) \\
					&=\Phi_i(X^Q_M)
			\end{align*}
			
			It only remains to prove the result for $M$ indecomposable in an homogeneous tube. $R_i^+$ is a triangle equivalence $\mathcal C_Q \fl \mathcal C_{\sigma_i Q}$ so according to lemma \ref{lem:Ftubes}, $R_i^+$ send the homogeneous tubes of $\mathcal C_Q$ to the homogeneous tubes of $\mathcal C_{\sigma_i Q}$. For any $\lambda \in \P^1(k)$, we denote by $\mathcal T_\lambda(Q)$ (resp. $\mathcal T_\lambda(\sigma_i Q)$) the tube in $\mathcal C_Q$ (resp. $\mathcal C_{\sigma_i Q}$) corresponding to the parameter $\lambda$. Up to re-indexation of the tubes, we can assume that $R_i^+\mathcal T_\lambda(Q)=T_\lambda(\sigma_i Q)$ for any $\lambda \in \P^1(k)$ and lemma \ref{lem:Ftubes} implies that $\P^1_0(\sigma_i Q)=\P^1_0(Q)$.
			The set $\ens{M_\lambda}_{\lambda \in \P^1_0}$ is a set of representatives of the quasi-simple modules in the homogeneous tubes in $\mathcal C_Q$ and $\ens{R_i^+M_\lambda}_{\lambda \in \P^1_0}$ is a set of representatives of the quasi-simple modules in the homogeneous tubes in $\mathcal C_{\sigma_i Q}$. Moreover, it follows from lemma \ref{lem:XMlambda} and corollary \ref{lem:Chebyshev} that 
			$$X^Q_{M_\lambda^{(n)}}=X^Q_{M_\mu^{(n)}}=C_n(X^Q_{M_\lambda}) \textrm{ and }X^{\sigma_i Q}_{R_i^+M_\lambda^{(n)}}=X^{\sigma_i Q}_{R_i^+M_\mu^{(n)}}=C_n(X^{\sigma_i Q}_{R_i^+M_\lambda})$$
			for any $\lambda,\mu \in \P^1_0(Q)$. It thus suffices to prove that $$\Phi_i(X^Q_{M_\lambda})=X^{\sigma_i Q}_{R_i^+M_\lambda}$$
			for some $\lambda \in \P^1_0(Q)$.
			
			The difference property (\ref{eq:differencedelta}) of $Q$ implies that for any quasi-simple $E$ in an exceptional tube, we have 
			$$X^Q_{M_\lambda}=X^Q_{M_E}-X^Q_{\regrad M_E/E}$$
			and the difference property (\ref{eq:differencedelta}) of $\sigma_iQ$
			$$X^{\sigma_i Q}_{R_i^+M_\lambda}=X^Q_{R_i^+M_E}-X^Q_{\regrad R_i^+M_E/\regsoc R_i^+M_E}$$
			In particular, 
			\begin{align*}
			\Phi_i(X^Q_{M_\lambda})
				&=\Phi_i(X^Q_{M_E})-\Phi_i(X^Q_{\regrad M_E/E})\\
				&= X^{\sigma_i Q}_{R_i^+(M_E)}-X^{\sigma_i Q}_{R_i^+(\regrad(M_E/E)}
			\end{align*}
			It follows from lemma \ref{lem:Ftubes} that $R_i^+(M_E)=M_{R_i^+(E)}$ and $\regrad R_i^+M_E/\regsoc R_i^+M_E=\regrad M_{R_i^+(E)}/R_i^+(E)$ and thus
			\begin{align*}
				\Phi_i(X^Q_{M_\lambda})
					&=X^{\sigma_i Q}_{M_{R_i^+(E)}}-X^Q_{\regrad M_{R_i^+E}/R_i^+E} \\
					&= X^{\sigma_i Q}_{R_i^+M_\lambda}.\\
			\end{align*}
		\end{proof}
		
		A direct induction on proposition \ref{prop:reflexionaffine} leads to the following proposition:
		\begin{maprop}\label{prop:reflexionCC}
			Let $Q$ be a quiver of affine type with at least three vertices satisfying the difference property. Let $Q'$ be a quiver reflection-equivalent to $Q$ and $(i_1, \ldots, i_n)$ be an admissible sequence of sinks such that $Q'=\sigma_{i_n} \circ \cdots \circ \sigma_{i_1}(Q)$. Assume moreover that $\sigma_{i_k} \cdots \sigma_{i_1}(Q)$ satisfies the difference property for every $k=1, \ldots, n$. Denote by $\Phi$ the canonical isomorphism of cluster algebras $\mathcal A(Q) \fl \mathcal A(Q')$ and by 
			$$R^+=R_{i_n}^+\cdots R_{i_1}^+:\mathcal C_{Q} \fl \mathcal C_{Q'}$$
			Then 
				$$\Phi(X^Q_{M}) =X^{Q'}_{R^+M}$$
			for every object $M$ in $\mathcal C_Q$.
		\end{maprop}
	\end{subsubsection}	

	\begin{subsubsection}{Reflections and generic variables}
		We are now able to prove that the generic variables are preserved under reflections.
	
		\begin{theorem}\label{theorem:invariancebase}
			Let $Q$ be an affine quiver of affine type with at least three vertices such that every quiver reflection-equivalent to $Q$ satisfies the difference property. Let $Q'$ be a quiver reflection-equivalent to $Q$. Write $\Phi: \mathcal A(Q) \fl \mathcal A(Q')$ the canonical isomorphism of cluster algebras. Then
			$$\Phi(\mathcal B'(Q))=\mathcal B'(Q').$$
		\end{theorem}
		\begin{proof}
			Let $(i_1, \ldots, i_n)$ be an admissible sequence of sinks such that $Q'=\sigma_{i_n} \circ \cdots \circ \sigma_{i_1}(Q)$. 
			We denote by $R^+$ the composition
			$$R^+=R_{i_n}^+\cdots R_{i_1}^+:\mathcal C_{Q} \fl \mathcal C_{Q'},$$
			it is an equivalence of triangulated categories.
			According to proposition \ref{prop:reflexionCC}, we have $\Phi(X^Q_{M}) =X^{Q'}_{R^+M}$ for every object $M$ in $\mathcal C_Q$.
			
			Fix $\textbf d \in \Z^{Q_0}$ and assume that $\textbf d \not \in \N^{Q_0}$. Then
			$$X_{\textbf d}=X_{[\textbf d]_+} \prod_{d_i <0} X^Q_{P_i[1]^{\oplus (-d_i)}}$$
			Because $[\textbf d]_+$ is not sincere, there exists some $M_+ \in \rep(Q,[\textbf d]_+)$ which is rigid. It follows that
			$M=M_+ \oplus \bigoplus_{d_i<0}P_i[1]^{-d_i}$ is also rigid and thus so is $R^+M$. Then,  $\Phi(X^Q_M)=X^{Q'}_{R^+M}$ is in $\mathcal B'(Q')$.
			
			If $\textbf d \in \N^{Q_0}$, assume that there is some rigid module $M \in U_{\textbf d}$, then $R^+M$ is rigid and thus $\Phi(X^Q_M)=X^{Q'}_{R^+M}$ is in $\mathcal B'(Q')$. 
			Otherwise, according to the canonical decomposition of $[\textbf d]_+$, $M \in U_{[\textbf d]_+}\cap \mathfrak M_{[\textbf d]_+}$ decomposes into
			$$M=\bigoplus_{i \in I} M_{\lambda_i} \oplus \bigoplus_{j \in J} N_j$$
			where the $N_j$ are indecomposable rigid modules with trivial endomorphism ring for all $i,j \in J$ and $\Ext^1_{\mathcal C_Q}(U,V)=0$ for any pairwise distinct indecomposable summand of $M$ and $I$ is a non-empty set.
			It follows that 
			$$R^+M=\bigoplus_{i \in I} R^+M_{\lambda_i} \oplus \bigoplus_{j \in J} R^+N_j$$
			where the $R^+N_j$ are indecomposable rigid modules with local endomorphism ring for all $i,j \in J$ and $\Ext^1_{\mathcal C_Q}(U,V)=0$ for any pairwise distinct indecomposable summand of $R^+M$. Assume that $R^+N_j \simeq P_k[1]$ for some projective $kQ'$-module $P_k$. Then
		 	$$\Ext^1_{\mathcal C_{Q'}}(R^+N_j,M_{\lambda_i})=\Ext^1_{\mathcal C_{Q'}}(P_k[1],M_{\lambda_i})=\dim M_{\lambda_i}(k)>0$$
		 	which is a contradiction. It follows that $R^+M$ is a $kQ'$-module and according to the study in subsection \ref{subsection:dcpcanonique}, we have 
			$$X^{Q'}_{R^+M}=X^{Q'}_{\ddim R^+M}$$
			and thus $\Phi(\mathcal B'(Q)) \subset \mathcal B'(Q')$.
			
			Conversely, $Q$ can be obtained from $Q'$ after a sequence of reflections and thus the same proof shows the inverse inclusion.
		\end{proof}
		
		If $i \in Q_0$, we denote by $s_i$ the standard reflection of the Weyl group associated to $i$. We denote by $\sigma_i$ the piecewise linear transformation defined on $\Z\Phi(Q)$ as follows. If $\alpha \in \Phi_{\geq -1}(Q)$, following \cite{MRZ} we set
		$$\sigma_i(\textbf d)=\left\{\begin{array}{rcl}
			\textbf d& \textrm{ if } \textbf d=-\alpha_j \textrm{ for }j \neq i,\\
			s_i(\textbf d) & \textrm{ otherwise.}
		\end{array}\right.$$
		If $\textbf d \in \N^{Q_0}$, we write $\textbf d=\textbf e_1 \oplus \cdots \oplus \textbf e_n$ the canonical decomposition of $\textbf d$ and we set
		$$\sigma_i(\textbf d)=\sum_{i=1}^n \sigma_i(\textbf e_i).$$
		If $\textbf d \in \Z^{Q_0}$, we set 
		$$\sigma_i(\textbf d)=\sigma_i([\textbf d]_+)+[\textbf d]_-+2d_i\alpha_i.$$
		This is an involution of $\Z^{Q_0}$.
		
		\begin{monlem}\label{lem:dimensionreflexion}
			Let $Q$ be an acyclic quiver, $i$ and sink in $Q$. Fix $M$ an object in $\mathcal C_Q$ such that
			$$M \simeq M_1 \oplus \cdots \oplus M_n$$
			where the $M_k$ are indecomposable objects such that $\Ext^1_{\mathcal C_Q}(M_k,M_j)=0$ for $k \neq j$. Then 
			$$\ddim R_i^+(M)=\sigma_i(\ddim M).$$
		\end{monlem}
		\begin{proof}
			If $n=1$, then $M=M_1$ is indecomposable and $\ddim M \in \Phi_{\geq 1}$ and the result is proved in \cite[theorem 3.4]{Zhu:equivalence}. Assume now $n >1$. As $\Ext^1_{\mathcal C_Q}(M_k,M_j)=0$ for $k \neq j$, if some $M_j$ is isomorphic to some $P_k[1]$, then $(\ddim M_l)_k \leq 0$ for any $l=1,\ldots, n$. It follows that 
			$$[\ddim M]_+=\ddim H^0(M).$$
			We write 
			$$M=H^0(M)\oplus \bigoplus_{k \in Q_0} P_k[1]^{\oplus p_k}.$$
			In particular, for any $k \in Q_0$, we have $-p_k=(\ddim M)_k$. We have 
			$$R_i^+(M)=\Sigma_i^+(H^0(M)) \oplus S_i^{\oplus p_i} \oplus \bigoplus_{k \neq i} P_k[1]^{\oplus p_k}$$
			where $\Sigma_i^+$ denotes the standard BGP-reflection functor. It follows that 
			\begin{align*}
				\ddim R_i^+M
					&=s_i(\ddim H^0(M))+p_i \alpha_i -p_k\alpha_k \\
					&=s_i([\ddim M]_+)+[\ddim M]_-+2p_i \alpha_i\\
					&=\sigma_i([\ddim M]_+)+[\ddim M]_-+2p_i \alpha_i\\
					&=\sigma_i(\ddim M).
			\end{align*}
		\end{proof}	
		
		\begin{corol}\label{corol:reflectionbase}
		 	Let $Q$ be an affine quiver of affine type with at least three vertices satisfying the difference property. Let $i$ be a sink in $Q$ such that $\sigma_iQ$ satisfies the difference property, and $\Phi_i: \mathcal A(Q) \fl \mathcal A(\sigma_i Q)$ be the canonical isomorphism of cluster algebras. Then for any $\textbf d \in \Z^{Q_0}$, we have
			$$\Phi_i(X^Q_{\textbf d})=X^{\sigma_i Q}_{\sigma_i (\textbf d)}.$$
		\end{corol}
		\begin{proof}
			According to the denominators theorem for any object $M$ in $\mathcal C_Q$, we have 
			$$\ddim R_i^+(M)=\delta(X^{\sigma_i Q}_{R_i^+M}).$$
			Fix $\textbf d \in \Z^{Q_0}$ and $M$ such that $X_M = X_{\textbf d}$. Then we can choose $M$ such that it satisfies the hypothesis of lemma \ref{lem:dimensionreflexion}, in this case, we have $\ddim R_i^+(M)=\sigma_i(\ddim M)$. Now, we know from proposition \ref{prop:reflexionaffine} that $\Phi_i(X^Q_M)=X^{\sigma_i}_{R_i^+M}$ and thus it follows from theorem \ref{theorem:invariancebase} 
			that $$\Phi_i(X^Q_{\textbf d})=X^{\sigma_i Q}_{\sigma_{i}(\textbf d)}.$$
		\end{proof}
	\end{subsubsection}
\end{subsection}
 
\begin{subsection}{Linear independence for generic values}\label{subsection:independence}
	\begin{subsubsection}{Gradability and linear independence}
		In \cite{CK1}, a condition was introduced on the quiver $Q$ in order to obtain a nice framework for problems of linear independence of generalized variables. In the sequel, we will refer to this condition as \emph{gradability}\index{gradability}. For results concerning gradability, one can also refer to \cite{CK1,Dupont:stabletubes}.
		
		We recall that for any acyclic quiver $Q$, the \emph{matrix $B$ associated to $Q$} is the anti-symmetric matrix whose entries are given by
		$$b_{ij}=|\ens{i \fl j \in Q_1}|-|\ens{j \fl i \in Q_1}|$$
		for all $i,j \in Q_0$.
		
		\begin{defi}
			Let $Q$ be an acyclic quiver with associated matrix $B$. $Q$ will be called \emph{graded}\index{graded} if there exists a linear form $\epsilon$ on $\Z^{Q_0}$ such that $\epsilon(B \alpha_i)<0$ for any $i \in Q_0$ where $\alpha_i$ still denotes the $i$-th vector of the canonical basis of $\Z^{Q_0}$.
		\end{defi}
		
		An orientation of $Q$ is called \emph{alternating}\index{alternating} if every vertex in $Q_0$ is either a sink or a source. 
		
		\begin{monexmp}\label{exmp:alternating}
			Every quiver equipped with an alternating orientation is graded. Indeed, in this case the matrix $B$ of $Q$ satisfies for all $j \in Q_0$, $b_{ij}>0$ if $i$ is a source and $b_{ij}<0$ if $i$ is a sink. We can choose any form $\epsilon$ in the dual basis of $\Z^{Q_0}$ such that $\epsilon_i<0$ if $i$ is a source and $\epsilon_i>0$ if $i$ is a sink.
		\end{monexmp}

		If $Q$ is a graded quiver, then it is proved in \cite{CK1} that we can endow $\mathcal A(Q)$ with a grading. Namely, the results are the following:
		
		For any Laurent polynomial $L$ in the variables $\textbf u$, the \emph{support}\index{support} $\supp(L)$ of $L$ is defined as the set of points $\lambda=(\lambda_i,i \in Q_0)$ of $\Z^{Q_0}$ such that the $\lambda$-component, that is, the coefficient of $\prod_{i \in Q_0} u_i^{\lambda_i}$ in $L$ is non-zero. For any $\lambda$ in $\Z^{Q_0}$, let $C_\lambda$ be the convex cone with vertex $\lambda$ and edge vectors generated by the $B\alpha_i$ for any $i \in Q_0$.

		\begin{maprop}[\cite{CK1}]\label{prop:supportconeCK1}
			Let $Q$ be an acyclic quiver with no multiple arrows. Fix an indecomposable object $M$ of $\mathcal C_Q$ and write $M=H^0(M) \oplus P_M[1]$. Then, $\supp(X_M)$ is in $C_{\lambda_M}$ with $\lambda_M:=(-\<\alpha_i,\ddim H^0(M)\>+\<\ddim P_M, \alpha_i\>)_{i \in Q_0}$.  Moreover, the $\lambda_M$-component of $X_M$ is 1.
		\end{maprop}
		
		\begin{maprop}[\cite{CK1}]\label{prop:graduationCK1}
			Let $Q$ be a graded acyclic quiver with no multiple arrows. For every $n \in \Z$, set 
			$$F_n=\left( \bigoplus_{\epsilon(\nu) \leq n} \Z\prod_{i \in Q_0}u_i^{\nu_i}\right) \cap \mathcal A(Q),$$
			then the set $(F_n)_{n \in \Z}$ defines a $\Z$-filtration of $\mathcal A(Q)$.		
		\end{maprop}
		
		The filtration given in proposition \ref{prop:graduationCK1} and the proposition \ref{prop:supportconeCK1} imply the following lemma:
		
		\begin{monlem}[\cite{CK1}]\label{lem:lemCK1}
			Fix $Q$ a graded quiver. Let $\ens{M_1, \ldots, M_r}$ be a family of objects in $\mathcal C_Q$ such that $\ddim M_i \neq \ddim M_j$ if $i \neq j$, then $X^Q_{M_1}, \ldots, X^Q_{M_r}$ are linearly independent over $\Q$.
		\end{monlem}
		
		\begin{defi}
			A quiver $Q$ is called \emph{gradable}\index{gradable} if there exists a graded quiver $Q'$ reflection-equivalent to $Q$.
		\end{defi}
		
		\begin{rmq}
			Note that not all the quivers are gradable. For example, consider the 3-cycle $Q$ with double arrows:
			$$
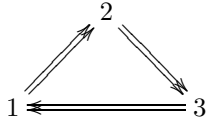
\xymatrix{
				&2 \ar@<+1pt>[rd]\ar@<-1pt>[rd]\\
				1 \ar@<+1pt>[ru]\ar@<-1pt>[ru] && \ar@<+1pt>[ll]\ar@<-1pt>[ll] 3
			}$$
			Denote by $c_i$ the $i$-th column of the matrix $B$ associated to $Q$. Then $c_3=-c_1-c_2$ and $Q$ is the only quiver in its reflection class (and also in its mutation class, see \cite{cluster2} for definitions).
		\end{rmq}
		
		We will see in the next subsection that all the quivers of affine types $\Aaffine$ are gradable. 
		
		\begin{maprop}\label{prop:gradablelinind}
		 	Let $Q$ be a gradable quiver of affine type $\Aaffine$ such that every quiver reflection-equivalent to $Q$ satisfies the difference property. Let $\ens{M_1, \ldots, M_r}$ be a family of objects such that $\ddim M_i \neq \ddim M_j$ if $i \neq j$. Then $\ens{X^Q_{M_1}, \ldots, X^Q_{M_r}}$ is linearly independent over $\Q$.
		\end{maprop}
		\begin{proof}
			If $Q$ is the Kronecker quiver, then every quiver $Q'$ reflection-equivalent to $Q$ is graded and thus the proposition holds by lemma \ref{lem:lemCK1}. From now on, we assume that $Q$ has at least three vertices. Fix $Q'$ a graded quiver reflection-equivalent to $Q$ and denote by $(i_1, \ldots, i_n)$ an admissible sequence of sinks such that $Q'=\sigma_{i_n} \circ \cdots \circ \sigma_{i_1}(Q)$. We denote by 
			$$R^+=R_{i_n}^+ \circ \cdots \circ R_{i_1}^+:\mathcal C_Q \fl \mathcal C_{Q'}$$
			the equivalence of triangulated categories from $\mathcal C_Q$ to $\mathcal C_{Q'}$
			and by $\sigma$ the composition of piecewise linear transformations
			$$\sigma_{i_n} \circ \cdots \sigma_{i_1}: \Z\Phi(Q) \fl \Z\Phi(Q').$$
			We know in particular that for any objects $M,N$ such that $\ddim M \neq \ddim N$, we have $\ddim R^+M=\sigma \ddim M \neq \sigma \ddim N= \ddim R^+N$.
			
			If we denote by $\Phi:\mathcal A(Q) \fl \mathcal A(Q')$ the canonical isomorphism, it follows from proposition \ref{prop:reflexionCC} that 
			$$\Phi(X^Q_M)=X^{Q'}_{R^+M}$$
			for any object $M$ in $\mathcal C_Q$. In particular
			$$\ens{\Phi(X^Q_{M_1}), \ldots, \Phi(X^Q_{M_r})}=\ens{X^{Q'}_{R^+M_1}, \ldots, X^{Q'}_{R^+M_r}}$$
			satisfies the conditions of lemma \ref{lem:lemCK1}. $Q'$ being graded it follows that the family $\ens{\Phi(X^Q_{M_1}), \ldots, \Phi(X^Q_{M_r})}$ is linearly independent over $\Q$. As $\Phi$ is a $\Q$-algebra homomorphism, it follows that $\ens{X^Q_{M_1}, \ldots, X^Q_{M_r}}$ is linearly independent over $\Q$.
		\end{proof}
	\end{subsubsection}

	\begin{subsubsection}{Gradability for quivers of affine type $\Aaffine$}
		As claimed before, we now prove that any quiver of affine type $\Aaffine$ is gradable. This will be done by a case-by-case inspection depending on the indices $r,s$ of the type $\affA{r}{s}$.
		
		\begin{monlem}\label{lem:Agradable}
			Let $Q$ be a quiver of affine type $\Aaffine$. Then $Q$ is gradable.
		\end{monlem}
		\begin{proof}
			We fix $r,s \in \N$ two integers and we consider a quiver $Q$ of affine type $\affA{r+1}{s+1}$. We recall that all the quivers of affine type $\affA{r+1}{s+1}$ are reflection-equivalent. We will thus always assume that $Q$ is equipped with the following orientation:
			$$
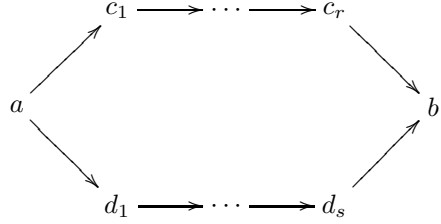
\xymatrix{
				& c_1 \ar[r] & \cdots \ar[r] & c_r \ar[rd] \\
				a \ar[ru] \ar[rd] &&&& b \\
				& d_1 \ar[r] & \cdots \ar[r] & d_s \ar[ru] \\
			}$$
			We will do a case-by-case inspection on the possible pairs $(r,s)$. By symmetry, it is sufficient to consider the pairs $\ens{r,s}$. In each case, we will construct a form $\epsilon$ such that $Q$ satisfies the condition of gradability for this $\epsilon$. If $\epsilon$ is such a form, we will simply say that $\epsilon$ \emph{fits} to $Q$.
			
			If $B \in M_{Q_0}(\Z)$ is a matrix, we denote by $C_i=B\alpha_i$ the column associated to the vertex $i \in Q_0$.
		
			Assume first that $s=r=0$. Then $Q$ is the Kronecker quiver  $\xymatrix{ a \ar@<0.5ex>[r] \ar@<-0.5ex>[r] & b}$, it has an alternating orientation and thus by example \ref{exmp:alternating}, such a form exists.
		
			Assume now that $\ens{r,s}=\ens{0,1}$. We assume that $r=1$ and $s=0$. Then $Q$ is the quiver
			$$
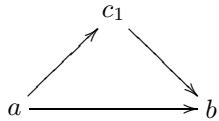
\xymatrix{
				& c_1 \ar[rd]\\
				a \ar[rr] \ar[ru] &&b
			}$$
			the associated matrix with indexation $(a,c_1,b)$ is
			$$B=\left[\begin{array}{ccc}
				0 & 1 & 1\\
				-1 & 0 & 1\\
				-1 & -1 & 0
			\end{array}\right]$$
			The columns $C_a$ and $C_{c_1}$ are linearly independent and $C_{b}=C_{c_1}-C_a$. We set $\epsilon(C_a)=-1$, $\epsilon(C_{c_1})=-2$ and then $\epsilon(C_b)=-1<0$. Thus $\epsilon$ fits.
		
			Assume that $r=s=1$. Then $Q$ is the quiver
			$$
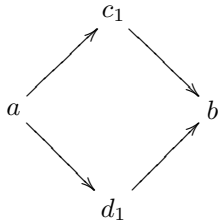
\xymatrix{
				& c_1 \ar[rd]\\
				a\ar[rd] \ar[ru] &&b\\
				& d_1 \ar[ru]
			}$$
			In the associated matrix $B$, we have $C_b=-C_a$, thus we will not be able to find any form $\epsilon$ fitting to $Q$. Consider then the quiver $Q'=\mu_b (Q)=\sigma_b(Q)$ given by 
			$$
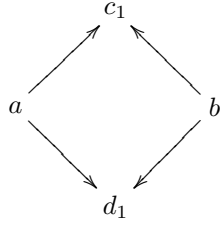
\xymatrix{
				& c_1\\
				a\ar[rd] \ar[ru] && \ar[lu] \ar[ld] b\\
				& d_1
			}$$
			$Q'$ is alternating, so it is graded and $Q$ is gradable.

			Now assume that $s=0$ and $r \geq 2$ (or equivalently $r=0$ and $s \geq 2$). Then $Q$ is the quiver
			$$
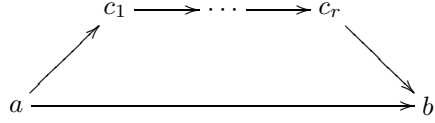
\xymatrix{
				& c_1 \ar[r] & \cdots \ar[r] & c_r \ar[rd]\\
				a \ar[ru] \ar[rrrr] &&&& b
			}$$
			The associated matrix in the indexation $(a, c_1, \ldots, c_r, b)$ is thus
			$$\left[ \begin{array}{c|cccc|c}
				0 & 1 & 0 & \cdots & 0 & 1\\ \hline
				-1 & 0 & 1 & & & \\
				& \ddots & \ddots & \ddots &  \\
				& & \ddots & \ddots &  1\\
					& &  & -1 & 0 &  1\\ \hline
				-1 & && &-1&0
			\end{array}\right]$$
			and we have
			$$\sum \lambda_j C_j=\left[ \begin{array}{c}
				\lambda_b + \lambda_{c_1} \\ \hline
				\lambda_{c_2} - \lambda_{a} \\
				\vdots \\
				\lambda_{b} - \lambda_{c_{r-1}} \\ \hline
				-\lambda_{c_r} - \lambda_a
			\end{array}\right]$$
			If $r \in 2\Z$, then 
			$$\sum \lambda_j C_j=0 
			\Leftrightarrow 
				\left\{\begin{array}{rl}
					\lambda_{c_1} &= - \lambda_b\\
					\lambda_a &=\lambda_{c_2} = \cdots = \lambda_{c_r} \\
					\lambda_b &=\lambda_{c_{r-1}} = \cdots = \lambda_{c_1}\\
					\lambda_a &=-\lambda_{c_r}
				\end{array}\right. 
			\Leftrightarrow 
				\lambda_j=0 \textrm{ for all }j \in Q_0$$
			Thus, $B$ is of full rank and $Q$ is graded.
			
			If $r \in 2\Z+1$, then 
			$$\sum_{j \neq b} \lambda_j C_j=0 
			\Leftrightarrow 
				\left\{\begin{array}{rl}
					\lambda_{c_1} &0\\
					\lambda_a &=\lambda_{c_2} = \cdots = \lambda_{c_{r-1}} \\
					\lambda_{c_{r-1}} &= \lambda_{c_{r-3}}= \cdots = \lambda_{c_2}=0\\
					\lambda_{c_r} &= \cdots \lambda_{c_{r-2}} = \cdots = \lambda_{c_1}\\
					\lambda_a &=-\lambda_{c_r} 
				\end{array}\right. 
			\Leftrightarrow 
				\lambda_j=0 \textrm{ for all }j$$
			Thus, the columns $(C_a,C_{c_1}, \ldots, C_{c_r})$ are linearly independent and 
			$$C_b=-C_a+C_{c_1}-C_{c_2}+C_{c_3}- \cdots -C_{c_{r-1}}+ C_{c_r}$$
			We thus set, $\epsilon(C_i)=-1$ for all $i \neq c_1,b$ and $\epsilon(C_{c_1})=-2$. Then  $\epsilon(C_b)<0$ and $\epsilon$ fits to the condition.
		
			Assume now that $s=1$ and $r \geq 2$ (or equivalently $r=1$ and $s \geq 2 \Z$)
			Then $Q$ is the quiver
			$$
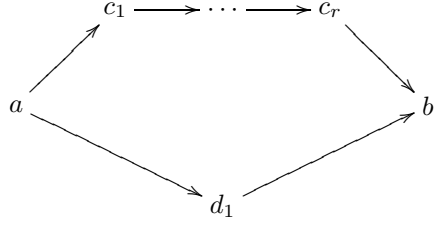
\xymatrix{
				& c_1 \ar[r] & \cdots \ar[r] & c_r \ar[rd] \\
				a \ar[ru] \ar[rrd] & & && b \\
				&& d_1 \ar[rru]
			}$$
			The associated matrix in the indexation $(a, c_1, \ldots, c_r, b,d_1)$ is 
			$$\left[\begin{array}{c|cccc|c|ccc}
					0 & 1 & 0 & \cdots & 0 & 0 & 1 \\ \hline
					-1 & 0 & 1 &&&&\\ 
					0 & -1 & 0 & \ddots &&&\\ 
					&& \ddots & \ddots & 1 &\\ 
					&&& -1 & 0 & 1 & \\ \hline
					&&&    & -1 &0& -1\\ \hline
					-1 & &&&& 1& 0  \\
				\end{array}\right].$$
			We have
			$$\sum_{j \in Q_0}\lambda_j C_j=\left[\begin{array}{c}
				\lambda_{c_1}+\lambda_{d_1} \\
				\lambda_{c_2}-\lambda_a\\
				\lambda_{c_3}-\lambda_{c_1}\\
				\vdots\\
				\lambda_b-\lambda_{c_{r-1}}\\ \hline
				-\lambda_{d_1}-\lambda_{c_r}\\ \hline
				\lambda_b - \lambda_a
			\end{array}\right]$$
			
			If $r \in 2\Z$, 
			$$\sum_{j \in Q_0}\lambda_j C_j = 0 \Leftrightarrow \lambda_a=\lambda_b=\lambda_{c_1}=\lambda_{c_2}=\cdots = \lambda_{c_r}=-\lambda_{d_1}$$
			So $\sum_{j \neq d_1}\lambda_j C_j = 0 \Rightarrow \lambda_j=0$ for all $j$ and thus the columns $(C_a, C_{c_1}, \ldots, C_{c_r}, C_b)$ are linearly independent and $C_{d_1}$ is in the spanning of the other columns. More precisely, 
			$$\sum_{j \neq d_1}\lambda_j C_j=\left[\begin{array}{c}
				\lambda_{c_1}\\
				\lambda_{c_2}-\lambda_a\\
				\lambda_{c_3}-\lambda_{c_1}\\
				\vdots\\
				\lambda_b-\lambda_{c_{r-1}}\\ \hline
				-\lambda_{c_r}\\ \hline
				\lambda_b - \lambda_a
			\end{array}\right]=\left[\begin{array}{c}
				1\\
				0\\
				\vdots\\
				\vdots\\
				0\\ \hline
				-1\\ \hline
				0
			\end{array}\right] 
			\Leftrightarrow \left\{\begin{array}{rl}
				\lambda_{c_1}& =\lambda_{c_2} = \cdots = \lambda_{c_{r}}=\lambda_b=\lambda_a=1\\
			\end{array} \right.$$
			Thus, $C_{d_1}=\sum_{j \neq d_1} C_j$ and thus it suffices to set $\epsilon(C_j)=-1$ for every $j \neq d_1$ and $Q$ is graded.
			
			If $r \in 2 \Z+1$, $Q$ is not graded. Consider the quiver
			$$\xymatrix{
				&& c_1 \ar[r] & \cdots \ar[r] & c_r  \\
				Q'=\mu_b(Q)= 	&a \ar[ru] \ar[rrd] & & && \ar[lu]\ar[lld] b \\
				&&& d_1
			}$$
			The associated matrix in the indexation $(a, c_1, \ldots, c_r, b,d_1)$ is 
			$$\left[\begin{array}{c|cccc|c|ccc}
				0 & 1 & 0 & \cdots & 0 & 0 & 1 \\ \hline
				-1 & 0 & 1 &&&&\\ 
				0 & -1 & 0 & \ddots &&&\\ 
				&& \ddots & \ddots & 1 &\\ 
				&&& -1 & 0 & -1 & \\ \hline
				&&&    & 1 &0& 1\\ \hline
				-1 & &&&& -1& 0  \\
			\end{array}\right]$$
			we have
			$$\sum_{j \neq d_1}\lambda_j C_j=\left[\begin{array}{c}
			\lambda_{c_1}+\lambda_{d_1}\\
			\lambda_{c_2}-\lambda_a\\
			\vdots\\
			\lambda_{c_r} - \lambda_{c_{r-2}}\\
			-\lambda_b-\lambda_{c_{r-1}}\\ \hline
			\lambda_{c_r}+\lambda_{d_1}\\ \hline
			-\lambda_b - \lambda_a
			\end{array}\right].$$
			Then,
			$$\sum_{j \in Q_0} \lambda_j C_j =0 
			\Leftrightarrow 
			\left\{ \begin{array}{rl}
				\lambda_{c_1} &=\lambda_{c_3} = \cdots = \lambda_{c_r} = - \lambda_{d_1}\\
				\lambda_{a} &=\lambda_{c_2} = \cdots = \lambda_{c_{r-1}} = - \lambda_{b}\\
			\end{array}\right.$$
			So the columns $(C_a, C_{c_1}, \ldots, C_{c_r})$ are linearly independent and $C_b$ and $C_{d_1}$ are in their spanning. More precisely, we have 
			$$C_b=C_a+C_{c_2}+\cdots + C_{c_{r-1}} \textrm{ and } C_{d_1}=C_{c_1}+C_{c_3}+ \cdots + C_{c_{r}}$$
			We thus set $\epsilon(C_i)=-1$ for all $i \neq b,d_1$ and thus $\epsilon$ fits. It follows that $Q'$ is graded and $Q$ is gradable.
	
			Now we assume that $r,s \geq 2$, $Q$ is the quiver
			$$\xymatrix{
				& c_1 \ar[r] & \cdots \ar[r] & c_r \ar[rd] \\
				a \ar[ru] \ar[rd] &&&& b \\
				& d_1 \ar[r] & \cdots \ar[r] & d_s \ar[ru] \\
			}$$
			the associated matrix $B$ in the indexation $(a,c_1,\ldots,c_r,b,d_1,\ldots,d_s)$ is
			$$\left[\begin{array}{c|cccc|c|ccccc}
				0 & 1 & 0 & \cdots & 0 & 0 & 1 & 0 & \cdots & 0\\ \hline
				-1 & 0 & 1 &&&&&&&\\ 
				0 & -1 & 0 & \ddots &&&&&\\ 
				&& \ddots & \ddots & 1 &&&&\\ 
				&&& -1 & 0 & 1\\ \hline
				&&&    & -1 &0&&&& -1\\ \hline
				-1 & &&&&& 0 & 1 \\
				& &&&&& -1 & \ddots & \ddots \\
				& &&&&& & \ddots & \ddots &1 \\
				&&&& &1&& & -1& 0\\
			\end{array}\right]$$
			We have
			$$\sum_{j \in Q_0} \lambda_j C_j=\left[ \begin{array}{c}
				\lambda_{c_1}+\lambda_{d_1} \\
				\lambda_{c_2}-\lambda_{a} \\
				\lambda_{c_3}-\lambda_{c_1} \\
				\vdots \\
				\lambda_{c_r}-\lambda_{c_{r-2}}\\
				\lambda_{b}-\lambda_{c_{r-1}}\\ \hline 
				-\lambda_{d_s}-\lambda_{c_{r}}\\ \hline
				\lambda_{d_2}-\lambda_{a}\\
				\lambda_{d_3}-\lambda_{d_1}\\
				\vdots \\
				\lambda_{d_s}-\lambda_{d_{s-2}}\\
				\lambda_{b}-\lambda_{d_{s-1}}
			\end{array}\right]$$
		
			If $r,s \in 2 \Z$, then 
			$$\sum_{j \in Q_0} \lambda_j C_j=0
			\Leftrightarrow
			\left\{\begin{array}{rl}
				\lambda_{c_1} &= - \lambda_{d_1}\\
				\lambda_a&=\lambda_{c_2}= \ldots = \lambda_{c_r}\\
				\lambda_b&=\lambda_{c_{r-1}}= \ldots = \lambda_{c_1}\\
				\lambda_{c_r}& =-\lambda_{d_s}\\
				\lambda_a&=\lambda_{d_2}= \ldots = \lambda_{d_s}\\
				\lambda_b&=\lambda_{d_{s-1}}= \ldots = \lambda_{d_1}
			\end{array}\right.
			\Leftrightarrow \lambda_i=0 \textrm{ pour tout i}$$
			and the columns are linearly independent. $B$ is thus of full rank and $Q$ is graded.
		
			If $r \in 2\Z$ and $s \in 2 \Z+1$ (or equivalently $r \in 2\Z+1$ and$s \in 2 \Z$), then 
			$$\sum_{j \in Q_0} \lambda_j C_j=0
			\Leftrightarrow
			\left\{\begin{array}{rl}
				\lambda_{c_1} &= - \lambda_{d_1}\\
				\lambda_a&=\lambda_{c_2}= \ldots = \lambda_{c_{r-1}}\\
				\lambda_b&=\lambda_{c_{r-1}}= \ldots = \lambda_{c_1}\\
				\lambda_{c_r}& =-\lambda_{d_s}\\
				\lambda_a&=\lambda_{d_2}= \ldots = \lambda_{d_s}=\lambda_b\\
				\lambda_b&=\lambda_{d_{s-1}}= \ldots = \lambda_{d_1}
			\end{array}\right.$$			
			Thus, $\sum_{j \neq d_s} \lambda_j C_j=0$ if and only if $\lambda_j=0$ for all $j$. 
			The columns $C_a$, $C_{c_1}, \ldots, C_{c_r}$, $C_b$, $C_{d_1}, \ldots, C_{d_{s-1}}$ are thus linearly independent and 
			$$C_{d_s}=C_a+C_{c_2}+\cdots + C_{c_r} + C_{d_2}+\cdots + C_{d_{s-1}}-C_{d_1}- \cdots - C_{d_{s-2}}$$
			We set $\epsilon(C_i)=-1$ for every $i \neq a,d_s$ and $\epsilon(C_a)=\min(0,r-1)-1$. We have then $\epsilon(C_i)<0$ for every $i \in Q_0$ and $Q$ is graded.
		
			If $r,s \in 2\Z+1$, $Q$ is not graded. We consider the quiver $Q'=\mu_b(Q)$ whose associated matrix in the indexation $(a,c_1,\ldots,c_r,b,d_1,\ldots,d_s)$ is
			$$B'\left[\begin{array}{c|cccc|c|ccccc}
				0 & 1 & 0 & \cdots & 0 & 0 & 1 & 0 & \cdots & 0\\ \hline
				-1 & 0 & 1 &&&&&&&\\ 
				0 & -1 & 0 & \ddots &&&&&\\ 
				&& \ddots & \ddots & 1 &&&&\\ 
				&&& -1 & 0 & -1\\ \hline
				&&&    & 1 &0&&&& 1\\ \hline
				-1 & &&&&& 0 & 1 \\
				& &&&&& -1 & \ddots & \ddots \\
				& &&&&& & \ddots & \ddots &1 \\
				&&&& &-1&& & -1& 0\\
			\end{array}\right].$$
			If $C'_i$ denotes the $i$-th column of $B'$ for any $i$, we have 
			$$\sum_{j \in Q_0} \lambda_j C'_j=\left[ \begin{array}{c}
				\lambda_{c_1}+\lambda_{d_1} \\
				\lambda_{c_2}-\lambda_{a} \\
				\lambda_{c_3}-\lambda_{c_1} \\
				\vdots \\
				\lambda_{c_r}-\lambda_{c_{r-2}}\\
				-\lambda_{b}-\lambda_{c_{r-1}}\\ \hline 
				-\lambda_{d_s}-\lambda_{c_{r}}\\ \hline
				\lambda_{d_2}-\lambda_{a}\\
				\lambda_{d_3}-\lambda_{d_1}\\
				\vdots \\
				\lambda_{d_s}-\lambda_{d_{s-2}}\\
				-\lambda_{b}-\lambda_{d_{s-1}}
			\end{array}\right]$$
			and
			$$\sum \lambda_j C_j'=0 
			\Leftrightarrow 
			\left\{\begin{array}{rl}
				\lambda_{c_1} &= -\lambda_{d_1}\\
				\lambda_{a} &= \lambda_{c_2} = \cdots = \lambda_{c_{r-1}} = - \lambda_b\\
				\lambda_{c_1}& =\cdots = \lambda_{c_r}\\
				\lambda_{c_r} &=-\lambda_{d_s}\\
				\lambda_{a} &= \lambda_{d_2} = \cdots = \lambda_{d_{s-1}} = - \lambda_b\\
				\lambda_{d_1}& =\cdots = \lambda_{d_r}\\
			\end{array}\right.$$
			Thus, $\sum_{j \neq b,d_s} \lambda_j C_j'=0 \Leftrightarrow \lambda_j=0$ for all $j$. The columns $C'_a$, $C'_{c_1}, \ldots, C'_{c_r}$, $C'_{d_1}, \ldots, C'_{d_{s-1}}$ are thus linearly independent. Moreover, we have
			$$C_{b}'=C_a'+C_{c_2}'+ \cdots + C_{c_{r-1}}' + C_{d_2}' + \cdots + C_{d_{s-1}}'$$
			and
			$$C_{d_s}'=C_{c_1}'+C_{c_3}' + \cdots + C_{c_r}' - C_{d_1}' - \ldots - C_{d_{s-2}}'$$
			We thus set $\epsilon(C'_i)=-1$ for all $i \neq b,d_s,c_1$ and $\epsilon(C_{c_1}')=\min(0, \frac{r+1}{2}-\frac{s-1}{2})-1$. Then, $\epsilon(C_b')<0$ and $\epsilon(C_{d_s}')<0$. $Q'$ is thus graded and so $Q$ is gradable.
			
			In every case, we proved that a quiver $Q$ of affine type $\affA rs$ is gradable.
		\end{proof}
	\end{subsubsection}
	
	\begin{theorem}\label{theorem:semicanonicalbasis}
		Let $Q$ be an affine quiver such that every quiver reflection-equivalent to $Q$ satisfies the difference property. Then the set of generic variables is a $\Z$-basis in $\mathcal A(Q)$.
	\end{theorem}
	\begin{proof}
		Let $Q$ be an affine quiver. If $Q$ is of affine type $\Aaffine$, then $Q$ is gradable by lemma \ref{lem:Agradable}. Otherwise, the underlying diagram $\Delta$ of $Q$ is a tree and thus admits an alternating orientation. As any orientations of a tree are reflection-equivalent, if follows that $Q$ is gradable. By theorem \ref{theorem:invariancebase}, we can thus assume that $Q$ is a graded quiver. It follows from lemma \ref{lem:lemCK1} that $\mathcal B'(Q)$ is linearly independent over $\Q$ and thus over $\Z$.
		
		Now, it follows from proposition \ref{prop:Qspan} that every element $X_M$ is a $\Q$-linear combination of elements of $\mathcal B'(Q)$. For every $\textbf d \in \Z^{Q_0}$, set $$\lambda_{\textbf d}=(-\<\alpha_i, [\textbf d]_+\>+\<[\textbf d]_-, \alpha_i\>)_{i\in Q_0},$$
		If $M$ is such that $X_M=X_{\textbf d}$, then it follows from the definition of $X_{\textbf d}$ that 
		$\ddim H^0(M)=[\textbf d]_+$ and $\ddim P_M[1]=[\textbf d]_-$. Then, it follows from proposition \ref{prop:supportconeCK1} that $\supp(X_{\textbf d}) \in C_{\lambda_{\textbf d}}$ and the $\lambda_{\textbf d}$ component of $X_{\textbf d}$ is 1.
		
		As $Q$ is graded, using the filtration given by proposition \ref{prop:graduationCK1}, we obtain
		$$X_M=X_{\ddim M} + \sum_{\epsilon(\textbf d) <\epsilon(\ddim M)} n_{\textbf d} X_{\textbf d}$$
		where the $n_{\textbf d}$ are rational numbers.
		By induction on the filtration, using the fact that $\chi(\Gr_{\textbf e}(M)) \in \Z$ for every dimension vector $\textbf e$, we see that $n_{\textbf d} \in \Z$ for every $\textbf d \in \Z^{Q_0}$ and thus $\mathcal A(Q)$ is generated by $\mathcal B'(Q)$ as a $\Z$-module.
	\end{proof}
		
	\begin{theorem}\label{theorem:semicanonicalbasisA}
		Let $Q$ be a quiver of affine type $\Aaffine$. Then $\mathcal B'(Q)$ is a $\Z$-basis of the $\Z$-module $\mathcal A(Q)$.
	\end{theorem}
	\begin{proof}
		According to theorem \ref{theorem:differencedelta}, a quiver of affine type $\Aaffine$ satisfies the difference property. Thus $\mathcal B'(Q)$ is a $\Z$-basis in $\mathcal A(Q)$ by theorem \ref{theorem:semicanonicalbasis}.
	\end{proof}
	
	\begin{rmq}
		In the works of \cite{CZ} on the Kronecker quiver (see also \cite{Cerulli:thesis} for affine quivers of rank three), the elements arising in their semicanonical basis with denominator vectors $n \delta$ turn out to be the $X_{M_\lambda^{(n)}}$ for $n \geq 2$ and $\lambda \in \P^1_0$. Nevertheless, it follows from lemma \ref{lem:imaginary} that $X_{n\delta} \neq X_{M_\lambda^{(n)}}$ and thus that generalized variables associated to the regular modules $M_\lambda^{(n)}$ for $n \geq 2$ do not appear as elements of $\mathcal B'(Q)$. Even if it seems more natural from the point of view of the AR quiver of $\mathcal C_Q$ to introduce the $X_{M_\lambda^{(n)}}$, the natural choice from the point of view of geometry is $X_{M_\lambda}^n$. We will see in subsection \ref{subsection:basesKronecker} that for the Kronecker quiver, Caldero-Zelevinsky's semicanonical basis differ from the set $\mathcal B'(Q)$ of generic variables only by a locally unipotent (see subsection \ref{subsection:basesKronecker} for definitions) transformation. 
		
		The reader can find a bit confusing to use also the terminology \emph{semicanonical} in our case whereas the considered set does not coincide with the semicanonical basis found in the previous literature. Nevertheless, because of the very strong analogy between our basis and Lusztig's dual semicanonical basis, we preferred to keep the terminology \emph{semicanonical} even if it can seem a bit confusing at first.
	\end{rmq}

	\begin{monexmp}
	Consider the quiver $Q$ of affine type $\Aaffine_{4,4}$:
		$$
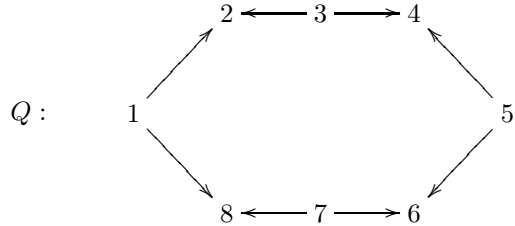
\xymatrix{
			& 	& 2 & \ar[l] 3 \ar[r] & 4 \\
			Q:& 1 \ar[ru] \ar[rd] &&&& 5 \ar[lu] \ar[ld] \\
			& 	& 8 & \ar[l] 7 \ar[r] & 6
		}$$
		
		For any $i \in \Z/8\Z$, denote by $E_{i,i+1}$ the quasi-simple module whose composition factors are $S_i$ and $S_{i+1}$. Then $\tau E_{i,i+1}=E_{i+2,i+3}\textrm{ for all  }i \in \Z/8\Z$. As usual, we denote by $E_{ij}^{(n)}$ the unique indecomposable regular module with quasi-length $n$ and quasi-socle $E_{ij}$. 	
	
		The AR-quiver of $kQ$-mod contains exactly two exceptional tubes which are of rank $p=4$. We consider the tube depicted in figure \ref{figure:tubeA44}.
		\begin{figure}[H]
			\begin{picture}(300,220)(-40,0)
				\setlength{\unitlength}{1mm}
				\put (-21,9){$\bullet$}
				\put (-25,5){$E_{12}$}
				\put (-1,9){$\bullet$}
				\put (-5,5){$E_{78}$}
				\put (0,10){\circle{5}}
				\put (19,9){$\bullet$}
				\put (15,5){$E_{56}$}
				\put (39,9){$\bullet$}
				\put (35,5){$E_{34}$}
				\put (59,9){$\bullet$}
				\put (55,5){$E_{12}$}
				
				\put (79,9){$\bullet$}
				\put (75,5){$E_{78}$}
				\put (80,10){\circle{5}}
				
				\put (99,9){$\bullet$}
				\put (95,5){$E_{56}$}
				
				\put (29,39){$\bullet$}
				\put (32,39){$E^{(4)}_{78}$}
				\put (49,39){$\bullet$}
				\put (52,39){$E^{(4)}_{56}$}
				
				\put (40,30){\circle{5}}
				\put (39,29){$\bullet$}
				\put (42,29){$E_{56}^{(3)}$}

				\put (49,39){$\bullet$}
				\put (52,19){$E_{34}^{(2)}$}
				\put (29,19){$\bullet$}
				\put (32,19){$E_{56}^{(2)}$}

				\multiput(-20,10)(0,20){3}{\multiput(0,0)(20,0){6}{\vector(1,1){10}}}
				\multiput(-20,30)(0,20){3}{\multiput(0,0)(20,0){6}{\vector(1,-1){10}}}
				\multiput(-10,20)(0,20){3}{\multiput(0,0)(20,0){6}{\vector(1,-1){10}}}
				\multiput(-10,20)(0,20){3}{\multiput(0,0)(20,0){6}{\vector(1,1){10}}}
		
				\put (0,10){\line(0,1){60}}
				\put (80,10){\line(0,1){60}}
						
			\end{picture}
			\caption{An exceptional tube of $\Aaffine_{4,4}$}\label{figure:tubeA44}
		\end{figure}	
		
		We consider the module $M=E_{78} \oplus E_{56}^{(3)}$. The multiplication formula gives
		$$X_M=X_{E^{(4)}_{56}}+X_{E_{56}^{(2)}}$$
		but according to the difference property proved in theorem \ref{theorem:differencedelta}, we have
		$$X_{E^{(4)}_{56}}=X_{M_\lambda}+X_{E_{34}^{(2)}}$$
		for every $\lambda \in \P^1_0(Q)$.
		As $E_{34}^{(2)}$ and $E_{56}^{(2)}$ are indecomposable rigid, they belong to $\mathcal B'(Q)$. Moreover $X_{M_\lambda}=X_\delta$ so 
		\begin{align*}
			X_M
			 	&=X_{M_\lambda}+X_{E_{34}^{(2)}}+X_{E_{34}^{(2)}}\\
			 	&=X_{\delta}+X_{\ddim E_{34}^{(2)}}+X_{\ddim E_{34}^{(2)}}\\
			 	& \in \Z \mathcal B'(Q). 
		\end{align*}
	\end{monexmp}
\end{subsection}

\section{Examples of semicanonical bases}\label{section:examples}
	\begin{subsection}{Semicanonical basis for Dynkin quivers}\label{subsection:Dynkin}
	In this short subsection, we prove that the generic variables are the elements of the Caldero-Keller basis found in \cite{CK1} for a quiver of Dynkin type.
	
	\begin{maprop}
		Let $Q$ be a Dynkin quiver. Then 
		$$\mathcal B'(Q)=\ens{\textrm{cluster monomials in } \mathcal A(Q)}$$
		is a $\Z$-basis for the $\Z$-module $\mathcal A(Q)$.
	\end{maprop}
	\begin{proof}
		We denote by $\mathcal M(Q)=\ens{\textrm{cluster monomials in }\mathcal A(Q)}$. It follows from lemma \ref{lem:clustermonomial} that $\mathcal M(Q) \subset \mathcal B'(Q)$. Fix now an element $\textbf d \in \N^{Q_0}$ and write $\textbf d=\bigoplus_{i=1}^n \textbf e_i$ its canonical decomposition. Then, according to proposition \ref{prop:Kacdcp}, for every $i=1, \ldots,n$, there is a Schur representation $M_i \in \rep(Q,\textbf e_i)$ such that $\Ext^1_{kQ}(M_i, M_j) = 0$ for $i \neq j$. As $Q$ is Dynkin, every root is a real root and then each $\textbf e_i$ is a real Schur root and thus each $M_i$ is rigid. It follows that $M=\bigoplus_{i=1}^n M_i$ is rigid in $\rep(Q,\textbf d)$ and then $\mathcal O_M \cap U_{\textbf d} \neq \emptyset$. It follows that $X_{\textbf d}=X_M$ for some rigid module $M$ and thus $X_{\textbf d}$ is a cluster monomial.
		
		Assume now that $\textbf d$ is any element in $\Z^{Q_0}$. Then 
		$$X_{\textbf d}=X_{[\textbf d]_+}\prod_{d_i<0}u_i^{-d_i}.$$
		By the above discussion, there exists some rigid module $M \in \rep(Q,[\textbf d]_+)$ such that $X_M=X_{[\textbf d]_+}$. Moreover, for every $i \in Q_0$ such that $d_i<0$, we have $\Ext^1_{\mathcal C}(M,P_i[1])=0$ so $N=M \oplus \bigoplus_{d_i<0}P_i[1]^{\oplus (-d_i)}$ is rigid in $\mathcal C_Q$ and $X_{\textbf d}=X_N$ is a cluster monomial.
		
		This is a $\Z$-basis according to \cite{CK1}.
	\end{proof}
\end{subsection}

\begin{subsection}{Bases for the Kronecker quiver}\label{subsection:basesKronecker}
	
	The first example of cluster algebra of non Dynkin type that has already been studied is the most simple case of affine quiver, namely the Kronecker quiver
	$$\xymatrix{ K : 1 & \ar@<-2pt>[l]\ar@<+2pt>[l] 2}.$$
	We write $\delta=(1,1)$ the minimal imaginary root of $K$ and we denote by $\mathcal A(K)$ the coefficient-free cluster algebra with initial seed $(K, \textbf u)$ with $\textbf u=(u_1,u_2)$.

	\begin{subsubsection}{Normalized Chebyshev polynomials}\label{subsection:CZChebyshev}
		In this subsection, we recall briefly some results of \cite{CZ} concerning the normalized Chebyschev polynomials.
		
		The \emph{normalized Chebyshev polynomials of the second kind}\index{Chebyshev polynomial!normalized of the second kind} are the polynomials defined inductively by :
		$$C_{-2}(x)=0, \, C_{-1}(x)=0, \, C_0(x)=1, \, C_{n+1}(x)=xC_n(x)-C_{n-1}(x)$$
		In particular, for every $n \geq 0$, $C_n$ is the unitary polynomial of degree $n$ characterized by the identity 
		$$C_n(t+t^{-1})=\sum_{k=0}^n t^{n-2k}$$
	
		The \emph{normalized Chebyschev polynomials of the first kind}\index{Chebyshev polynomial!normalized of the first kind} are the polynomials defined inductively by :
		$$P_n(x) = C_n(x) - C_{n-2}(x)$$
		for every $n \geq 0$. In particular, $P_n$ is also a unitary polynomial of degree $n$ for any $n \geq 0$ satisfying 
		$$P_n(t+t^{-1})=t^n+t^{-n}$$
		
		We can express the $C_n$ in terms of $P_n$ with the relation
		$$C_n(x)=\sum_{k=1}^{\left[\frac n2\right]+1}P_{n-2k}.$$
		In particular $C_n$ is a positive linear combination of the $P_k$.
	\end{subsubsection}

	\begin{subsubsection}{The canonical basis for the Kronecker quiver}
		We recall briefly the results and notations of \cite{shermanz}. 
		\begin{defi}
			An element $y \in \mathcal A(K)$ is called positive if for every cluster $\textbf x=(x,x')$, the coefficients in the expansion of $y$ as a Laurent polynomial in $x$ and $x'$ are positive.
		\end{defi}
		
		\begin{theorem}[\cite{shermanz}]\label{theorem:canonicalbasis}
			There exists an unique $\Z$-basis $\mathcal B(K)$ of $\mathcal A(K)$ such that the semi-ring of positive elements in $\mathcal A(K)$ consists precisely of positive integer linear combinations of elements of $\mathcal B(K)$.
		\end{theorem}
		
		The set $\mathcal B(K)$ is called the \emph{canonical basis}\index{canonical basis} of the cluster algebra $\mathcal A(K)$. The explicit value of the canonical basis is explicitly computed in \cite{shermanz} and is given by 
		$$\mathcal B(K)=\ens{\textrm{cluster monomials}} \sqcup \ens{P_n(z) \ : \ n \geq 1}$$
		where 
		$$z=\frac{1+u_1^2+u_2^2}{u_1u_2}.$$
		It is proved in \cite{shermanz} that for any $n \geq 1$, we have
			$$\delta(P_n(z))=n \delta$$
		Thus the canonical basis is the disjoint union of the set of cluster variables and a family of variables whose denominator vector are the imaginary roots of $Q$.	
	\end{subsubsection}

	\begin{subsubsection}{The semi-canonical basis for the Kronecker quiver}
		We denote by $\mathcal B'(K)$ the semi-canonical basis of the cluster algebra associated to $K$. This is given by
		$$\mathcal B'(K)=\ens{X_\alpha \ : \ \alpha \in \Z^{Q_0}}$$
		
		The explicit description is given by proposition \ref{prop:explicitbase} but we give an independent and explicit construction of all the $X_{\textbf d}$ based on the known results about the canonical decomposition of elements in the root lattice of the Kronecker quiver.
		\begin{maprop}
			$$\mathcal B'(K)=\ens{\textrm{cluster monomials}} \sqcup \ens{z^n \ : \ n \geq 1}$$
		\end{maprop}
		\begin{proof}
			It follows from lemma \ref{lem:clustermonomial} that cluster monomials are elements of the canonical basis. Also, a direct computation proves that
			$$X_\delta=X_{M_\lambda}=z$$
			for any homogeneous simple $M_\lambda$.
			
			We also know that $X_{n\delta}=X_\delta^n=z^n$. Now we claim the above union is disjoint. In deed, consider an element $z_n$ and a cluster monomial $x$. Then there is some $M$ without self-extension such that $X_M=x$. Thus it follows that $\ddim M \not \in \Z_{\geq 0} \delta$ (there is no representation without self extension in $\rep(Q, n \delta)$ for every $n$). Thus $\delta(x) \neq \delta(z^n)$ for every $n$ and the above union is disjoint.
			
			Now it only remains to prove that every element $X_{\textbf d}$ is either a cluster monomial or a $z^n$ for some $n > 0$. First assume that ${\textbf d} \in \N^{Q_0}$. Then we know that $X_{\textbf d}=X_M$ for some $M$ in the open set of Kac's canonical decomposition $\mathfrak M_{\textbf d}$. We write ${\textbf d}=(d_1, d_2)$. Then it follows from \cite[subsection 3]{DW} that the canonical decomposition of ${\textbf d}$ depends on the quotient $a=d_1/d_2$. More precisely, if $a>1$, then the canonical decomposition of ${\textbf d}$ is of the form $\lambda p_m+\mu p_{m+1}$ where $p_{2k+1}=\ddim P_{2k+1}=\ddim \tau^{-k} P_1$ and $p_{2(k+1)}=\ddim P_{2(k+1)}=\ddim \tau^{-k} P_2$. Thus $$X_{{\textbf d}}=X_{P_m^{\lambda}\oplus P_{m+1}^{\mu}}$$
			but $\Ext^1_{kQ}(P_m,P_{m+1})=\Ext^1_{kQ}(P_{m+1}, P_m)=0$ the $X_{{\textbf d}}$ is a cluster monomial. Similarly, if $a <1$, $X_{{\textbf d}}$ will be a cluster monomial. Now if $a=1$, ${\textbf d} =n\delta$ for some $n$ and then we know that 
			$$X_{{\textbf d}}=X_{n\delta}=X_{\delta}^n=X_{M_\lambda}^n=z^n$$
			
			Now suppose that ${\textbf d} \in \Z^{Q_0}$ is such that $d_i<0$ for some $i$. If $d_1<0$ and $d_2<0$, then 
			$$X_{{\textbf d}}=X_{P_1[1]^{\oplus(-d_1)} \oplus P_2[1]^{\oplus(-d_2)}}$$
			is a cluster monomial because $P_1[1]^{\oplus(-d_1)} \oplus P_2[1]^{\oplus(-d_2)}$ has no self-extension. 
			
			If $d_1<0$ and $d_2 \geq 0$, then 
			$$X_{{\textbf d}}=X_{P_1[1]^{\oplus(-d_1)} \oplus I_2^{\oplus d_2}}$$
			is a cluster monomial because $P_1[1]^{\oplus(-d_1)} \oplus I_2^{\oplus d_2}$ has no self-extension. 
			
			Finally, if $d_1 \geq 0$ and $d_2<0$, then 
			$$X_{{\textbf d}}=X_{P_1^{\oplus d_1} \oplus P_2[1]^{\oplus(-d_2)}}$$
			is a cluster monomial because $P_1^{\oplus d_1} \oplus P_2[1]^{\oplus(-d_2)}$ has no self-extension. 
			
			Now it proves that
			$$\mathcal B'(K)=\ens{\textrm{cluster monomials}} \sqcup \ens{z^n \ : \ n \geq 1}$$
		\end{proof}
		
		Notice that $P_n(z) \neq z^n$ for every $n >1$. It follows that the semi-canonical basis and the canonical basis do not coincide. Thus, the semi-canonical do not have the positivity property of theorem \ref{theorem:canonicalbasis}. 
		
		Nevertheless, for any $n \geq 1$, the denominator vector $\delta(z^n)$ of $z^n$ is $n \delta(z)=n\delta$. Thus, as for the canonical basis, the semi-canonical basis is the disjoint union of the set of cluster variables and a family of variables whose denominator vector are the positive imaginary roots of $Q$.
	\end{subsubsection}

	\begin{subsubsection}{Caldero-Zelevinsky basis for the Kronecker quiver}
		In \cite{CZ}, the authors have computed another $\Z$-basis for the Kronecker quiver, this basis is given by
		$$\mathcal B''(K)=\ens{\textrm{cluster monomials}} \sqcup \ens{X_{M_\lambda^{(n)}} \ : \ n \geq 1}$$
		
		In corollary \ref{lem:Chebyshev}, we proved that $X_{M_\lambda^{(n)}}=C_n(z)$ where the $C_n$ are the normalized Chebyschev polynomial of the second kind defined in subsection \ref{subsection:CZChebyshev}.
	
		As it was noticed before, the denominator theorem of \cite{CK2} implies that $$\delta(C_n(z))=\delta(X_{M_\lambda^{(n)}})=n\delta.$$
		Thus the basis $\mathcal B''(K)$ is also the disjoint union of the set of cluster monomials and of a sets of variables whose denominator vectors correspond to the positive imaginary roots of $Q$.
	\end{subsubsection}

	\begin{subsubsection}{Base change between $\mathcal B(K)$ and $\mathcal B'(K)$}
	 	\begin{defi}
	 		Let $\mathfrak a=\ens{a_n, n \geq 0}$ and $\mathfrak b=\ens{b_n, n \geq 0}$ be two bases of the $\Z$-module $\mathcal A(K)$. We say that there is a \emph{locally unipotent base change}\index{locally unipotent base change} from $A$ to $B$ if for every $n \in \Z$, the $\Z$-modules spanned by $\ens{a_k, 0 \leq k \leq n}$ and $\ens{b_k, 0 \leq k \leq n}$ coincide and if the base change matrix $P$ from $(a_k, 0 \leq k \leq n)$ to $(b_k, 0 \leq k \leq n)$ is unipotent in $M_n(\Z)$. If moreover $P$ has positive entries, then the base change is called \emph{positive}\index{positive base change}.
	 	\end{defi}
	 	
		\begin{maprop}\label{prop:BtoB'}
			There is a positive locally unipotent base change from $\mathcal B(K)$ to $\mathcal B'(K)$.
		\end{maprop}
		\begin{proof}
			As 
			$$\mathcal B(K)=\ens{\textrm{cluster monomials}} \sqcup \ens{P_n(z), n \in \N}$$
			and
			$$\mathcal B'(K)=\ens{\textrm{cluster monomials}} \sqcup \ens{z^n, n \in \N},$$
			it suffices to prove that there is a positive locally unipotent base change from $\ens{P_n(z), n \in \N}$ to $\ens{z^n, n \in \N}$. It is equivalent to prove that every $z^n$ can be written as a positive $\Z$-linear combination of the $P_k(z)$ for $0 \leq k \leq n$.
			
			As $\mathcal B(K)$ is a $\Z$-basis of $\mathcal A(K)$ and $z^n \in \mathcal A(K)$ for every $n$, it follows that each $z^n$ can be written as a $\Z$-linear combination of $P_k(z)$ for $k \in \N$. Each $P_k(z)$ being a unitary polynomial of degree $k$, it follows that $z^n$ can be written as a $\Z$-linear combination of the $P_k(z)$ for $0 \leq k \leq n$. Thus, there is a locally unipotent base change from $\mathcal B(K)$ to $\mathcal B'(K)$.
			
			According to \cite{shermanz}, $z=X_{M_\lambda}$ is a positive element in $\mathcal A(K)$. Thus, as positive elements form a semiring in $\mathcal A(K)$, each $X_{M_\lambda}^n$ is a positive element in $\mathcal A(K)$ and can thus be written as a positive $\Z$-linear combination of elements of $\mathcal B(K)$. The base change is then positive and the proposition is proved.
		\end{proof}
		
		\begin{monexmp}
			If we look at the base change matrix from the family $(z^n, 0 \leq n \leq 10) \subset \mathcal B'(K)$ to the family $(P_n(z), 0 \leq n \leq 10)$ of the canonical basis for $0 \leq n \leq 10$, we obtain
			$$\left[\begin{array}{ccccccccccc}
				1 & 0 & 2 & 0 & 6 & 0 & 20 & 0 & 70 & 0 & 252 \\
				0 & 1 & 0 & 3 & 0 & 10 & 0 & 35 & 0 & 126 & 0 \\
				0 & 0 & 1 & 0 & 4 & 0 & 15 & 0 & 56 & 0 & 210 \\
				0 & 0 & 0 & 1 & 0 & 5 & 0 & 21 & 0 & 84 & 0 \\
				0 & 0 & 0 & 0 & 1 & 0 & 6 & 0 & 28 & 0 & 120 \\
				0 & 0 & 0 & 0 & 0 & 1 & 0 & 7 & 0 & 36 & 0 \\
				0 & 0 & 0 & 0 & 0 & 0 & 1 & 0 & 8 & 0 & 45 \\
				0 & 0 & 0 & 0 & 0 & 0 & 0 & 1 & 0 & 9 & 0 \\
				0 & 0 & 0 & 0 & 0 & 0 & 0 & 0 & 1 & 0 & 10 \\
				0 & 0 & 0 & 0 & 0 & 0 & 0 & 0 & 0 & 1 & 0 \\
				0 & 0 & 0 & 0 & 0 & 0 & 0 & 0 & 0 & 0 & 1
			\end{array}\right] $$
			which is positive and unipotent. The inverse matrix is
			$$\left[\begin{array}{rrrrrrrrrrr}
				1 & 0 & -2 & 0 & 2 & 0 & -2 & 0 & 2 & 0 & -2 \\
				0 & 1 & 0 & -3 & 0 & 5 & 0 & -7 & 0 & 9 & 0 \\
				0 & 0 & 1 & 0 & -4 & 0 & 9 & 0 & -16 & 0 & 25 \\
				0 & 0 & 0 & 1 & 0 & -5 & 0 & 14 & 0 & -30 & 0 \\
				0 & 0 & 0 & 0 & 1 & 0 & -6 & 0 & 20 & 0 & -50 \\
				0 & 0 & 0 & 0 & 0 & 1 & 0 & -7 & 0 & 27 & 0 \\
				0 & 0 & 0 & 0 & 0 & 0 & 1 & 0 & -8 & 0 & 35 \\
				0 & 0 & 0 & 0 & 0 & 0 & 0 & 1 & 0 & -9 & 0 \\
				0 & 0 & 0 & 0 & 0 & 0 & 0 & 0 & 1 & 0 & -10 \\
				0 & 0 & 0 & 0 & 0 & 0 & 0 & 0 & 0 & 1 & 0 \\
				0 & 0 & 0 & 0 & 0 & 0 & 0 & 0 & 0 & 0 & 1
			\end{array}\right]$$ 
		\end{monexmp}
	\end{subsubsection}
	
	\begin{subsubsection}{Base change between $\mathcal B'(K)$ and $\mathcal B''(K)$}
		
			For any $n \geq 0$, we write $P_n=P_n(z)$ and
			$$z^n=\sum_{i \leq n} \lambda_{i,n} P_i$$
			the expansion of $z^n$ in the $P_n$. It follows from proposition \ref{prop:BtoB'} that that each $\lambda_{i,n}$ is positive.

		\begin{monlem}\label{lemcoeff}
			For any $n \geq 1$, we have :
			\begin{enumerate}
				\item $\lambda_{i,n}=0$ if $i \not \equiv n [2]$,
				\item $\lambda_{i,n} < \lambda_{i-2,n}$ for any $i \geq 2$ such that $i \equiv n[2]$.
			\end{enumerate} 
		\end{monlem}
		\begin{proof}
			We prove it by induction on $n$. If $n=1$, then $z^n=P_1$ and thus $\lambda_{1,1}=1$ and $\lambda_{i,1}=0$ for all $i \neq 1$, the above assertions are true. We now prove the induction step. 
			$$z^{n+1}=z.z^n=P_1.\left(\sum_{i \leq n} \lambda_{i,n} P_i \right)$$
			Now according to \cite[prop. 5.4 (1)]{shermanz}, we have
			$$P_1P_i=\left\{ \begin{array}{rl}
				P_{i-1}+P_{i+1} & \textrm{ if } n >1 \\
				2 + P_2 & \textrm{ if } i =1\\
				P_1 & \textrm{ if } i=0
			\end{array}\right.$$
			it follows that
			\begin{align*}
				z^{n+1}
					&= \lambda_{0,n}P_1P_0 + \lambda_{1,n} P_1P_1+\sum_{2 \leq i \leq n}\lambda_{i,n}P_1P_i \\
					&= \lambda_{0,n}P_1 + \lambda_{1,n} (2+P_2)+\sum_{2 \leq i \leq n}\lambda_{i,n}(P_{i-1}+P_{i+1}) \\
					&= 2\lambda_{1,n}P_0 + (\lambda_{0,n}+\lambda_{2,n})P_1 + \sum_{i \geq 2}(\lambda_{i-1,n}+\lambda_{i+1,n}) P_i
			\end{align*}
			A direct check proves that the induction step is verified.
		\end{proof}

		\begin{maprop}\label{prop:B''toB'}
			There is a positive locally unipotent base change from $\mathcal B''(K)$ to $\mathcal B'(K)$.
		\end{maprop}
		\begin{proof}
			As 
			$$\mathcal B'(K)=\ens{\textrm{cluster monomials}}\sqcup\ens{z^n, n \geq 0}$$
			and
			$$\mathcal B''(K)=\ens{\textrm{cluster monomials}}\sqcup\ens{C_n(z), n \geq 0},$$
			it suffices to prove that for any $n \geq 0$, the coefficients of the expansion of $z^n$ as a linear combination of the $C_n(z)$ is positive.
			
			We denote by $C_n=C_n(z)$. Then we recall that $P_n=C_n-C_{n-2}$. We write
			$$z^n=\sum_{i \leq n} \lambda_{i,n} P_n$$
			the expansion of $z^n$ as a linear combination of the $P_n$. Then
			\begin{align*}
				z^n
					&=\sum_{i} \lambda_{i,n} (C_n-C_{n-2}) \\
					&=\sum_{i} (\lambda_{i,n}-\lambda_{i+2,n}) C_n\\
			\end{align*}
			but according to lemma \ref{lemcoeff}, the difference $(\lambda_{i,n}-\lambda_{i+2,n})$ is positive and the $z^n$ can be written as a positive linear combination of the $C_n$.
		\end{proof}
		
		\begin{monexmp}
			If we look at the base change matrix from $(C_n(z), 0 \leq n \leq 10) \subset \mathcal B''(K)$ to $(z^n, 0 \leq n \leq 10) \subset \mathcal B'(K)$ is 
			$$\left[\begin{array}{ccccccccccc}
				1 & 0 & 1 & 0 & 2 & 0 & 5 & 0 & 14 & 0 & 42 \\
				0 & 1 & 0 & 2 & 0 & 5 & 0 & 14 & 0 & 42 & 0 \\
				0 & 0 & 1 & 0 & 3 & 0 & 9 & 0 & 28 & 0 & 90 \\
				0 & 0 & 0 & 1 & 0 & 4 & 0 & 14 & 0 & 48 & 0 \\
				0 & 0 & 0 & 0 & 1 & 0 & 5 & 0 & 20 & 0 & 75 \\
				0 & 0 & 0 & 0 & 0 & 1 & 0 & 6 & 0 & 27 & 0 \\
				0 & 0 & 0 & 0 & 0 & 0 & 1 & 0 & 7 & 0 & 35 \\
				0 & 0 & 0 & 0 & 0 & 0 & 0 & 1 & 0 & 8 & 0 \\
				0 & 0 & 0 & 0 & 0 & 0 & 0 & 0 & 1 & 0 & 9 \\
				0 & 0 & 0 & 0 & 0 & 0 & 0 & 0 & 0 & 1 & 0 \\
				0 & 0 & 0 & 0 & 0 & 0 & 0 & 0 & 0 & 0 & 1
			\end{array}\right]$$
			which is positive and unipotent. The inverse matrix is
			$$\left[\begin{array}{rrrrrrrrrrr}
				1 & 0 & -1 & 0 & 1 & 0 & -1 & 0 & 1 & 0 & -1 \\
				0 & 1 & 0 & -2 & 0 & 3 & 0 & -4 & 0 & 5 & 0 \\
				0 & 0 & 1 & 0 & -3 & 0 & 6 & 0 & -10 & 0 & 15 \\
				0 & 0 & 0 & 1 & 0 & -4 & 0 & 10 & 0 & -20 & 0 \\
				0 & 0 & 0 & 0 & 1 & 0 & -5 & 0 & 15 & 0 & -35 \\
				0 & 0 & 0 & 0 & 0 & 1 & 0 & -6 & 0 & 21 & 0 \\
				0 & 0 & 0 & 0 & 0 & 0 & 1 & 0 & -7 & 0 & 28 \\
				0 & 0 & 0 & 0 & 0 & 0 & 0 & 1 & 0 & -8 & 0 \\
				0 & 0 & 0 & 0 & 0 & 0 & 0 & 0 & 1 & 0 & -9 \\
				0 & 0 & 0 & 0 & 0 & 0 & 0 & 0 & 0 & 1 & 0 \\
				0 & 0 & 0 & 0 & 0 & 0 & 0 & 0 & 0 & 0 & 1
			\end{array}\right]$$
		\end{monexmp}
	\end{subsubsection}	
\end{subsection}

\begin{subsection}{Semicanonical basis for cluster algebras of type $\affA 21$}\label{subsection:A21}
		We are now interested in the particular case of quivers of affine type $\affA 21$, which is the only example of simply laced affine cluster algebra of rank 3. Such a quiver is necessarily isomorphic to 
		$$\xymatrix{
			&&2 \ar[rd]\\
			Q:& 1 \ar[ru] \ar[rr] &&3
		}.$$
		
		For any $\lambda \in k$, we set $M_\lambda$ to be the representation given by
		$$\xymatrix{
			&&k \ar[rd]^\lambda\\
			M_\lambda:& k \ar[ru]^1 \ar[rr]^1 &&k
		}$$
		and we set 
		$$\xymatrix{
			&&k \ar[rd]^1\\
			M_\infty:& k \ar[ru]^1 \ar[rr]^0 &&k
		}.$$
		
		We identify $\P^1$ and $k \sqcup \ens \infty$. For any $0 \neq \lambda \in \P^1 \setminus \ens{\infty}$, $M_\lambda$ is a quasi-module in an homogeneous tube and $\P^1_0=\P^1 \setminus \ens 0$. The only exceptional tube is $\mathcal T_0$ of rank 2 whose quasi-simples are 
		$$\xymatrix{
			&&k \ar[rd]^0\\
			E_0:& 0 \ar[ru] \ar[rr] &&0
		}$$
		and
		$$\xymatrix{
			&&0 \ar[rd]\\
			E_1:& k \ar[ru]^0 \ar[rr]^1 &&k
		}.$$
		
		We denote by 
		$$x_0=X_{E_0}=\frac{u_1+u_3}{u_2}$$
		$$x_1=X_{E_1}=\frac{1+u_1u_2+u_2u_3}{u_1u_3}$$
		$$z=X_{M_\lambda}=\frac{u_1^2u_2+u_1+u_3+u_2u_3^2}{u_1u_2u_3u_4}$$
		
		Note that for $i \neq j$, we have 
		$$\Ext^1_{\mathcal C_Q}(E_i,E_j) \neq 0,$$
		thus, the only rigid modules in the additive closure of $\ens{E_i\ : \ i =0,1}$ are $E_0^{\oplus n}$ and $E_1^{\oplus n}$ for $n \geq 1$.
		
		According to proposition \ref{prop:explicitbase}, we have then 
		$$\mathcal B'(Q)=\ens{\textrm{cluster monomials}} \sqcup \ens{z^nx_i^r \ : \ n >0, r \geq 0, i=0,1}.$$
\end{subsection}

\begin{subsection}{Semicanonical basis for cluster algebras of type $\affA 31$}\label{subsection:A31}
		We are now interested in a slightly more complicated example that we already met in example \ref{exmp:differencedeltaA31}. Consider a quiver $Q$ of type $\affA 31$, it is necessarily isomorphic to the quiver 
		$$
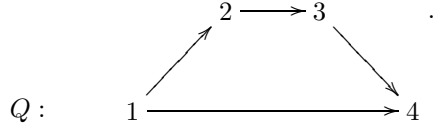
\xymatrix{
			&&2 \ar[r] & 3\ar[rd]\\
			Q:& 1 \ar[ru] \ar[rrr] &&&4
		}.$$
		
		We keep the notations of example \ref{exmp:differencedeltaA31}. We compute
		$$x_0=X_{E_0}=\frac{u_2+u_4}{u_3}$$
		$$x_1=X_{E_1}=\frac{u_1+u_3}{u_2}$$
		$$x_2=X_{E_2}=\frac{1+u_1u_3+u_2u_4}{u_1u_4}$$
		$$y_0=X_{E_0^{(2)}}=\frac{u_2u_1+u_1u_4+u_4u_3}{u_2u_3}$$
		$$y_1=X_{E_1^{(2)}}=\frac{u_3u_1^2+u_1+u_3^2u_1+u_3+u_3u_2u_4}{u_1u_2u_4}$$
		$$y_2=X_{E_2^{(2)}}=\frac{u_2u_3u_1+u_2+u_2^2u_4+u_4+u_2u_4^2}{u_1u_3u_4}$$
		$$z=X_{M_\lambda}=\frac{u_3u_2u_1^2+u_2u_1+u_1u_4+u_4u_3+u_3u_2u_4^2}{u_1u_2u_3u_4}$$
		
		Note that for $i \neq j$, we have 
		$$\Ext^1_{\mathcal C_Q}(E_i,E_j) \neq 0,$$
		$$\Ext^1_{\mathcal C_Q}(E_i^{(2)},E_j^{(2)}) \neq 0,$$
		$$\Ext^1_{\mathcal C_Q}(E_i,E_j^{(2)}) \neq 0,$$
		Finally, for any $i=0,1,2$, we have $\Ext^1_{\mathcal C_Q}(E_i,E_i^{(2)})=0$. Thus, the only rigid modules in the additive closure of $\ens{E_i, E_i^{(2)} \ : \ i =0,1,2}$ are the $E_i^{\oplus r}\oplus (E_i^{(2)})^{\oplus s}$ for $r,s \geq 0$ and $i=0,1,2$.
		
		According to proposition \ref{prop:explicitbase}, we have then 
		$$\mathcal B'(Q)=\ens{\textrm{cluster monomials}} \sqcup \ens{z^nx_i^ry_i^s \ : \ n >0, r,s \geq 0, i= 0,1,2}.$$
\end{subsection}

\begin{subsection}{Semicanonical basis for cluster algebras of type $\affA 22$}\label{subsection:A22}
		The only other example of (simply laced) rank 4 affine cluster algebra is the affine type $\affA 22$. Up to one mutation, we can assume that the considered quiver is
		$$
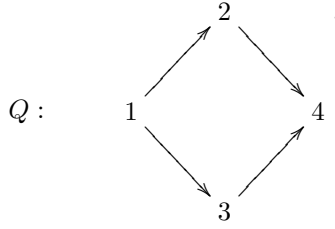
\xymatrix{
			&&2 \ar[rd]\\
			Q:& 1 \ar[ru] \ar[rd] &&4\\
			&&3 \ar[ru]\\
		}.$$
		
		The AR-quiver of $Q$ contains two exceptional tubes of rank 2 $\mathcal T_0$ and $T_\infty$. The quasi-simples of $\mathcal T_0$ are denoted by $E_0$ and $E_1$ and the quasi-simples of $\mathcal T_\infty$ are denoted by $F_0$ and $F_1$. 
		
		We compute
		$$x_0=X_{E_0}=\frac{u_1+u_4}{u_2}$$
		$$x_1=X_{E_1}=\frac{u_1+u_4+u_1u_2u_3+u_2u_3u_4}{u_1u_3u_4}$$
		$$y_0=X_{F_0}=\frac{u_1+u_4}{u_3}$$
		$$y_1=X_{F_2}=\frac{u_1+u_4+u_1u_2u_3+u_2u_3u_4}{u_1u_2u_4}$$
		$$z=X_{M_\lambda}=\frac{u_1^2+u_4^2+2u_1u_4+u_1^2u_2u_3+u_2u_3u_4^2}{u_1u_2u_3u_4}$$
		
		Note that, we have 
		$$\Ext^1_{\mathcal C_Q}(E_0,E_1) \simeq k^2 \simeq \Ext^1_{\mathcal C_Q}(F_0,F_1),$$
		$$\Ext^1_{\mathcal C_Q}(E_i,F_j)=0 \textrm{ for }i \neq j.$$
		Thus, the only rigid modules in the additive closure of $\ens{E_0,E_1,F_0,F_1}$ are the $E_i^{\oplus r}\oplus F_j^{\oplus s}$ for $r,s \geq 0$ and $i,j=0,1$.
		
		According to proposition \ref{prop:explicitbase}, we have then 
		$$\mathcal B'(Q)=\ens{\textrm{cluster monomials}} \sqcup \ens{z^nx_i^ry_j^s \ : \ n >0, r,s \geq 0, i,j= 0,1}.$$
\end{subsection}

\section{Conjectures and questions}\label{section:conjectures}
	\begin{subsection}{Semicanonical basis for cluster algebras of affine types $\tilde {\mathbb E},\tilde {\mathbb D}$}
	We conjecture that corollary \ref{theorem:semicanonicalbasisA} holds for any quiver of affine type.
	\begin{maconj}
		Let $Q$ be an affine quiver. Then the set $\mathcal B'(Q)$ of all generic variables is a $\Z$-basis in $\mathcal A(Q)$.
	\end{maconj}
	In order to prove this conjecture, it follows from theorem \ref{theorem:semicanonicalbasis} that one only has to prove that a quiver of affine type $\Daffine$ or $\Eaffine$ satisfies the difference property.

	It seems unlikely to carry out proposition \ref{prop:grassmannians} to the other affine types. The possible multiplicities in the grassmannians do not leave any hope to adapt directly the proof given in this article. Using the notations of proposition \ref{prop:grassmannians}, the first conjecture is:
	
	\begin{maconj}\label{conj:grassmannians}
		Let $Q$ be an affine quiver of affine type, $\lambda\in \P^1_0(Q)$ and $E$ be a quasi-simple module in an exceptional tube. Then, for any dimension vector $\textbf v$, the following equality holds:
		$$\chi(\Gr_{\textbf v}(M_E)) =\chi(\Gr_{\textbf v}(M_\lambda)) +\chi(\Gr_{\textbf v}^E(\regrad M_E))$$
	\end{maconj}
	
	An unsatisfying method, but at least fruitful in simple examples, consists in doing explicit computations of cluster variables. In order to have a finite number of computations, one can chose a given orientation and try to deal later with reflection functors in order to obtain the result for any orientations. Nevertheless, it seems rather hopeless to do the computations for affine types $\tilde {\mathbb{E}}$.
	
	As proposition \ref{prop:grassmannians} is only used in order to prove theorem \ref{theorem:differencedelta}, one can also try to prove directly theorem \ref{theorem:differencedelta}. The following conjecture would be a direct consequence of conjecture \ref{conj:grassmannians} but it might be simpler to prove it directly.
	
	\begin{maconj}\label{conj:differencedelta}
		Every quiver of affine type satisfies the difference property.
	\end{maconj}
	
	The following lemma can turn out to be very useful, it proves that the difference $X_{M_E}-X_{\regrad M_E/E}$ is invariant under translation. Namely:
	\begin{monlem}\label{lem:differenceconstante}
		Let $Q$ be a quiver of affine type, let $E$ be a quasi-simple module in an exceptional tube $\mathcal T$. Then 
		$$X_{M_E}-X_{\regrad M_E/E}=X_{M_F}-X_{\regrad M_F/F}$$
		for any quasi-simple module $F$ in $\mathcal T$.
	\end{monlem}
	\begin{proof}
		Fix $Q$ an affine quiver, $E$ a quasi-simple in an exceptional tube $\mathcal T$ of rank $p>1$. Denote by $E_i(Q), i \in \Z/p\Z$ the quasi-simple modules in $\mathcal T$ such that $\tau E_i(Q)=E_{i-1}(Q)$ for any $i \in \Z/p\Z$. We denote by $X^Q_?$ the Caldero-Chapoton map on the cluster category $\mathcal C_Q$.
			
		Consider the quiver $A$ of affine type $\affA p,p$ with an alternating orientation. Fix $\mathcal T(A)$ an exceptional tube of $kA$-mod. Then $\mathcal T(A)$ has rank $p$. Denote by $E_i(A), i \in \Z/p\Z$ the quasi-simple modules in $\mathcal T$ such that $\tau E_i(A)=E_{i-1}(A)$ for any $i \in \Z/p\Z$. We denote by $X^A_?$ the Caldero-Chapoton map on the cluster category $\mathcal C_A$. It is proved in \cite{Dupont:stabletubes} that the family $\ens{X^A_{E_i(A)}, i \in \Z/p\Z}$ is algebraically independent over $\Q$.
		
		We consider the surjective $\Z$-algebra homomorphism:
		$$\pi: \left\{\begin{array}{rcll}
			\Z[X^A_{E_i(A)}, i \in \Z/p\Z] & \fl & \Z[X^Q_{E_i(Q)}, i \in \Z/p\Z]\\
			X^A_{E_i(A)} & \mapsto & X^Q_{E_i(Q)} & \textrm{ for any } i \in \Z/p\Z
		\end{array}\right.$$

		If follows from \cite{Dupont:stabletubes} that
		$$\pi(X^A_{E_i(A)^{(k)}})=X^Q_{E_i(Q)^{(k)}}$$
		for any $i \in \Z/p\Z$ and any $k \geq 0$.
		
		In order to prove the lemma, it is equivalent to prove that 
		$$X^Q_{E_i(Q)^{(p)}}-X^Q_{E_{i-1}(Q)^{(p-1)}}$$
		does not depend on $i \in \Z/p\Z$.
		
		By theorem \ref{theorem:differencedelta}, 
		$$X^A_{E_i(A)^{(p)}}-X^A_{E_{i-1}(A)^{(p-1)}}=X^A_{M_\lambda}$$
		where $\lambda \in \P^1_0(Q)$. In particular, the difference does not depend on $i \in \Z/p\Z$. Thus
		$$X^Q_{E_i(Q)^{(p)}}-X^Q_{E_{i-1}(Q)^{(p-1)}}=\pi(X^A_{E_i(A)^{(p)}}-X^A_{E_{i-1}(A)^{(p-1)}})$$
		does not depend on $i \in \Z/p\Z$.
	\end{proof}

	This lemma finds its interest in the fact that for an affine quiver $Q$, Crawley-Boevey's construction by one point extensions allows to realize an indecomposable module of dimension $\delta$ in each tube. This way, in any exceptional tube, there is a ``privileged'' module $M_E$, explicitly realized, for which it seems to be reasonable to prove equality $(\ref{eq:differencedelta})$. Lemma \ref{lem:differenceconstante} can then prove that equality $(\ref{eq:differencedelta})$ holds for any indecomposable in the tube.
\end{subsection}
	
\begin{subsection}{Semicanonical basis for acyclic quivers}
	For quivers of wild type, it is not clear that our methods can generalize. Indeed, our work involve deeply the knowledge of the Auslander-Reiten combinatorics of the considered quiver and at this time this Auslander-Reiten combinatorics is not so clear for wild quivers. Nevertheless, the very strong analogy between our works and those of Geiss, Leclerc and Schro\"er (see \cite{GLS:rigid2}) about nilpotent varieties, there is a hope that the answer to the following question is positive:
	\begin{question}
		Let $Q$ be an acyclic quiver, is $\mathcal B'(Q) \cap \mathcal A(Q)$ a $\Z$-basis of the $\Z$-module $\mathcal A(Q)$?
	\end{question}
\end{subsection}

\subsection*{Acknowledgements}
	This work is part of my PhD thesis \cite{mathese}. I would like to thank my adviser Philippe Caldero for very helpful conversations and fruitful ideas about the subject. I would also like to thank Bernhard Keller and Robert Marsh for their careful reading through of this work and their various remarks and corrections. I am also grateful to Michel Brion for indicating the proof of lemma \ref{lem:Ude}. Finally, I would like to thank Idun Reiten and Claus Michael Ringel for interesting discussions and advices concerning the topic.

%

\begin{thebibliography}{BMR{\etalchar{+}}06}

\bibitem[ASS05]{ASS}
I.~Assem, D.~Simson, and A.~Skowro{\'n}ski.
\newblock {\em Elements of representation theory of Associative Algebras,
  Volume 1: Techniques of Representation Theory}, volume~65 of {\em London
  Mathematical Society Student Texts}.
\newblock Cambridge University Press, 2005.
\newblock MR2197389 (2006j:16020).

\bibitem[BFZ05]{cluster3}
A.~Berenstein, S.~Fomin, and A.~Zelevinsky.
\newblock Cluster algebras {III}: Upper bounds and double {B}ruhat cells.
\newblock {\em Duke Mathematical Journal}, 126(1):1--52, 2005.
\newblock MR2110627 (2005i:16065).

\bibitem[BMR{\etalchar{+}}06]{BMRRT}
A.~Buan, R.~Marsh, M.~Reineke, I.~Reiten, and G.~Todorov.
\newblock Tilting theory and cluster combinatorics.
\newblock {\em Adv. Math.}, 204(2):572--618, 2006.
\newblock MR2249625 (2007f:16033).

\bibitem[BMR07]{BMR1}
A.~Buan, R.~Marsh, and I.~Reiten.
\newblock Cluster-tilted algebras.
\newblock {\em Trans. Amer. Math. Soc.}, 359(1):323--332, 2007.
\newblock MR2247893 (2007f:16035).

\bibitem[BMR08]{BMR2}
A.~Buan, R.~Marsh, and I.~Reiten.
\newblock Cluster mutation via quiver representations.
\newblock {\em Commentarii Mathematici Helvetici}, 83(1):143--177, 2008.

\bibitem[BMR09]{BMR3}
A.~Buan, R.~Marsh, and I.~Reiten.
\newblock Denominators of cluster variables.
\newblock {\em J. Lond. Math. Soc. (2)}, 79(3):589--611, 2009.

\bibitem[CB92]{CB:lectures}
W.~Crawley-Boevey.
\newblock Lectures on representations of quivers.
\newblock 1992.

\bibitem[CBS02]{CBS}
W.~Crawley-Boevey and J.~Schröer.
\newblock Irreducible components of varieties of modules.
\newblock {\em J. Reine Angew. Math.}, 553:201--220, 2002.
\newblock MR1944812 (2004a:16020).

\bibitem[CC06]{CC}
P.~Caldero and F.~Chapoton.
\newblock Cluster algebras as {H}all algebras of quiver representations.
\newblock {\em Commentarii Mathematici Helvetici}, 81:596--616, 2006.
\newblock MR2250855 (2008b:16015).

\bibitem[CCS06a]{CCS2}
P.~Caldero, F.~Chapoton, and R.~Schiffler.
\newblock Quivers with relations and cluster tilted algebras.
\newblock {\em Algebras and Representation Theory}, 9:359--376, 2006.
\newblock MR2250652 (2007f:16036).

\bibitem[CCS06b]{CCS1}
P.~Caldero, F.~Chapoton, and R.~Schiffler.
\newblock Quivers with relations arising from clusters ({$A_n$} case).
\newblock {\em Transactions of the AMS}, 358:1347--1354, 2006.
\newblock MR2187656 (2007a:16025).

\bibitem[{Cer}08]{Cerulli:thesis}
G.~{Cerulli Irelli}.
\newblock {\em Structure theory for affine cluster algebras of rank 3}.
\newblock PhD thesis, University of Padova, 2008.

\bibitem[CK06]{CK2}
P.~Caldero and B.~Keller.
\newblock From triangulated categories to cluster algebras {II}.
\newblock {\em Annales Scientifiques de l'Ecole Normale Sup{\'e}rieure},
  39(4):83--100, 2006.
\newblock MR2316979 (2008m:16031).

\bibitem[CK08]{CK1}
P.~Caldero and B.~Keller.
\newblock From triangulated categories to cluster algebras.
\newblock {\em Inventiones Mathematicae}, 172:169--211, 2008.
\newblock MR2385670.

\bibitem[CR08]{CR}
P.~Caldero and M.~Reineke.
\newblock On the quiver grassmannian in the acyclic case.
\newblock {\em Journal of Pure and Applied Algebra}, 212(11):2369--2380, 2008.
\newblock MR2440252.

\bibitem[CZ06]{CZ}
P.~Caldero and A.~Zelevinsky.
\newblock Laurent expansions in cluster algebras via quiver representations.
\newblock {\em Moscow Mathematical Journal}, 6:411--429, 2006.
\newblock MR2274858 (2008j:16045).

\bibitem[DR76]{DR:memoirs}
V.~Dlab and C.M. Ringel.
\newblock Indecomposable representations of graphs and algebras.
\newblock {\em Memoirs of the AMS}, 173:1--57, 1976.
\newblock MR0447344.

\bibitem[Dup08]{mathese}
G.~Dupont.
\newblock {\em Alg{\`e}bres amass{\'e}es affines}.
\newblock PhD thesis, Universit{\'e} Claude Bernard Lyon 1,
  http://tel.archives-ouvertes.fr/tel-00338684/fr, november 2008.

\bibitem[Dup09]{Dupont:stabletubes}
G.~Dupont.
\newblock Cluster multiplication in regular components via generalized
  {C}hebyshev polynomials.
\newblock {\em arXiv:0801.3964v2 [math.RT]}, 2009.

\bibitem[DW02]{DW}
H.~Derksen and J.~Weyman.
\newblock On the canonical decomposition of quiver representations.
\newblock {\em Compositio Mathematica}, 133:245--265, 2002.
\newblock MR1930979 (2003h:16017).

\bibitem[FZ02]{cluster1}
S.~Fomin and A.~Zelevinsky.
\newblock Cluster algebras {I}: Foundations.
\newblock {\em J. Amer. Math. Soc.}, 15:497--529, 2002.
\newblock MR1887642 (2003f:16050).

\bibitem[FZ03]{cluster2}
S.~Fomin and A.~Zelevinsky.
\newblock Cluster algebras {II}: Finite type classification.
\newblock {\em Inventiones Mathematicae}, 154:63--121, 2003.
\newblock MR2004457 (2004m:17011).

\bibitem[FZ07]{cluster4}
S.~Fomin and A.~Zelevinsky.
\newblock Cluster algebras {IV}: Coefficients.
\newblock {\em Composition Mathematica}, 143(1):112--164, 2007.
\newblock MR2295199 (2008d:16049).

\bibitem[GLS05]{GLS}
C.~Geiss, B.~Leclerc, and J.~Schr{\"o}er.
\newblock Semicanonical bases and preprojective algebras.
\newblock {\em Ann. Sci. \'Ecole Norm. Sup. (4)}, 38(2):193--253, 2005.

\bibitem[GLS07]{GLS2}
C.~Geiss, B.~Leclerc, and J.~Schr{\"o}er.
\newblock Semicanonical bases and preprojective algebras {II}: A multiplication
  formula.
\newblock {\em Compos. Math.}, 143(5):1313--1334, 2007.
\newblock MR2360317 (2009b:17031).

\bibitem[GLS08]{GLS:rigid2}
C.~Geiss, B.~Leclerc, and J.~Schr{\"o}er.
\newblock Cluster algebra structures and semicanonical bases for unipotent
  groups.
\newblock {\em arXiv:math/0703039v3 [math.RT]}, 2008.

\bibitem[HU05]{HU}
D.~Happel and L.~Unger.
\newblock On the set of tilting objects in hereditary categories.
\newblock {\em Fields Institute Communications}, 45:141--159, 2005.
\newblock MR2146246 (2006h:18006).

\bibitem[Hub06]{Hubery:cluster}
A.~Hubery.
\newblock Acyclic cluster algebras via {R}ingel-{H}all algebras.
\newblock {\em preprint}, 2006.

\bibitem[Kac80]{Kac:infroot1}
V.G. Kac.
\newblock Infinite root systems, representations of graphs and invariant
  theory.
\newblock {\em Inventiones Mathematicae}, 56:57--92, 1980.
\newblock MR0557581 (82j:16050).

\bibitem[Kac82]{Kac:infroot2}
V.G. Kac.
\newblock Infinite root systems, representations of graphs and invariant theory
  {II}.
\newblock {\em Journal of algebra}, 78:163--180, 1982.
\newblock MR0677715 (85b:17003).

\bibitem[Kel05]{K}
B.~Keller.
\newblock On triangulated orbit categories.
\newblock {\em Documenta Mathematica}, 10:551--581, 2005.
\newblock MR2184464 (2007c:18006).

\bibitem[Kel08]{Keller:categorification}
B.~Keller.
\newblock Categorification of acyclic cluster algebras: an introduction.
\newblock {\em arXiv:0801.3103v1 [math.RT]}, 2008.

\bibitem[KR08]{KR}
B.~Keller and I.~Reiten.
\newblock Acyclic calabi-yau categories.
\newblock {\em Compositio Mathematicae}, 144(5):1332--1348., 2008.
\newblock MR2457529.

\bibitem[Lec03]{Leclerc:imaginary_vectors}
B.~Leclerc.
\newblock Imaginary vectors in the dual canonical basis of ${U_q(n)}$.
\newblock {\em Transform. Groups}, 8(1):95--104, 2003.
\newblock MR1959765 (2004d:17020).

\bibitem[MRZ03]{MRZ}
R.~Marsh, M.~Reineke, and A.~Zelevinsky.
\newblock Generalized associahedra via quiver representations.
\newblock {\em Trans. Amer. Math. Soc.}, 355(1):4171--4186, 2003.
\newblock MR1990581 (2004g:52014).

\bibitem[Pal08]{Palu}
Y.~Palu.
\newblock Cluster characters for 2-{C}alabi-{Y}au triangulated categories.
\newblock {\em Ann. Inst. Fourier (Grenoble)}, 58(6):2221--2248, 2008.

\bibitem[Rin84]{ringel:1099}
C.M. Ringel.
\newblock Tame algebras and integral quadratic forms.
\newblock {\em Lecture Notes in Mathematics}, 1099:1--376, 1984.
\newblock MR0774589 (87f:16027).

\bibitem[Sch92]{Schofield:generalrepresentations}
A.~Schofield.
\newblock General representations of quivers.
\newblock {\em Proc. London Math. Soc.}, 65(3):46--64, 1992.
\newblock MR1162487 (93d:16014).

\bibitem[SZ04]{shermanz}
P.~Sherman and A.~Zelevinsky.
\newblock Positivity and canonical bases in rank 2 cluster algebras of finite
  and affine types.
\newblock {\em Mosc. Math. J.}, 4:947--974, 2004.
\newblock MR2124174 (2006c:16052).

\bibitem[Ver76]{verdier}
J.L. Verdier.
\newblock Stratification de {W}hitney et th{\'e}or{\`e}me de {B}ertini-{S}ard.
\newblock {\em Inventiones Mathematicae}, 36:295--312, 1976.
\newblock MR0481096.

\bibitem[Xu10]{Xu}
F.~Xu.
\newblock On the cluster multiplication theorem for acyclic cluster algebras.
\newblock {\em Trans. Amer. Math. Soc.}, 362(2):753--776, 2010.

\bibitem[XX07]{XX}
J.~Xiao and F.~Xu.
\newblock Green's formula with ${{\mathbb {C}}^*}$-action and
  {C}aldero-{K}eller's formula for cluster algebras.
\newblock {\em arXiv:0707.1175v2 [math.QA]}, 2007.

\bibitem[Zhu06]{Zhu:equivalence}
B.~Zhu.
\newblock Equivalence between cluster categories.
\newblock {\em Journal of Algebra}, 304:832--850, 2006.
\newblock MR2264281 (2007h:18017).

\end{thebibliography}

\newcommand{\etalchar}[1]{$^{#1}$}

\end{document}